\documentclass{imaiai-custom}

\pdfoutput=1

\jno{}          %%% is for doi number
\received{} %%% for received date
\revised{}  %%% for revised date
\accepted{} %%% for accepted date

\usepackage{mystyle}

\usepackage{url}

\begin{document}
  \title{Model Selection with \\ Low Complexity Priors}

  \author{%
    {\sc Samuel Vaiter}\\[2pt]
    CNRS, CEREMADE, Universit\'{e} Paris-Dauphine\\
    Corresponding author: \email{vaiter@ceremade.dauphine.fr}\\[6pt]
    {\sc Mohammad Golbabaee}\\[2pt]
    CNRS, CEREMADE, Universit\'{e} Paris-Dauphine\\
    {\email{golbabaee@ceremade.dauphine.fr}}\\[6pt]
    {\sc Jalal Fadili}\\[2pt]
    GREYC, CNRS-ENSICAEN-Universit\'{e} de Caen\\
    {\email{Jalal.Fadili@greyc.ensicaen.fr}}\\[6pt]
    {\sc and}\\[6pt]
    {\sc Gabriel Peyr\'{e}} \\[2pt]
    CNRS, CEREMADE, Universit\'{e} Paris-Dauphine\\
    {\email{peyre@ceremade.dauphine.fr}}}
 
  \maketitle

  % !TEX root = ../IMAIAI-PartlySmoothLinear.tex

\begin{abstract}
{Regularization plays a pivotal role when facing the challenge of solving ill-posed inverse problems, where the number of observations is smaller than the ambient dimension of the object to be estimated. A line of recent work has studied regularization models with various types of low-dimensional structures. In such settings, the general approach is to solve a regularized optimization problem, which combines a data fidelity term and some regularization penalty that promotes the assumed low-dimensional/simple structure. This paper provides a general framework to capture this low-dimensional structure through what we coin partly smooth functions relative to a linear manifold. These are convex, non-negative, closed and finite-valued functions that will promote objects living on low-dimensional subspaces. This class of regularizers encompasses many popular examples such as the $\lun$ norm, $\lun-\ldeux$ norm (group sparsity), as well as several others including the $\linf$ norm. We also show that the set of partly smooth functions relative to a linear manifold is closed under addition and pre-composition by a linear operator, which allows to cover mixed regularization, and the so-called analysis-type priors (e.g. total variation, fused Lasso, finite-valued polyhedral gauges). Our main result presents a unified sharp analysis of exact and robust recovery of the low-dimensional subspace model associated to the object to recover from partial measurements. This analysis is illustrated on a number of special and previously studied cases, and on an analysis of the performance of $\linf$ regularization in a compressed sensing scenario.}
{Convex regularization, Inverse problems, Model selection, Partial smoothness, Compressed Sensing, Sparsity, Total variation.}
\end{abstract}

%%% Local Variables: 
%%% mode: latex
%%% TeX-master: "../robust-convex-regularization"
%%% End: 

  % !TEX root = ../IMAIAI-PartlySmoothLinear.tex

\section{Introduction}
\label{sec:intro}

%%%%%%%%%%%%%%%%%%%%%%%%%%%%%%%%%%%%%%%%%%%%%%%%%%%%%%%%%%%%%%%%%%%%%%%%%%%%%%%%%%%%%%%%%%%%%
%%%%%%%%%%%%%%%%%%%%%%%%%%%%%%%%%%%%%%%%%%%%%%%%%%%%%%%%%%%%%%%%%%%%%%%%%%%%%%%%%%%%%%%%%%%%%
\subsection{Regularization of Linear Inverse Problems}
\label{sec:intro-linear}

Inverse problems are encountered in various areas throughout science and engineering. The goal is to provably recover the structure underlying an object $x_0 \in \RR^N$, either exactly or to a good approximation, from the partial measurements
\begin{equation}\label{eq:ip}
  y = \Phi x_0 + w ,
\end{equation}
where $y \in \RR^Q$ is the vector of observations, $w \in \RR^Q$ stands for the noise, and $\Phi \in \RR^{Q \times N}$ is a linear operator which maps the $N$-dimensional signal domain onto the $Q$-dimensional observation domain. The operator $\Phi$ is in general ill-conditioned or singular, so that solving for an accurate approximation of $x_0$ from~\eqref{eq:ip} is ill-posed.

The situation however changes if one imposes some prior knowledge on the underlying object $x_0$, which makes the search for solutions to~\eqref{eq:ip} feasible. This can be achieved via regularization which plays a fundamental role in bringing back ill-posed inverse problems to the land of well-posedness. We here consider solutions to the regularized optimization problem
\begin{equation}\label{eq:reg}\tag{$\Pp_\lambda(y)$}
  \xsol \in 
  \uArgmin{x \in \RR^N}
    \dfrac{1}{2} \norm{y - \Phi x}^2
  + \lambda J(x) ,
\end{equation}
where the first term translates the fidelity of the forward model to the observations, and $\J$ is the regularization term intended to promote solutions conforming to some notion of simplicity/low-dimensional structure, that is made precise later. The regularization parameter $\lambda > 0$ is adapted to balance between the allowed fraction of noise level and regularity as dictated by the prior on $x_0$. Before proceeding with the rest, it is worth mentioning that although we focus our analysis on the penalized form~\eqref{eq:reg}, our results can be extended with minor adaptations to the constrained formulation, i.e. the one where the data fidelity is put as a constraint. Note also that though we focus our attention on quadratic data fidelity for simplicity, our analysis carries over to more general fidelity terms of the form $F \circ \Phi$, for $F$ smooth and strongly convex. 

When there is no noise in the observations, i.e. $w = 0$ in~\eqref{eq:ip}, the equality-constrained minimization problem should be solved
\begin{equation}\label{eq:reg-noiseless}\tag{$\Pp_0(y)$}
  \xsol \in 
  \uArgmin{x \in \RR^N} J(x)
  \qsubjq
  \Phi x = y .
\end{equation}

In this paper, we consider the general case where the function $\J$ is convex, non-negative and finite-valued, hence everywhere continuous with a full domain.
This class of regularizers $\J$ include many well-studied ones in the literature. Among them, one can think of the $\lun$ norm used to enforce sparse solutions~\cite{tibshirani1996regre}, the discrete total variation semi-norm~\cite{rudin1992nonlinear}, the $\lun-\ldeux$ norm to induce block/group sparsity~\cite{yuan2005model}, or finite polyhedral gauges~\cite{vaiter13polyhedral}.

Assuming furthermore that $J$ enjoys a partial smoothness property (to be defined shortly) relative to a model subspace associated to $x_0$, our goal in this paper is to provide a unified analysis of exact and robust recovery guarantees of that subspace by solving~\eqref{eq:reg} from the partial measurements in~\eqref{eq:ip}. As a by-product, this will also entail a control on the $\ldeux$-recovery error.

%%%%%%%%%%%%%%%%%%%%%%%%%%%%%%%%%%%%%%%%%%%%%%%%%%%%%%%%%%%%%%%%%%%%%%%%%%%%%%%%%%%%%%%%%%%%%
%%%%%%%%%%%%%%%%%%%%%%%%%%%%%%%%%%%%%%%%%%%%%%%%%%%%%%%%%%%%%%%%%%%%%%%%%%%%%%%%%%%%%%%%%%%%%
\subsection{Contributions}
\label{sec:intro-contrib}

Our main contributions are as follows.

%%%%%%%%%%%%%%%%%%%%%%%%%%%%%%%%%%%%%%%%%%%%%%%%%%%%%%%%%%%%%%%%%%%%%%%%%%%%%%%%%%%%%%%%%%%%%
\subsubsection{Subdifferential Decomposability of Convex Functions}

Building upon Definition~\ref{defn:linmod}, which introduces the model subspace $\T_x$ at $x$, we provide an equivalent description of the subdifferential of a finite-valued convex function at $x$ in Theorem~\ref{thm:decomp}. Such a description isolates and highlights a key property of a regularizer, namely \emph{decomposability}. In turn, this property allows to rewrite the first-order minimality conditions of~\eqref{eq:reg} and~\eqref{eq:reg-noiseless} in a convenient and compact way, and this lays the foundations of our subsequent developments.

%%%%%%%%%%%%%%%%%%%%%%%%%%%%%%%%%%%%%%%%%%%%%%%%%%%%%%%%%%%%%%%%%%%%%%%%%%%%%%%%%%%%%%%%%%%%%
\subsubsection{Uniqueness}

In Theorem~\ref{thm:snsp}, we state a sharp sufficient condition, dubbed the Strong Null Space Property, to ensure that the solution of~\eqref{eq:reg} or~\eqref{eq:reg-noiseless} is unique. In Corollary~\ref{prop:uniqueness-topo}, we provide a weaker sufficient condition, stated in terms of a dual vector, the existence of which certifies uniqueness. Putting together Theorem~\ref{thm:decomp} and Corollary~\ref{prop:uniqueness-topo}, Theorem~\ref{thm:focu} states the sufficient uniqueness condition in terms of a specific dual certificate built from~\eqref{eq:reg} and~\eqref{eq:reg-noiseless}.

%%%%%%%%%%%%%%%%%%%%%%%%%%%%%%%%%%%%%%%%%%%%%%%%%%%%%%%%%%%%%%%%%%%%%%%%%%%%%%%%%%%%%%%%%%%%%
\subsubsection{Partly Smooth Functions Relative to a Linear Manifold}

In the quest for establishing robust recovery of the subspace model $\T_{x_0}$, we first need to quantify the stability of the subdifferential of the regularizer $\J$ to local perturbations of its argument. Thus, to handle such a change of geometry, we introduce the notion of \emph{partly smooth function relative to a linear manifold}.

We show in particular that two important operations preserve partial smoothness relative to a linear manifold. In Proposition~\ref{prop:sum-prg} and Proposition~\ref{prop:gauge-analysis-stable}, we show that partial smoothness relative to a linear manifold is preserved under addition and pre-composition by a linear operator. Consequently, more intricate regularizers can be built starting from simple functions, e.g. $\lun$-norm, which are known to be partly smooth relative to a linear manifold (see the review given in Section~\ref{sec:examples}).

%%%%%%%%%%%%%%%%%%%%%%%%%%%%%%%%%%%%%%%%%%%%%%%%%%%%%%%%%%%%%%%%%%%%%%%%%%%%%%%%%%%%%%%%%%%%%
\subsubsection{Exact and Robust Subspace Recovery}

This is the core contribution of the paper. Assuming the function is partly smooth relative to a linear manifold, we show in Theorem~\ref{thm:local-stab} that under a generalization of the irrepresentable condition~\cite{fuchs2004on-sp}, and with the proviso that the noise level is bounded and the minimal signal-to-noise ratio is high enough, there exists a whole range of the parameter $\lambda$ for which problem~\eqref{eq:reg} has a unique solution $\xsol$ living in the same subspace as $x_0$. In turn, this yields a control on $\ldeux$-recovery error within a factor of the noise level, i.e. $\norm{\xsol - x_0} = O(\norm{w})$. In the noiseless case, the irrepresentable condition implies that $x_0$ is exactly identified by solving~\eqref{eq:reg-noiseless}.

%%%%%%%%%%%%%%%%%%%%%%%%%%%%%%%%%%%%%%%%%%%%%%%%%%%%%%%%%%%%%%%%%%%%%%%%%%%%%%%%%%%%%%%%%%%%%
\subsubsection{Compressed Sensing Regularization with $\linf$ Norm}

To illustrate the usefulness of our findings, we apply this model identification result to the case of the $\linf$ norm in Section~\ref{sec:cs}. While there exists previous works on the stable $\ell^2$ recovery with $\linf$ regularization form random measurements, it is the first result to assess the recovery of the anti-sparse model associated to the data to recover, which is an important additional information. Our result shows that stability of the model operates at a different regime than $\ell^2$ recovery in term of scaling between the number of measurements and the anti-sparsity level. This somehow contrasts with classical results in sparse recovery where it is known that they hold at similar level (up to logarithmic terms), see Section~\ref{sec-pw-cs}.

%%%%%%%%%%%%%%%%%%%%%%%%%%%%%%%%%%%%%%%%%%%%%%%%%%%%%%%%%%%%%%%%%%%%%%%%%%%%%%%%%%%%%%%%%%%%%
%%%%%%%%%%%%%%%%%%%%%%%%%%%%%%%%%%%%%%%%%%%%%%%%%%%%%%%%%%%%%%%%%%%%%%%%%%%%%%%%%%%%%%%%%%%%%
\subsection{Related Work}
\label{sec:intro-pw}

%%%%
\subsubsection{Decomposability}

In~\cite{candes2011simple}, the authors introduced a notion of decomposable norms. In fact, we show that their regularizers are a subclass of ours that corresponds to strong decomposability in the sense of the Definition~\ref{defn:strong-gauge}, beside symmetry since norms are symmetric gauges. Moreover, their definition involves two conditions, the second of which turns out to be an intrinsic property implied by polarity rather than an assumption; see the discussion after Proposition~\ref{prop:strong-dec}. Typical examples of (strongly) decomposable norms are the $\lun$, $\lun-\ldeux$ and nuclear norms. However, strong decomposability excludes many important cases. One can think of analysis-type semi-norms since strong decomposability is not preserved under pre-composition by a linear operator, or the $\linf$ norm among many others. The analysis provided in~\cite{candes2011simple} deals only with identifiability in the noiseless case. Their work was extended in~\cite{oymak2012simultaneously} when $\J$ is the sum of decomposable norms.

%%%%
\subsubsection{Convergence rates}

In the inverse problems literature, a convergence (stability) rates have been derived in~\cite{burger2004convergence} with respect to the Bregman divergence for general convex regularizations $\J$. The author in~\cite{grasmair2011linear} established a stability result for general sublinear functions $\J$. The stability is however measured in terms of $\J$, and $\ell^2$-stability can only be obtained if $\J$ is coercive, which, again, excludes a large class of functions. In~\cite{fadili13stable}, a $\ell^2$-stability result for decomposable norms (in the sense of~\cite{candes2011simple}) precomposed by a linear operator is proved. However, none of these works deals with exact and robust recovery of the subspace model underlying $x_0$.

%%%%
\subsubsection{Model selection}

There is large body of previous works on the problem of the model selection properties (sometimes referred to as model consistency) of low-complexity regularizers. These previous works are targeting specific regularizers, most notably sparsity, group sparsity and low rank. We thus refer to Section~\ref{sec:examples} for a discussion of these relevant previous works. A distinctive feature of our analysis is that it is generic, so it covers all these special cases, and many more. Note however that is does not cover the nuclear norm, because its associated manifolds are not linear (they are indeed composed of algebraic manifolds of low rank matrices). We have recently proposed an extension of our results to this more general non-linear case in~\cite{2014-vaiter-ps-stability}. Note however that this new analysis uses a different proof technique, and is not able to provide explicit values for the constant involved in the robustness to noise.

% In this work, we prove stability of the linear manifold $\Mm$ relative to which the regularizer $J$ is partly smooth at the original $x_0$. All constants in the bound on the allowed noise level and the regularization paremeter regime are explicitly given. Those are expressed and interpreted precisely in terms of the parameters characterizing the regularity of $J$ along $\Mm$. It is worth point out that very recently, and while a first version of the manuscript was submitted, three of the co-authors of this paper extended the manifold stability result to any convex and finite-valued partly smooth convex function, should $\Mm$ be linear or not~\cite{2014-vaiter-ps-stability}. However, the proof technique is quite different. More importantly, unlike the present work, all the constants involved, including those in the bound on the noise level and the regularization parameter, were not provided.

%%%%
\subsubsection{Compressed sensing}
\label{sec-pw-cs}

Arguments based on the Gaussian width were used in~\cite{parrilo2010convex} to provide sharp estimates of the number of generic measurements required for exact and $\ell^2$-stable recovery of atomic set models from random partial information by solving a constrained form of~\eqref{eq:reg} regularized by an atomic norm. The atomic norm framework was then exploited in~\cite{rao2012signal} in the particular case of the group Lasso and union of subspace models. 
This was further generalized in \cite{amelunxen2013living} who developed for the noiseless case reliable predictions about the quantitative aspects of the phase transition in regularized linear inverse problems with random measurements. The location and width of the transition are controlled by the statistical dimension of a cone associated with the regularizer and the unknown. All this work is however restricted to a random (compressed sensing) scenario.

A notion of decomposability closely related to that of~\cite{candes2011simple}, but different, was first proposed in~\cite{negahban2010unified}. There, the authors study $\ell^2$-stability for this class of decomposable norms with a general sufficiently smooth data fidelity. This work however only handles norms, and their stability results require however stronger assumptions than ours (typically a restricted strong convexity which becomes a type of restricted eigenvalue property for linear regression with quadratic data fidelity).

%%%%%%%%%%%%%%%%%%%%%%%%%%%%%%%%%%%%%%%%%%%%%%%%%%%%%%%%
\subsection{Paper Organization}
\label{sec:intro-org}

The outline of the paper is the following.
Section~\ref{sec:decomposable} fully characterizes the canonical decomposition of the subdifferential of a convex function with respect to the subspace model at $x$. Sufficient conditions ensuring uniqueness of the minimizers to~\eqref{eq:reg} and~\eqref{eq:reg-noiseless} are provided in Section~\ref{sec:uniqueness}. In Section~\ref{sec:stable}, we introduce the notion of a partly smooth function relative to a linear manifold and show that this property is preserved under addition and pre-composition by a linear operator. Section~\ref{sec:small} is dedicated to our main result, namely theoretical guarantees for exact subspace recovery in presence of noise, and identifiability in the noiseless case. Section~\ref{sec:examples} exemplifies our results on several previously studied priors, and a detailed discussion on the relation with respect to relevant previous work is provided. Section~\ref{sec:cs} delivers a bound for the sampling complexity to guarantee exact recovery of the model subspace of antisparsity minimization from noisy Gaussian measurements. Some conclusions and possible perspectives of this work are drawn in Section~\ref{sec:conclusion}. The proofs of our results are collected in the appendix.

%%%%%%%%%%%%%%%%%%%%%%%%%%%%%%%%%%%%%%%%%%%%%%%%%%%%%%%%%%%%%%
\subsection{Notations and Elements from Convex Analysis}
\label{sec:intro-notations}

In the following, if $T$ is a vector space, $\proj_T$ denotes the orthogonal projector on $T$, and
\begin{equation*}
  x_T = \proj_T x \qandq \Phi_T = \Phi \proj_T.
\end{equation*}
For a subset $I$ of $\ens{1,\dots,N}$, we denote by $I^c$ its complement, $\abs{I}$ its cardinality. $x_{(I)}$ is the subvector whose entries are those of $x$ restricted to the indices in $I$, and $\Phi_{(I)}$ the submatrix whose columns are those of $\Phi$ indexed by $I$. For any matrix $A$, $A^*$ denotes its adjoint matrix and $A^+$ its Moore--Penrose pseudo-inverse. We denote the right-completion of the real line by $\overline{\RR} = \RR \cup \ens{+\infty}$.

A real-valued function  $f: \RR^N \to \overline{\RR}$ is coercive, if $\lim_{\norm{x} \to +\infty}f(x)=+\infty$. The effective domain of $f$ is defined by $\dom f = \enscond{x\in\RR^N}{f(x) < +\infty}$ and $f$ is proper if $\dom f \neq \emptyset$. We say that a real-valued function $f$ is lower semi-continuous (lsc) if $\liminf_{z \to x} f(z) \geq f(x)$. A function is said sublinear if it is convex and positively homogeneous.

Let the kernel of a function be denoted $\Ker f = \enscond{x \in \RR^N}{f(x) = 0}$. $\Ker f$ is a cone when $f$ is positively homogeneous.

We now provide some elements from convex analysis that are necessary throughout this paper. A comprehensive account can be found in~\cite{rockafellar1996convex,hiriart1996convex}.

%%%%
\subsubsection{Sets} 

For a non-empty set $C \subset \RR^N$, we denote $\co C$ the closure of its convex hull. Its \emph{affine hull} $\Aff C$ is the smallest affine manifold containing it, i.e.
\begin{equation*}
  \Aff C = \enscond{\sum_{i=1}^k \rho_i x_i}{k > 0, \rho_i \in \RR, x_i \in C, \sum_{i=1}^k \rho_i = 1} .
\end{equation*}
It is included in the \emph{linear hull} $\Lin C$ which is the smallest subspace containing $C$. 

The interior of $C$ is denoted $\interop C$. The \emph{relative interior} $\ri C$ of a convex set $C$ is the interior of $C$ for the topology relative to its affine full.

Let $C$ be a non-empty convex set. The set $C^\circ$ given by
\begin{equation*}
  C^\circ = \enscond{v \in \RR^N}{\dotp{v}{x} \leq 1 \text{ for all } x \in C}
\end{equation*}
is called the \emph{polar} of $C$.
$C^\circ$ is a closed convex set containing the origin. When the set $C$ is also closed and contains the origin, then it coincides with its bipolar, i.e. $C^{\circ\circ}=C$.

%%%%
\subsubsection{Functions} 

Let $C$ a nonempty convex subset of $\RR^N$. The \emph{indicator function} $\iota_{C}$ of $C$ is 
\begin{equation*}
\iota_{C} (x) =
  \begin{cases}
    0, & \text{if } x \in C ~ ,\\
    +\infty, & \text{otherwise}.
  \end{cases}
\end{equation*}

The Legendre-Fenchel \emph{conjugate} of a proper, lsc and convex function $f$ is
\begin{equation*}
f^*(u) = \sup_{x \in \dom f} \dotp{u}{x} - f(x) ~ ,
\end{equation*}
where $f^*$ is proper, lsc and convex, and $f^{**}=f$. For instance, the conjugate of the indicator function $\iota_C$ is the \emph{support function} of $C$
\[
	\sigma_{C}(u)=\sup_{x \in C} \dotp{u}{x} ~.
\]
$\sigma_C$ is sublinear, is non-negative if $0 \in C$, and is finite everywhere if, and only if, $C$ is a bounded set.

Let $f$ and $g$ be two functions proper closed convex functions from $\RR^N$ to $\overline{\RR}$. Their infimal convolution is the function
\[
    (f \infc g)(x) = \inf_{x_1+x_2 = x} f(x_1) + g(x_2) = \inf_{z\in\RR^N} f(z) + g(x-z)~.
\]

Let $C \subseteq \RR^N$ be a non-empty closed convex set containing the origin.
The \emph{gauge} of $C$ is the function $\gauge_C$ defined on $\RR^N$ by
\begin{equation*}
  \gauge_C(x) =
  \inf \enscond{\lambda > 0}{x \in \lambda C} .
\end{equation*}
As usual, $\gauge_C(x) = + \infty$ if the infimum is not attained. $\gauge_C$ is a non-negative, lsc and sublinear function. It is moreover finite everywhere, hence continuous, if, and only if, $C$ has the origin as an interior point, see Lemma~\ref{lem:convex-equivalence} for details.

The \emph{subdifferential} $\partial f(x)$ of a convex function $f$ at $x$ is the set
\[
	\partial f(x) = \enscond{u \in \RR^N}{f(x') \geq f(x) + \dotp{u}{x'-x}, \quad \forall x' \in \dom f} ~.
\]
An element of $\partial f(x)$ is a subgradient. If the convex function $f$ is G\^ateaux-differentiable at $x$, then its only subgradient is its gradient, i.e. $\partial f(x) = \ens{\nabla f(x)}$.

The \emph{directional derivative} $f'(x,\delta)$ of a lsc function $f$ at the point $x \in \dom f$ in the direction $\delta \in \RR^N$ is 
\begin{equation*}
  f'(x,\delta) = \lim_{t \downarrow 0} \frac{f(x + t\delta) - f(x)}{t} .
\end{equation*}
When $f$ is convex, then the function $\delta \mapsto f'(x,\cdot)$ exists and is sublinear. When $f$ has also full domain, then for any $x \in \RR^N$, $\partial f(x)$ is a non-empty compact convex set of $\RR^N$ whose support function is $f'(x,\cdot)$, i.e.
\begin{equation*}
  f'(x,\delta) = \sigma_{\partial f(x)}(\delta) = \sup_{\eta \in \partial f(x)} \dotp{\eta}{\delta} .
\end{equation*}

We also recall the fundamental first-order minimality condition of a convex function: $\xsol$ is the global minimizer of a convex function $f$ if, and only if, $0 \in \partial f(x)$.

%%%%
\subsubsection{Operator norm} 

Let $\J_1$ and $\J_2$ be two finite-valued gauges defined on two vector spaces $V_1$ and $V_2$, and $A : V_1 \to V_2$ a linear map.
The \emph{operator bound} $\normOP{A}{\J_1}{\J_2}$ of $A$ between $\J_1$ and $\J_2$ is given by
\begin{equation*}
  \normOP{A}{\J_1}{\J_2} = \sup_{\J_1(x) \leq 1} \J_2(Ax) .
\end{equation*}
Note that $\normOP{A}{\J_1}{\J_2} < + \infty$ if, and only if $A \Ker(\J_1) \subseteq \Ker(\J_2)$.
In particular, if $\J_1$ is coercive (i.e. $\Ker \J_1=\{0\}$ from Lemma~\ref{lem:convex-equivalence}(iv)), then $\normOP{A}{\J_1}{\J_2}$ is finite.
As a convention, $\normOP{A}{\J_1}{\norm{\cdot}_p}$ is denoted as $\normOP{A}{\J_1}{\ell^p}$.
An easy consequence of this definition is the fact that for every $x \in V_1$,
\begin{equation*}
  \J_2(Ax) \leq \normOP{A}{\J_1}{\J_2} \J_1(x) .
\end{equation*}

% Let $k \geq 1$.
% A $\Calt{k}$-\emph{manifold} $\Mm$ \emph{around} $x \in \RR^n$ \emph{of codimension} $m$ is a subset of $\RR^n$ such that there exists an open set $U$ of $\RR^n$ and a $\Calt{k}$-function $g : U \to \RR^{m}$ satisfying
% \begin{equation*}
%   \Mm \cap U = \enscond{\bar x \in U}{g(\bar x) = 0} ,
% \end{equation*}
% and $g$ has surjective derivative throughout $U$.
% We say that $\Mm$ is a $\Calt{k}$-\emph{manifold} if $\Mm$ is a $\Calt{k}$-manifold around every $x \in \Mm$ of codimension $m$.

% Let $\Mm$ be a $\Calt{k}$-manifold around $x \in \Mm$ of codimension $m$ associated to a $\Calt{k}$ function $g$.
% The \emph{tangent space} of $\Mm$ at $x$ is defined as
% \begin{equation*}
%   \tgtManif{\Mm}{x} = \Ker \mathrm{D} g(x) .
% \end{equation*}

%%% Local Variables: 
%%% mode: latex
%%% TeX-master: "../robust-convex-regularization"
%%% End: 

  % !TEX root = ../IMAIAI-PartlySmoothLinear.tex

\section{Model Subspace and Decomposability}
\label{sec:decomposable}

The purpose of this section is to introduce one of the main concepts used throughout this paper, namely the \textit{model subspace} associated to a convex function. The main result, Theorem~\ref{thm:decomp}, proves that the subdifferential of any convex function exhibits a decomposability property with respect to this subspace.

%%%%%%%%%%%%%%%%%%%%%%%%%%%%%%%%%%%%%%%%%%%%%%%%%%%%%%%%%%%%%%%%
\subsection{Model Subspace Associated to a Convex Function}
\label{sec:convex-geom}

Let $\J$ our regularizer, i.e. a finite-valued convex function.
\begin{defn}[Model Subspace]\label{defn:linmod}
  For any vector $x \in \RR^N$, denote $\bar \S_x$ the affine hull of the subdifferential of $\J$ at $x$
  \begin{equation*}
    \bar \S_x = \Aff \partial \J(x) ,
  \end{equation*}
  and $e_x$ the orthogonal projection of $0$ onto $\bar \S_x$
  \begin{equation*}
    e_x = \uargmin{e \in \bar \S_x} \norm{e}.
  \end{equation*}
  Let 
  \begin{equation*}
    \S_x = \bar \S_x - e_x = \Lin (\partial \J(x)) \qandq \T_x = \S_x^\bot .
  \end{equation*}
  $\T_x$ is coined the \emph{model subspace} of $x$ associated to $\J$.
\end{defn}
When $J$ is G\^ateaux-differentiable at $x$, i.e. $\partial J(x)=\ens{\nabla J(x)}$, $e_x=\nabla J(x)$ and $\T_x=\RR^N$. Note that the decomposition of $\RR^N$ as a sum of the two orthogonal subspaces $\T_x$ and $\S_x$ is also the core idea underlying the $\Uu-\Vv$-decomposition/theory developed in ~\cite{Lemarechal-ULagrangian}.
%On the contrary, when $\J$ is not smooth at $x$, the dimension of $\T_x$ is smaller dimension.\\
 
We start by summarizing some key properties of the objects $e_x$ and $\T_x$.
\begin{prop}\label{prop:convex-basics-decompos}
  For any $x \in \RR^N$, one has
  \begin{enumerate}[(i)]
  \item $e_x \in \T_x \cap \bar \S_x$.
  \item $\bar \S_x = \enscond{\eta \in \RR^N}{\eta_{\T_x} = e_x}$.
  \end{enumerate}
\end{prop}

In general $e_x \not\in \partial J(x)$, which is the situation displayed on~Figure~\ref{fig:geometry-gauge}.

\begin{figure}[htbp]
  \centering
  \includegraphics[bb=7 0 131 126,scale=1]{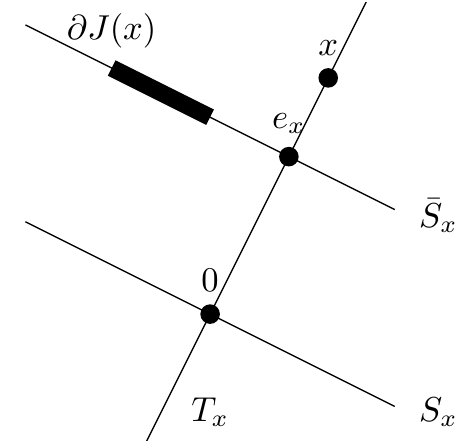}
  \caption{Illustration of the geometrical elements $(\S_x,\T_x,e_x)$.}
  \label{fig:geometry-gauge}
\end{figure}

From this section until Section~\ref{sec:stable}, we use the $\lun$-$\ldeux$ and the $\linf$ norms as illustrative examples. A more comprehensive treatment is provided in Section~\ref{sec:examples} which is completely devoted to examples.

\begin{exmp}[$\lun$-$\ldeux$ norm]
  We consider a uniform disjoint partition $\Bb$ of $\{1,\cdots,N\}$, 
  \[
	\{1,\ldots,N\} = \bigcup_{b \in \Bb} b, \quad b \cap b' = \emptyset, ~  \forall b \neq b' ~.
  \]
  The $\lun-\ldeux$ norm of $x$ is 
  \begin{equation*}
    \J(x) = \norm{x}_{\Bb} = \sum_{b \in \Bb} \norm{x_b}.
  \end{equation*}
  The subdifferential of $\J$ at $x \in \RR^N$ is
  \begin{equation*}
    \partial \J(x) =
    \enscond
    {\eta \in \RR^N}
    {
      \forall b \in I, \,
      \eta_b = \frac{x_b}{\norm{x_b}}
      \qandq
      \forall b \not\in I, \,
      \norm{\eta_b} \leq 1
    } ,
  \end{equation*}
  where $I = \enscond{b \in \Bb}{x_b \neq 0}$.
  Thus, the affine hull of $\partial J(x)$ reads
  \begin{equation*}
    \bar \S_x =
    \enscond
    {\eta \in \RR^N}
    {
      \forall b \in I, \,
      \eta_b = \frac{x_b}{\norm{x_b}}
    } .
  \end{equation*}
  Hence the projection of $0$ onto $\bar \S_x$ is
  \begin{equation*}
    e_x = (\Nn(x_b))_{b \in \Bb}
  \end{equation*}
  where $\Nn(a)=a/\norm{a}$ if $a\neq 0$, and $\Nn(0)=0$ and
  \begin{equation*}
    \S_x 
    = \bar \S_x - e_x
    = \enscond
    {\eta \in \RR^N}
    {
      \forall b \in I, \,
      \eta_b = 0
    } ,
  \end{equation*}
  and
  \begin{equation*}
    \T_x 
    = \S_x^\bot
    = \enscond
    {\eta \in \RR^N}
    {
      \forall b \not\in I, \,
      \eta_b = 0
    } .
  \end{equation*} 
\end{exmp}
\medskip
\begin{exmp}[$\linf$ norm]
  The $\linf$ norm is $\J(x) = \norm{x}_\infty = \umax{1 \leq i \leq N} \abs{x_i}$. 
  For $x=0$, $\partial \J(x)$ is the unit $\lun$ ball, hence $\bar \S_x = \S_x=\RR^N$, $\T_x=\{0\}$ and $e_x=0$. 
  For $x \neq 0$, we have
  \begin{align*}
  	  \partial \J(x) = \enscond{\eta}{ \foralls i \in I^c, \; \eta_{i} = 0, \; \dotp{\eta}{s} = 1, \; \eta_i s_i \geq 0 \; \foralls i \in I } ~.
  \end{align*}
  where $I=\enscond{i \in \{1,\ldots,N\}}{\abs{x_i}=\norm{x}_\infty}$, $s_i = \sign(x_i)$ if $i \in I$ with $\sign(0)=0$, and $s_i=0$ if $i \in I^c$. It is clear that $\bar{\S}_x$ is the affine hull of an $|I|$-dimensional face of the unit $\lun$ ball exposed by the sign subvector $s_{(I)}$. Thus $\e_x$ is the barycenter of that face, i.e. 
  \[
  \e_x = s/\abs{I} \qandq \S_x =\enscond{\eta}{\eta_{(I^c)} = 0 \qandq \dotp{\eta_{(I)}}{s_{(I)}} = 0} ~. 
  \]
  In turn
  \begin{equation*}
    \T_x = \S_x^\bot = \enscond{\alpha}{\alpha_{(I)} = \rho s_{(I)} \qforq \rho \in \RR} .
  \end{equation*} 
\end{exmp}

%%%%%%%%%%%%%%%%%%%%%%%%%%%%%%%%%%%%%%%%%%%%%%%%%%%%%%%%%%%%%%%%
\subsection{Decomposability Property}
\label{sec:decomposable-prop}

%%%%
\subsubsection{The subdifferential gauge and its polar}

Before providing an equivalent description of the subdifferential of $\J$ at $x$ in terms of the geometrical objects $e_x$, $\T_x$ and $\S_x$, we introduce a gauge that plays a prominent role in this description.
\begin{defn}[Subdifferential Gauge]\label{defn:subdifferential-gauge}
  Let $\J$ be a convex function.
  Let $x \in \RR^N \setminus \ens{0}$ and $\f_x \in \ri \partial \J(x)$.
  The \emph{subdifferential gauge} associated to $f_x$ is the gauge $\antigx = \gauge_{\partial \J(x) - f_x}$.
\end{defn}
Since $\partial \J(x) - f_x$ is a closed (in fact compact) convex set containing the origin, it is uniquely characterized by the subdifferential gauge $\antigx$ (see Lemma~\ref{lem:convex-equivalence}(i)).\\

The following proposition states the main properties of the gauge $\antigx$.
\begin{prop}\label{prop:anti-coer}
  The subdifferential gauge $\antigx$ is such that $\dom \antigx=\S_x$, and is coercive on $\S_x$.
\end{prop}

We now turn to the gauge polar to the subdifferential gauge $\antigPx = {(\antigx)}^\circ$, where the last equality is a consequence of Lemma~\ref{lem:convex-polar-gauge}(i). $\antigPx$ comes into play in several results in the rest of the paper.  
The following proposition summarizes its most important properties.
\begin{prop}\label{prop:antig-polar}
The gauge $\antigPx$ is such that
\begin{enumerate}[(i)] 
\item Its has a full domain.
\item $\antigPx(d)=\antigPx(d_\S)=\sup_{\antigx(\eta_{\S_x}) \leq 1} \dotp{\eta_{S_x}}{d}$.
\item $\Ker \antigPx=\T_x$ and $\antigPx$ is coercive on $\S_x$.
\end{enumerate}
\end{prop}

Let's derive the subdifferential gauge on the illustrative example of the $\linf$ norm. The case of $\lun-\ldeux$ norm is detailed in Section~\ref{sec:decomposable-special}.
\begin{exmp}[$\linf$ norm]
  Recall from Section~\ref{sec:convex-geom} that for $\J = \normi{\cdot}$, $f_x=e_x=s/\abs{I}$, with $s_{(I)}=\sign(x_{(I)})$, and $s_{(I^c)}=0$. Let $\Kk_x = \partial \J(x) - e_x$.
  It can be straightforwardly shown that in this case,
  \begin{align*}
    \Kk_x &= \enscond{v}{\foralls (i,j) \in I \times I^c, \, v_j = 0, \, \dotp{v_{(I)}}{s_{(I)}} = 0, \, - \abs{I}v_is_i \leq 1}.
  \end{align*}
  This is rewritten as
  \begin{equation*}
    \Kk_x = \S_x \cap \underbrace{\enscond{v}{\foralls i \in I, \, - \abs{I}v_is_i \leq 1}}_{=\Kk_x^\prime} .
  \end{equation*}
  Thus the subdifferential gauge reads
  \begin{equation*}
    \antigx(\eta) =
    \gauge_{\Kk_x}(\eta) =
    \max(\gauge_{\S_x}(\eta), \gauge_{\Kk_x^\prime}(\eta)) .
  \end{equation*}
  We have $\gauge_{\S_x}(\eta) = \iota_{\S_x}(\eta)$ and $\gauge_{\Kk_x^\prime}(\eta) = \umax{i \in I} (-\abs{I}s_i\eta_i)_+$, where $(\cdot)_+$ is the positive part, hence we obtain
  \begin{equation*}
    \antigx(\eta) = 
    \begin{cases}
      \umax{i \in I} (-\abs{I}s_i\eta_i)_+ & \text{if } \eta \in \S_x \\
      + \infty & \text{otherwise} .
    \end{cases}
  \end{equation*}
  Therefore the subdifferential of $\normi{\cdot}$ at $x$ takes the form
  \begin{equation*}
    \partial J(x) =
    \enscond{\eta \in \RR^N}
    {
      \eta_{T_x} = e_x = \frac{s}{\abs{I}}
      \qandq
      \umax{i \in I} (-\abs{I}s_i\eta_i)_+ \leq 1
    } .
  \end{equation*}
\end{exmp}

%%%%
\subsubsection{Decomposability of the subdifferential}

Piecing together the above ingredients yields a fundamental pointwise decomposition of the subdifferential of the regularizer $\J$. This decomposability property is at the heart of our results in the rest of the paper.
\begin{thm}[Decomposability]\label{thm:decomp}
  Let $\J$ be a convex function.
  Let $x \in \RR^N \setminus \ens{0}$ and $\f_x \in \ri \partial J(x)$.
  Then the subdifferential of $\J$ at $x$ reads
    \begin{equation*}
      \partial J(x) =
      \enscond{\eta \in \RR^N}
      {
        \eta_{T_x} = e_x
        \qandq
        \antigx(\proj_{\S_x}(\eta - \f_x)) \leq 1
      } .
    \end{equation*}
\end{thm}

%%%%
\subsubsection{First-order minimality condition}

Capitalizing on Theorem~\ref{thm:decomp}, we are now able to deduce a convenient necessary and sufficient first-order (global) minimality condition of~\eqref{eq:reg} and~\eqref{eq:reg-noiseless}.
\begin{prop}\label{prop:foc}
  Let $x \in \RR^N$, and denote for short $\T = \T_x$ and $\S = \S_x$.
  The two following propositions hold.
  \begin{enumerate}[(i)]
  \item The vector $x$ is a global minimizer of~\eqref{eq:reg} if, and only if,
    \begin{equation*}
      \Phi_{\T}^*(y - \Phi x) = \lambda \e_x
      \qandq
      \antigx(\lambda^{-1} \Phi_\S^* (y - \Phi x) - \proj_\S(f_x)) \leq 1 .
    \end{equation*}
  \item The vector $x$ is a global minimizer of~\eqref{eq:reg-noiseless} if, and only if, there exists a dual vector $\alpha \in \RR^Q$ such that
    \begin{equation*}
      \Phi_{\T}^* \alpha = \e_x
      \qandq
      \antigx(\Phi_\S^* \alpha - \proj_\S(f_x)) \leq 1 .
    \end{equation*}
  \end{enumerate}
\end{prop}

%%%%%%%%%%%%%%%%%%%%%%%%%%%%%%%%%%%%%%%%%%%%%%%%%%%%%%%%%%%%%%%%
\subsection{Strong Gauge}
\label{sec:decomposable-special}
In this section, we study a particular subclass of regularizers $\J$ that we dub strong gauges. We start with some definitions.

\begin{defn}
  A finite-valued regularizing gauge $\J$ is \emph{separable} with respect to $\T = \S^\bot$ if
  \begin{equation*}
    \foralls (x,x') \in \T \times \S, \quad
    \J(x + x') = \J(x) + J(x') .
  \end{equation*}
\end{defn}
%Separable norms were studied for instance in~\cite{negahban2010unified}.

Separability of $\J$ is equivalent to the following property on the polar $\J^\circ$.
\begin{lem}\label{lem:eq-separable}
  Let $\J$ be a finite-valued gauge.
  Then, $\J$ is separable w.r.t. to $\T = \S^\bot$ if, and only if its polar $\J^\circ$ satisfies
    \begin{equation*}
      \Js(x+x') = \max \left( \Js(x), \Js(x') \right), \quad \foralls (x,x') \in \T \times \S ~.
    \end{equation*}
\end{lem}

The decomposability of $\partial \J(x)$ as described in Theorem~\ref{thm:decomp} depends on the particular choice of the map $x \mapsto f_x \in \ri \partial \J(x)$. An interesting situation is encountered when $e_x \in \ri \partial J(x)$, in which case, one can just choose $f_x = e_x$, hence implying that $f_{\S_x} = 0$. Strong gauges are precisely a class of gauges for which this situation occurs.

In the sequel, for a given model subspace $\T$, we denote $\widetilde \T$ the set of vectors sharing the same $\T$,
\begin{equation*}
    \widetilde \T =
    \enscond{x \in \RR^N}{\T_x = T} .
  \end{equation*}
Using positive homogeneity, it is easy to show that $\T_{\rho x} = \T_{x}$ and $e_{\rho x}=e_{x}$ $\forall \rho > 0$, see Proposition~\ref{prop:convex-basics-decompos-gauge}(i). Thus $\widetilde \T$ is a non-empty cone which is contained in $\T$ by Proposition~\ref{prop:convex-basics-decompos-gauge}(ii).

\begin{defn}[Strong Gauge]\label{defn:strong-gauge}
  A \emph{strong gauge} on $\T$ is a finite-valued gauge $\J$ such that 
  \begin{enumerate}
  \item For every $x \in \widetilde \T$, $e_x \in \ri \partial \J(x)$.
  \item $\J$ is separable with respect to $\T$ and $\S = \T^\bot$.
  \end{enumerate}
\end{defn}

The following result shows that the decomposability property of Theorem~\ref{thm:decomp} has a simpler form when $\J$ is a strong gauge.
\begin{prop}\label{prop:strong-dec}
  Let $\J$ be a strong gauge on $\T_x$.
  Then, the subdifferential of $\J$ at $x$ reads
    \begin{equation*}
      \partial J(x) =
      \enscond{\eta \in \RR^N}
      {
        \eta_{T_x} = e_x
        \qandq
        \Js(\eta_{\S_x}) \leq 1
      } .
    \end{equation*}    
\end{prop}
When $\J$ is in addition a norm, this coincides exactly with the decomposability definition of~\cite{candes2011simple}. Note however that the last part of assertion (ii) in Proposition~\ref{prop:antig-polar} is an intrinsic property of the polar of the subdifferential gauge, while it is stated as an assumption in ~\cite{candes2011simple}.  

\begin{exmp}[$\lun$-$\ldeux$ norm]
Recall the notations of this example in Section~\ref{sec:convex-geom}. Since $e_x = (\Nn(x_b))_{b \in \Bb} \in \ri \partial \J(x)$, and the $\lun$-$\ldeux$ norm is separable, it is a strong norm according to Definition~\ref{defn:strong-gauge}. Thus, its subdifferential at $x$ reads
  \begin{equation*}
    \partial \J(x) =
    \enscond{\eta \in \RR^N}
    {
      \eta_{T_x} = e_x = (\Nn(x_b))_{b \in \Bb}
      \qandq
      \umax{b \not\in I} \norm{\eta_b} \leq 1
    } .    
  \end{equation*}
\end{exmp}
Note however that, except for $N=2$, $\linf$ is not a strong gauge.

%%% Local Variables: 
%%% mode: latex
%%% TeX-master: "../robust-convex-regularization"
%%% End: 

  % !TEX root = ../IMAIAI-PartlySmoothLinear.tex

\section{Uniqueness}
\label{sec:uniqueness}

This section derives sufficient conditions under which the solution of problems~\eqref{eq:reg} and~\eqref{eq:reg-noiseless} is unique.
%We first provide in Section~\ref{sec:uniqueness-nsp} a sufficient condition of uniqueness which can be viewed as a generalization of the null space property.
%Then in Section~in~\ref{sec:uniqueness-topological}, we give a weaker sufficient condition, stated in terms of a dual vector, the existence of which certifies uniqueness.\\

We start with the key observation that although~\eqref{eq:reg} does not necessarily have a unique minimizer in general, all solutions share the same image under $\Phi$.
\begin{lem}\label{lem:same-image}
  Let $x,x'$ be two solutions of~\eqref{eq:reg}.
  Then,
  \begin{equation*}
    \Phi x = \Phi x'.
  \end{equation*}
\end{lem}
Consequently, the set of the minimizers of~\eqref{eq:reg} is a closed convex subset of the affine space $x + \Ker(\Phi)$, where $x$ is any minimizer of~\eqref{eq:reg}. This is also obviously the case for~\eqref{eq:reg-noiseless} since all feasible solutions belong to the affine space $x_0+ \Ker \Phi$.

\subsection{The Strong Null Space Property}
\label{sec:uniqueness-nsp}
The following theorem gives a sufficient condition to ensure uniqueness of the solution to~\eqref{eq:reg} and~\eqref{eq:reg-noiseless}, that we coin \emph{Strong Null Space Property}. This condition is a generalization of the Null Space Property introduced in~\cite{donoho01uncertainty} and popular in $\lun$ regularization.

\begin{thm}\label{thm:snsp}
  Let $x$ be a solution of~\eqref{eq:reg} (resp. a feasible point of~\eqref{eq:reg-noiseless}). Denote $\T=\S^\perp=\T_x$ the associated model subspace.
  If the \emph{Strong Null Space Property} holds
  \begin{equation}\label{eq:nsps}\tag{NSP$^\text{S}$}
    \forall \delta \in \Ker(\Phi) \setminus \ens{0}, \quad
    \dotp{e_x}{\delta_{\T}} + \dotp{\proj_{\S}(\f_{x})}{\delta_{\S}}
    < \antigPx(- \delta_{\S}) ,
  \end{equation}
  then the vector $x$ is the unique minimizer of~\eqref{eq:reg} (resp.~\eqref{eq:reg-noiseless}).  
\end{thm}

This result reduces to the one proved in~\cite{fadili13stable} when $\J$ is a strong norm, i.e. decomposable in the sense of~\cite{candes2011simple}, pre-composed by a linear operator. Note that when specializing \eqref{eq:nsps} to a strong gauge $\J$, it reads
\begin{equation*}
  \forall \delta \in \Ker(\Phi) \setminus \ens{0}, \quad
  \dotp{e_x}{\delta_{\T_x}}
  < \J(- \delta_{\S_x}).  
\end{equation*}

\subsection{Dual Certificates}
\label{sec:uniqueness-topological}

In this section we derive from~\eqref{eq:nsps} a weaker sufficient condition, stated in terms of a dual vector, the existence of which certifies uniqueness.

For some model subspace $\T$, the restricted injectivity of $\Phi$ on $\T$ plays a central role in the sequel. This is achieved by imposing that
\begin{equation}\label{eq:injT}\tag{$\Cc_T$}
  \Ker(\Phi) \cap \T = \ens{0} .
\end{equation}
To understand the importance of~\eqref{eq:injT}, consider the noiseless case where we want to recover a vector $x_0$ from $y=\Phi x_0$, whose model subspace is $\T$. Assume that the latter is known. From Proposition~\ref{prop:convex-basics-decompos}(iv), $x_0 \in \T \cap \{x: y = \Phi x\}$. For $x_0$ to be uniquely recovered from $y$,~\eqref{eq:injT} must be verified. Otherwise, if~\eqref{eq:injT} does not hold, then any $x_0+\delta$, with $\delta \in \Ker \Phi \cap \T \setminus \{0\}$, is also a candidate solution. Thus, such objects cannot be uniquely recovered.\\

We can derive from Theorem~\ref{thm:snsp} the following corollary.
\begin{cor}\label{prop:uniqueness-topo}
  Let $x$ be a solution of~\eqref{eq:reg} (resp. a feasible point of~\eqref{eq:reg-noiseless}).
  Assume that there exists a dual vector $\alpha$ such that $\eta = \Phi^*\alpha \in \ri( \partial \J(x) )$, and~\eqref{eq:injT} holds where $T = \T_x$.
  Then $x$ is the unique solution of~\eqref{eq:reg} (resp.~\eqref{eq:reg-noiseless}).
\end{cor}

Piecing together Theorem~\ref{thm:decomp} and Corollary~\ref{prop:uniqueness-topo}, one can build a particular dual certificate for~\eqref{eq:reg}, and then state a sufficient uniqueness explicitly in terms of the decomposable structure of the subdifferential of the regularizer $\J$.

\begin{thm}\label{thm:focu}
  Let $x \in \RR^N$, and suppose that $f_x \in \ri \partial \J(x)$.
  Assume furthermore that~\eqref{eq:injT} holds for $\T = \T_x$ and let $\S = \T^\bot$.
  \begin{enumerate}[(i)]
  \item
    If
    \begin{align}
      &\Phi_{\T}^* ( y-\Phi x)  = \lambda e_x, \label{eq:cs-min-T} \\
      &\antigx \left( \lambda^{-1} \Phi_{\S}^*(y - \Phi x)  - \proj_{\S}(f_x) \right) < 1 .  \label{eq:cs-min-S} 
    \end{align}
    then $x$ is the unique solution of~\eqref{eq:reg}.
  \item
    If there exists a dual certificate $\alpha$ such that
    \begin{equation*}
      \Phi_{\T}^* \alpha = e_x
      \qandq \antigx \left( \Phi_{\S}^* \alpha  - \proj_\S(f_x) \right) < 1, 
    \end{equation*}
    then $x$ is the unique solution of~\eqref{eq:reg-noiseless}.
  \end{enumerate}
\end{thm}

%%% Local Variables: 
%%% mode: latex
%%% TeX-master: "../robust-convex-regularization"
%%% End: 

  % !TEX root = ../IMAIAI-PartlySmoothLinear.tex

\section{Partly Smooth Functions Relative to a Linear Manifold}
\label{sec:stable}

Until now, except of being convex and finite-valued (i.e. full domain), no other assumption was imposed on the regularizer $\J$. But, toward the goal of studying robust recovery by solving~\eqref{eq:reg}, more will be needed. This is the main reason underlying the introduction of a subclass of finite-valued convex function $\J$ for which the mappings $x \mapsto e_x$, $x \mapsto \proj_{\S_x}(f_x)$ and $x \mapsto \antig_{f_x}$ exhibit local regularity.

%%%%%%%%%%%%%%%%%%%%%%%%%%%%%%%%%%%%%%%%%%%%%%%%%%%%%%%%%%%%%%%%%%%%%%%%%%%
\subsection{Partly Smooth Functions}
\label{sec:stable-psg}

The notion of ``partly smooth'' functions~\cite{lewis2002active} unifies many non-smooth functions known in the literature.
Partial smoothness (as well as identifiable surfaces~\cite{wright1993ident}) captures essential features of the geometry of non-smoothness which are along the so-called ``active/identifiable manifold''.
Loosely speaking, a partly smooth function behaves smoothly as we move on the partial smoothness manifold, and sharply if we move normal to the manifold.
In fact, the behaviour of the function and of its minimizers (or critical points) depend essentially on its restriction to this manifold, hence offering a powerful framework for sensitivity analysis theory.
In particular, critical points of partly smooth functions move stably on the manifold as the function undergoes small perturbations~\cite{lewis2002active,lewis2013partial}.

Specialized to finite-valued convex functions, the definition of partly smooth functions reads as follows. 

\begin{defn}\label{defn:psg}
  A finite-valued convex function $J$ is said to be \emph{partly smooth} at $x$ relative to a set $\Mm \subseteq \RR^N$ if
  \begin{enumerate}
  \item \textbf{Smoothness.}
    $\Mm$ is a $\Cdeux$-manifold around $x$ and $J$ restricted to $\Mm$ is $\Cdeux$ around x.
  \item \textbf{Sharpness.}
    The tangent space of $\Mm$ at $x$ is the model space $T_x$,
    \begin{equation*}
      \tgtManif{x}{\Mm} = T_x .
    \end{equation*}
  \item \textbf{Continuity.}
    The set-valued mapping $\partial J$ is continuous at $x$ relative to $\Mm$.
  \end{enumerate}
  The manifold $\Mm$ is coined the \emph{model manifold} of $x \in \RR^N$.
  $J$ is said to be \emph{partly smooth relative to a set $\Mm$} if $\Mm$ is a manifold and $J$ is partly smooth at each point $x \in \Mm$ relative to $\Mm$.
\end{defn}

Since $J$ is proper convex and finite-valued, the subdifferential of $\partial \J(x)$ is everywhere non-empty and compact and every subgradient is regular.
Therefore, the Clarke regularity property~\cite[Definition~2.7(ii)]{lewis2002active} is automatically verified.
In view of~\cite[Proposition~2.4(i)-(iii)]{lewis2002active}, our sharpness property is equivalent to that of~\cite[Definition~2.7(iii)]{lewis2002active}.
Obviously, any smooth function $\J: \RR^N \to \RR$ is partly smooth relative to the manifold $\RR^N$.
Moreover, if $\Mm$ is a manifold around $x$, the indicator function $\iota_\Mm$ is partly smooth at $x$ relative to $\Mm$.
Remark that in the previous definition, $\Mm$ needs only to be defined locally around $x$, and it can be shown to be locally unique, see~\cite[Corollary~ 4.2]{HareLewis04}. Hence the notation $\Mm$ is unambiguous.

%%%%%%%%%%%%%%%%%%%%%%%%%%%%%%%%%%%%%%%%%%%%%%%%%%%%%%%%%%%%%%%%%%%%%%%%%%%
\subsection{Partial Smoothness Relative to a Linear Manifold}
\label{sec:stable-gauges}

Many of the partly smooth functions considered in the literature are associated to linear manifolds, i.e. in which case the model subspace is the model manifold $\Mm = T_x$ (see the sharpness property). This class of functions, coined partly smooth functions with linear manifolds, encompasses most of the popular regularizers in signal/image processing, machine learning and statistics. One of course thinks of the $\lun$, $\lun-\ldeux$, $\linf$ norms, their composition by a linear operator, and/or positive combinations of them, to name a few. However, this family of regularizers does not include the nuclear norm, whose manifold is obviously not linear.In the sequel, we restrict our attention to the class functions $\J$ which are finite-valued convex and partly smooth at $x$ with respect to $T_x$.

When the continuity property (Definition~\ref{defn:psg}(iii)) of the set-valued mapping $\partial \J$ is strenghthned to Lipschitz-continuity, it turns out that one can quantify precisely the local regularity of the mappings $x \mapsto e_x$, $x \mapsto \proj_{\S_x}(f_x)$ and $x \mapsto \antigx$. This is formalized as follows.

\begin{thm}\label{thm:lip-prg}
  Let $\Gamma$ be any gauge which is finite and coercive on $T_x$ for $x \in \RR^N$.
  Let $J$ be a partly smooth function at $x$ relative to $T_x$, and assume that $\partial J : \RR^N \rightrightarrows \RR^N$ is Lipschitz-continuous around $x$ relative to $T_x$. Then for any Lipschitz-continuous mapping
  \begin{equation*}
    f:
    \begin{cases}
      T_x &\to \RR^N \\
      \tilde x &\mapsto f_{\tilde x} \in \ri \partial J(\tilde x) ,
    \end{cases}
  \end{equation*}
  there exist four non-negative reals $\nu_x,\mu_x,\tau_x,\xi_x$ such that
  \begin{gather}
    \forall x' \in T, \certCo{x - x'} \leq \nu_x \Rightarrow T_x = T_{x'} \label{eq:lip-nu}
  \end{gather}
  and for every $x' \in \T$ with $\certCo{x - x'} < \nu_x$, one has
  \begin{align}
    \certCo{e_x - e_{x'}} &\leq \mu_x \certCo{x-x'} , \label{eq:lip-mu} \\
    \antigx(\proj_\S(f_x - f_{x'})) &\leq \tau_x \certCo{x-x'} , \label{eq:lip-tau} \\
    \sup_{\substack{u \in \S \\ u \neq 0}} \frac{\antigxp(u) - \antigx(u)}{\antigx(u)}
                          &\leq
                            \xi_x \certCo{x - x'} \label{eq:lip-xi} .
  \end{align}
  Moreover, the mapping $f$ always exists.
\end{thm}

This result motivates the following definition.

\begin{defn}\label{def:PRG}
The set of finite-valued convex and partly smooth functions at $x$ relative to $T_x$, such that $\partial J$ is Lipschitz around $x$ relative to $T_x$, with parameters $(\Gamma,f_x,\nu_x,\mu_x,\tau_x,\xi_x)$, is denoted $\psfl{x}{\Gamma,f_x,\nu_x,\mu_x,\tau_x,\xi_x}$.
\end{defn}

%%%%%%%%%%%%%%%%%%%%%%%%%%%%%%%%%%%%%%%%%%%%%%%%%%%%%%%%%%%%%%%%%%%%%%%%%%%
\subsection{Operations Preserving Partial Smoothness Relative to a Linear Manifold}
\label{sec:stable-algebraic}

The set $\PSFL_x$ is closed under addition and pre-composition by a linear operator.

%%%%%%
\subsubsection{Addition}
\label{sec:alge-sum}

The following proposition determines the model subspace and the subdifferential gauge of the sum of two functions 
\begin{equation*}
  H = J + G
\end{equation*}
in terms of those associated to $J$ and $G$. 

\begin{prop}\label{prop:sum-dec}
  Let $J$ and $G$ be two finite-valued convex functions. Denote $\TtJ$ and $e_J$ (resp. $\TtG$ and $e_G$) the model subspace and vector at a point $x$ corresponding to $J$ (resp. $G$).
  Then the subdifferential of $H$ has the decomposability property with
  \begin{enumerate}[(i)]
  \item $\TtH = \TtJ \cap \TtG$, or equivalently $\SsH = (\TtH)^\bot = \Span\pa{ \SsJ \cup \SsG}$.
  \item $e_H = \proj_{\TtH}(e_J+e_G)$.
  \item Moreover, let $\AntiG{J}$ and $\AntiG{G}$ denote the subdifferential gauges for the pairs $(J, \f_{x}^{J}\in \ri \partial J(x))$ and $(G, \f_{x}^{G}\in \ri \partial G(x))$, correspondingly.
  Then, for the particular choice of
  \begin{equation*}
    f_{x}^{H} = f_{x}^{J}+ f_{x}^{G}
  \end{equation*}
  we have $f_{x}^{H} \in \ri \partial H(x)$, and for a given $\eta \in \SsH$, the subdifferential gauge of $H$ reads
  \[
  \AntiG{H} (\eta) = \inf_{\eta_1+\eta_2=\eta} \max(\AntiG{J}(\eta_1),\AntiG{G}(\eta_2)) ~.
  \]
  \end{enumerate}
\end{prop}

Armed with this result, we show the following.
\begin{prop}\label{prop:sum-prg}
Let $x \in \RR^N$. Suppose that
\[
J \in \psfl{x}{\certCof^J,\nu_x^J,\mu_x^J,\tau_x^J,\xi_x^J} \qandq G \in \psfl{x}{\certCof^G, \nu_x^G,\mu_x^G,\tau_x^G,\xi_x^G}.
\] 
Then, for the choice $f_{x}^{H} = f_{x}^{J}+ f_{x}^{G}$ and $\certCof^H = \max (\certCof^J,\certCof^G)$, we have
\[
H=J+G \in \psfl{x}{\certCof^H, \nu_x^H,\mu_x^H,\tau_x^H,\xi_x^H}
\] 
with  
  \begin{align*}
    &\nu_x^H = \min(\nu_x^J, \nu_x^G) \\
    &\mu^H_x = \mu_x^J \, \normOP{\proj_{\TtH}}{\certCof^J} {\certCof^H} + \mu_x^G\, \normOP{\proj_{\TtH}}{\certCof^G} {\certCof^H}\\
    &\tau_x^H = \tau^J_x+\tau^G_x+ \mu^J_x \,  \normOP{\proj_{\SsH \cap \TtJ}}{\certCof^J}{\AntiG{H}} + \mu^G_x \,  \normOP{\proj_{\SsH \cap \TtG}}{\certCof^G}{\AntiG{H}} \\ 
    &\xi_x^H = \max (\xi_x^J, \xi_x^G) .
  \end{align*}
\end{prop}

%%%%
\subsubsection{Smooth perturbation} 

It is common in the litterature to find regularizers of the form $J_\epsilon(x) = J(x) + \frac{\epsilon}{2} \norm{x}_2^2$, such as the Elastic net ~\cite{zou2005elastic} . More generally, we consider any smooth perturbation of $J$. The following is a straighforward consequence of Proposition~\ref{prop:sum-dec}.
\begin{cor}\label{cor:sum-smooth-sum}
  Let $J$ be a finite-valued convex function, $x \in \RR^N$ and $G$ a convex function which is G\^ateaux-differentiable at $x$.
  Then,
  \begin{equation*}
    T_x^{J+G} = T^{J} \qandq e_x^{J+G} = e_x^{J} + P_{T_x^{J}} \nabla G(x) .
  \end{equation*}
  Moreover, for the particular choice of
  \begin{equation*}
    f_{x}^{J+G} = f_x^{J} + \nabla G(x),
  \end{equation*}
  we have $f_{x}^{J+G} \in \ri (J+G)(x)$ and for a given $\eta \in S_x^J$, the subdifferential gauge of $J + G$ reads
  \begin{equation*}
    (J+G)_{f_x^{J+G},x}^{x,\circ}(\eta) = J_{f_x^J,x}^{x,\circ}(\eta) .
  \end{equation*}
\end{cor}
Hence, the model subspace $T_x$ and the subdifferential gauge are insensitive to smooth perturbations.
%Remark that the function $G : x \mapsto \frac{\epsilon}{2} \norm{x}_2^2$ is $\Calt{\infty}$ everywhere.
%If $J$ is a gauge, hence $x \in T_x$, we get that the model vector reads $e_x^{J + G} = e_x^J + \epsilon x$.
Combining Proposition~\ref{prop:sum-prg} and Corollary~\ref{cor:sum-smooth-sum} yields the partial smoothness Lipschitz constants of smooth perturbation.
\begin{cor}\label{cor:sum-smooth-prg}
Let $x \in \RR^N$. Suppose that $J \in \psfl{x}{\certCof^J,\nu_x^J,\mu_x^J,\tau_x^J,\xi_x^J}$, that $G$ is $\Cdeux$ on $\RR^N$ with a $\beta$-Lipschitz gradient. Then for the choice $f_{x}^{H} = f_x^{J} + \nabla G(x)$ and $\certCof^H = \max (\certCof^J,\norm{\cdot})$, $H=J+G \in \psfl{x}{\certCof^H, \nu_x^H,\mu_x^H,\tau_x^H,\xi_x^H}$ with  
  \begin{gather*}
    \nu_x^H = \nu_x^J, \quad
    \mu^H_x = \mu_x^J \, \normOP{\proj_{\TtJ}}{\certCof^J}{\certCof^H} + \beta \, \normOP{\proj_{\TtJ}}{\ldeux} {\certCof^H},  \\
    \tau_x^H = \tau^J_x, \quad
    \xi_x^H = \xi_x^J . 
  \end{gather*}
\end{cor}

%%%%%
\subsubsection{Pre-composition by a Linear Operator}
\label{sec:alge-precomp}

Convex functions of the form $\J_0 \circ D^*$, where $\J_0$ is a finute-valued convex function, correspond to the so-called analysis-type regularizers.
The most popular example in this class if the total variation where $\J_0$ is the $\lun$ or the $\lun-\ldeux$ norm, and $D^*=\nabla$ is a finite difference discretization of the gradient.

In the following, we denote $\T=\T_x=\S^\bot$ and $e = \e_x$ the subspace and vector in the decomposition of the subdifferential of $\J$ at a given $x \in \RR^N$. Analogously, $\Tz=\Sz^\bot$ and $\ez$ are those of $\J_0$ at $D^* x$. The following proposition details the decomposability structure of analysis-type regularizers.

\begin{prop}\label{prop:analysis-dec-form} 
Let $\J_0$ be a convex finite-valued function. Then the subdifferential of $\J = \J_0 \circ D^*$ has the decomposability property with
  \begin{enumerate}[(i)]
  \item $\T = \Ker( D_{\Sz}^* )$, or equivalently $\S = \Im( D_{\Sz} )$.
  \item $e = \proj_\T D \ez$. 
  \item Moreover, let $\antigxa$ denote the subdifferential gauge for the pair $(\J_{0}, \f_{0,D^* x} \in \ri \partial \J_0(x))$.
  Then, for the particular choice of
  \begin{equation*}
    \f_x = D \f_{0,D^*x}
  \end{equation*}
  we have $f_{x} \in \ri \partial \J(x)$, $\dom \antigx = \S$ and for every $\eta \in \S$
  \[
  \antigx(\eta) = \inf_{z \in \ker(D_{\Sz})} \antigxa(D_{\Sz}^+ \eta + z) ~.
  \]
  The infimum can be equivalently taken over $\ker(D) \cap \Sz$.
  \end{enumerate}
\end{prop}

Capitalizing on these properties, we now establish the following.
\begin{prop}\label{prop:gauge-analysis-stable}
Let $x \in \RR^N$ and $u=D^* x$. Suppose that $\J_0 \in \psfl{u}{\Gamma_0,\nu_{0,u},\mu_{0,u},\tau_{0,u},\xi_{0,u}}$. Then with the choice $f_x = D f_{0,u}$ and $\Gamma$ any finite-valued coercive gauge on $\T$, $\J = \J_0 \circ D^* \in \psfl{x}{\Gamma,\nu_{x},\mu_{x},\tau_{x},\xi_{x}}$, with
  \begin{align*}
    \nu_x  &= 
    \frac{1}{\normOP{D^*}{\Gamma}{\Gamma_0}}\nu_{0,u}\\
    \mu_x  &=
    \mu_{0,u}
    \normOP{P_\T D}{\Gamma}{\Gamma_0}
    \normOP{D^*}{\Gamma}{\Gamma_0} \\
    \tau_x &=
    \left(
      \tau_{0,u}
      \normOP{D_{\Sz}^+ \proj_\S D}{\J_{0,f_{0,u}}^{u,\circ}}{\J_{0,f_{0,u}}^{u,\circ}}
      +
      \mu_{0,u}
      \normOP{D_{\Sz}^+ \proj_\S D}{\Gamma_0}{\J_{0,f_{0,u}}^{u,\circ}}
    \right)
    \normOP{D^*}{\Gamma}{\Gamma_0}     \\
    \xi_x  &=
    \xi_{0,u}\normOP{D^*}{\Gamma}{\Gamma_0} ~.
  \end{align*}
\end{prop}

%%% Local Variables: 
%%% mode: latex
%%% TeX-master: "../robust-convex-regularization"
%%% End: 

  % !TEX root = ../IMAIAI-PartlySmoothLinear.tex

\section{Exact Model Selection and Identifiability}
\label{sec:small}

In this section, we state our main recovery guarantee, which asserts that under appropriate conditions, \eqref{eq:reg} with a partly smooth function $\J$ relative to a linear manifold has a unique solution $\xsol$ such that its model subspace $\T_{\xsol} = \T_{x_0}$, even in presence of small enough noise. Put differently, regularization by $\J$ is able to stably select the correct model subspace underlying $x_0$.

\subsection{Linearized Precertificate}
\label{sec:small-identifiability}

Let us first introduce the definition of the linearized precertificate.
\begin{defn}\label{defn:minimal-pre}
  The \emph{linearized precertificate} $\ce_F$ for $x \in \RR^N$ is defined by
  \begin{equation*}
    \ce_F = \uargmin{\Phi_{T_x}^* \ce = e_x} \norm{\ce}.
  \end{equation*}
\end{defn}
The subscript $F$ is used as a salute to J.-J. Fuchs~\cite{fuchs2004on-sp} who first considered this vector as a dual certificate for $\lun$ minimization. The intuition behind it is well-understood if one realizes that the existence of a dual certificate $\ce$ is equivalent to $\eta = \Phi^* \ce$ for some $\ce$ such that $\eta_T = e_x$ and $\antigx (\eta_S - P_S f_x) \leq 1$.
Dropping the last constraint, and choosing the minimal $\ldeux$-norm solution to the first constraint recovers the definition of $\ce_F$.

A convenient property of this vector, is that under the restricted injectivity condition, it has a closed form expression.

\begin{lem}\label{lem:closed-form-cef}
  Let $x \in \RR^N$ and suppose that~\eqref{eq:injT} is verified with $T=T_x$.
  Then $\ce_F$ is well-defined and
  \begin{equation*}
    \ce_F = \Phi_{T_x}^{+,*} e_x^{} .
  \end{equation*}
\end{lem}

Beside condition $(\Cc_{\T_x})$ stated above, the following Irrepresentability Criterion will play a pivotal role. 
\begin{defn}\label{defn:ic}
  For $x \in \RR^N$ such that $(\Cc_{\T_x})$ with $T = T_x$ holds, we define the \emph{Irrepresentability Criterion} at $x$ as
  \begin{equation*}
    \IC(x) = \antigx(  \Phi_{\S_x}^* \Phi_{\T_x}^{+,*} e_x - \proj_{S_x} f_x ) .
  \end{equation*}
\end{defn}
A fundamental remark is that $\IC(x) < 1$ is the analytical equivalent to the topological non-degeneracy condition $\Phi^* \ce_F \in \ri \partial J(x)$.
Note that if $\J$ is a strong gauge on $T$, then it reads $\IC(x) = \Js(  \Phi_{\S_x}^* \Phi_{\T_x}^{+,*} e_x )$.
The Irrepresentability Criterion clearly brings into play the promoted subspace $\T_{x}$ and the interaction between the restriction of $\Phi$ to $\T_{x}$ and $\S_{x}$. It is a generalization of the irrepresentable condition that has been studied in the literature for some popular regularizers, including the $\lun$-norm \cite{fuchs2004on-sp}, analysis-$\lun$ \cite{vaiter2011robust}, and $\lun$-$\ldeux$ \cite{bach2008consistency}. See Section~\ref{sec:examples} for a comprehensive discussion.

We begin with the noiseless case, i.e. $w=0$ in \eqref{eq:ip}. In fact, in this setting, $\IC(x_0) < 1$ is a sufficient condition for identifiability without any any other particular assumption on the finite-valued convex function $\J$, such as partial smoothness. By identifiability, we mean the fact that $x_0$ is the unique solution of~\eqref{eq:reg-noiseless}.

\begin{thm}\label{thm:identifiability}
  Let $x_0 \in \RR^N$ and~$\T=\T_{x_0}$.
  We assume that~\eqref{eq:injT} holds and $\IC(x_0) < 1$.
  Then $x_0$ is the unique solution of \eqref{eq:reg-noiseless}.
\end{thm}

%%%%%%%%%%%%%%%%%%%%%%%%%%%%%%%%%%%%%%%%%%%%%%%%%%%%%%%
\subsection{Exact Model Selection}
\label{sec:small-exact}

It turns out that even in presence of noise in the measurements $y$ according to~\eqref{eq:ip}, condition $\IC(x_0) < 1$ is also sufficient for \eqref{eq:reg} with $PSFL_{x_0}$ regularizer to stably recover the model subspace underlying $x_0$. This is stated in the following theorem.

\begin{thm}\label{thm:local-stab}
  Let $x_0 \in \RR^N$ and $\T=\T_{x_0}$.
  Suppose that $\J \in \psfl{x_0}{\Gamma,\nu_{x_0},\mu_{x_0},\tau_{x_0},\xi_{x_0}}$.
  Assume that~\eqref{eq:injT} holds and $\IC(x_0) < 1$.
  Then there exist positive constants $(A_\T,  B_\T)$ that solely depend on $\T$ and a constant
  $C(x_0)$ such that if $w$ and $\lambda$ obey
  \begin{equation}\label{eq-constr-lambda-robustness}
    \frac{A_\T}{1-\IC(x_0)} \norm{w} 
    \leq \lambda \leq 
    \nu_{x_0} \min \big(  B_\T, C(x_0) \big)    
  \end{equation}
  the solution $\xsol$ of \eqref{eq:reg} with noisy measurements $y$ is unique, and satisfies $\T_{\xsol} = T$. 
  Furthermore, one has
  \begin{equation*}
    \norm{x_0-\xsol} = O\Big(\max(\norm{w},\lambda)\Big).
  \end{equation*}
\end{thm}
Clearly this result asserts that exact recovery of $\T_{x_0}$ from noisy partial measurements is possible with the proviso that the regularization parameter $\lambda$ lies in the interval~\eqref{eq-constr-lambda-robustness}. The value $\lambda$ should be large enough to reject noise, but small enough to recover the entire subspace $\T_{x_0}$. In order for the constraint~\eqref{eq-constr-lambda-robustness} to be non-empty, the noise-to-signal level $\norm{w}/\nu_{x_0}$ should be small enough, i.e.
\begin{equation*}
  \frac{\norm{w}}{\nu_{x_0}} \leq 
  \frac{1-\IC(x_0)} {A_\T}  \min\pa{B_\T, C(x_0)} ~. 
\end{equation*}
The constant $C(x_0)$ involved in this bound depends on $x_0$ and has the form
\begin{gather*}
  C(x_0) =  \frac{1-\IC(x_0)}{\xi_{x_0} \nu_{x_0} } H\pa{ \frac{D_\T \:\mu_{x_0} +  \tau_{x_0}}{\xi_{x_0}} } \\
  \qwhereq
  H(\beta) = \frac{ \beta + 1/2 }{ E_T\, \beta }	
  \, \phi\pa{ \frac{2\beta}{ (\beta + 1 )^2 } }  \qandq \phi(u) = \sqrt{1+u}-1 ~.
\end{gather*}
The constants $(D_\T,E_\T)$ only depend on $\T$. $C(x_0)$ captures the influence of the parameters $\pi_{x_0} = (\mu_{x_0},\tau_{x_0},\xi_{x_0})$, where the latter reflect the geometry of the partly smooth regulrizer $\J$ at $x_0$. More precisely, the larger $C(x_0)$, the more tolerant the recovery is to noise. Thus favorable regularizers are those where $C(x_0)$ is large, or equivalently where $\pi_{x_0}$ has small entries, since $H$ is a strictly decreasing function.

It is worth noting that this analysis is in some sense sharp following the argument in~\cite[Proposition 1]{2014-vaiter-ps-stability}.
The only case not covered by our analysis is when $\IC(x) = 1$.

%%% Local Variables: 
%%% mode: latex
%%% TeX-master: "../robust-convex-regularization"
%%% End: 

  % !TEX root = ../IMAIAI-PartlySmoothLinear.tex

\section{Examples of Partly Smooth Functions Relative to a Linear Manifold}
\label{sec:examples}

\subsection{Synthesis $\ell^1$ Sparsity}
\label{sec:ex-sparsity}
The regularized problem \eqref{eq:reg} with $\J(x)= \norm{x}_1 = \sum_{i=1}^N \abs{x_i}$ promotes sparse solutions. It goes by the name of Lasso \cite{tibshirani1996regre} in the statistical literature, and Basis Pursuit DeNoising (or Basis Pursuit in the noiseless case) \cite{chen1999atomi} in signal processing. 

%%%%
\subsubsection{Structure of the $\lun$ norm}

The norm $\J(x)= \norm{x}_1$ is a symmetric (finite-valued) strong gauge. 
More precisely, we have the following result.
\begin{prop}\label{prop:ex-lun-swdg}
  $J = \normu{\cdot}$ is a symmetric strong gauge with
  \begin{gather*}
    \T_x = \enscond{\eta \in \RR^N}{ \forall j \not\in I, \, \eta_j = 0 }, \quad
    \S_x = \enscond{\eta \in \RR^N}{ \forall i \in I, \, \eta_i = 0 }, \\
    e_x = \sign(x), \quad
    f_x = e_x, \quad \Js_{f_x} = \normi{\cdot} + \iota_{\S_x} ~,
  \end{gather*}
  where $I=I(x) = \enscond{i}{x_i \neq 0}$.
  Moreover, it is partly smooth relative to a linear manifold with
  \begin{equation*}
    \certCof = \normi{\cdot}, \quad
    \nu_x = (1-\delta)\umin{i \in I} \abs{x_i} \, , \delta \in ]0,1]
    \qandq
    \mu_x = \tau_x = \xi_x = 0 .
  \end{equation*}
\end{prop}

% Figure~\ref{fig:l1-gauge-geometry} shows the underlying geometry of the $\lun$ regularization in two dimensions.
% Note that $\partial \J(x)$ is included in the dual closed ball.
% \begin{figure}[htbp]
%   \centering
%   \includegraphics[scale=1]{img/l1-gauge-geometry.pdf} 
%   \caption{$\lun$ geometry.}
%   \label{fig:l1-gauge-geometry}
% \end{figure}

%%%%
\subsubsection{Relation to previous works}

The theoretical recovery guarantees of $\lun$-regularization have been extensively studied in the recent years. There is of course a huge literature on the subject, and covering it comprehensively is beyond the scope of this paper. In this section, we restrict our overview to those works pertaining to ours, i.e., sparsity pattern recovery in presence of noise.

For instance, an irrepresentability criterion was introduced in \cite{fuchs2004on-sp}. Let $s \in \ens{-1,0,+1}^N$ and $I$ its support. Suppose that $\Phi_{(I)}$ has full column rank, which is precisely \eqref{eq:injT} in this case. The synthesis irrepresentability criterion $\IC_{\lun}$ of $s$ is defined as
\begin{equation*}
  \IC_{\lun}(s) = \normi{\Phi_{(I^c)}^* \Phi_{(I)}^{+,*} s_{(I)}^{ }} 
           = \umax{j \in I^c} \abs{\dotp{\Phi_j}{\Phi_{(I)}^{+,*} s_{(I)}}} .
\end{equation*}
From Definition~\ref{defn:ic} and Proposition~\ref{prop:ex-lun-swdg}, one immediately recognizes that $\IC_{\lun}(\sign(x)) = \IC(x)$.
%Since $\normi{\Phi_{(J)}^* \Phi_{(I)}^{+,*} s_{(I)}^{ }} = \normi{\Phi_{\T}^* \Phi_{\S}^{+,*} s}$, thus $\IC_{\lun}(\sign(x)) = \IC(x)$. 
The condition $\IC_{\lun}(\sign(x)) < 1$, also known as the irrepresentable condition in the statistical literature, was proposed \cite{fuchs2004on-sp} for exact support (and sign) pattern recovery with $\ell^1$-regularization from partial noisy measurements. In this respect, this work can then be viewed as a special instance of ours, as Theorem~\ref{thm:local-stab} in this case ensures recovery of the support pattern.
%For partial robust recovery, the the Exact Recovery Criterion (ERC) introduced in~\cite{tropp2006just-} gives a weaker result where only the support $I$ is robust, but the sign might differs.

%%% Local Variables: 
%%% mode: latex
%%% TeX-master: "../robust-convex-regularization"
%%% End: 

\subsection{Analysis $\lun$ Sparsity}
\label{sec:ex-sparsity-analysis}

Let $D = (d_i)_{i=1}^P$ be a collection of $P$ atoms $d_i \in \RR^N$.
The analysis semi-norm associated to $D$ is $\J(x) = \normu{D^* x} = \sum_{i=1}^P \abs{\dotp{d_i}{x}}$.
Obviously, the synthesis $\lun$-regularization corresponds to $D = \Id$.
Popular examples of analysis-type $\lun$ semi-norms include for instance the discrete anisotropic total variation~\cite{rudin1992nonlinear}, the Fused Lasso~\cite{tibshirani2004sparsity} and shift invariant wavelets~\cite{steidl2004equivalence}.

%%%%
\subsubsection{Structure of the analysis $\lun$ semi-norm}

The semi-norm $\J(x) = \normu{D^* x}$ is a symmetric partly smooth function relative to a linear manifold. This is formalized in the following proposition whose proof is a straightforward application of Proposition~\ref{prop:analysis-dec-form}, Proposition~\ref{prop:gauge-analysis-stable} and Proposition~\ref{prop:ex-lun-swdg}.
\begin{prop}\label{prop:ex-lunanalysis-swdg}
  $\J = \normu{D^* \cdot}$ is a symmetric (finite-valued) gauge with
  \begin{gather*}
    \T_x      = \ker(D_{(I^c))}^*) = \enscond{\eta \in \RR^N}{ \forall j \not\in I, \, \dotp{d_j}{\eta_j} = 0 }, \quad
    \S_x      = \Im(D_{I^c}), \\
    e_x       = \proj_{\Ker(D_{I^c}^*)}D\sign(D^* x), \quad
    f_x       = D \sign(D^* x), \\
    \antig_{f_x}(\eta) = \inf_{z \in \ker(D_{(I^c)})} \normi{D_{(I^c)}^+ \eta + z}, \qforq \eta \in \S_x ~,
  \end{gather*}
  where $I=I(x) = \enscond{i}{\dotp{d_i}{x_i} \neq 0}$.
  Moreover, it is partly smooth relative to a linear manifold with parameters
  \begin{equation*}
    \nu_x = (1-\delta)\umin{i \in I} |\dotp{d_i}{x_i}|, \delta \in ]0,1]
    \qandq
    \mu_x = \tau_x = \xi_x = 0 . 
  \end{equation*}
\end{prop}

%%%%
\subsubsection{Relation to previous works}

Some insights on the relation and distinction between synthesis- and analysis-based sparsity regularizations were first given in~\cite{elad2007analysis}. When $D$ is orthogonal, and more generally when $D$ is square and invertible, the two forms of regularization are equivalent in the sense that the set of minimizers of one problem can be retrieved from that of an equivalent form of the other through a bijective change of variable. It is only recently that theoretical guarantees of $\lun$-analysis sparse regularization have been investigated, see \cite{vaiter2011robust} for a comprehensive review.
Among such a work, the authors in~\cite{nam2012cosparse} propose a null space property for identifiability in the noiseless case. The most relevant work to ours here is that of~\cite{vaiter2011robust}, where the authors prove exact robust recovery of the support and sign patterns under conditions that are a specialization of those in Theorem \ref{thm:local-stab}. 

More precisely, let $I$ be the support of $D^* x_0$, and $s$ its sign vector. Denote $\T=\T_{x_0}=\S^\perp=\ker(D_{I^c}^*)$, $e_{x_0}=\sign(D^* x_0)=s$, $e = e_{x_0} = \proj_{\T} D s$, $f = f_{x_0}=D s$. From Definition~\ref{defn:ic} and Proposition~\ref{prop:ex-lunanalysis-swdg}, the criterion $\IC(x_0)$ in this case takes the form
\begin{align*}
\IC(x_0) &= \antig_{f_x}(  \Phi_{\S}^* \Phi_{\T}^{+,*} \proj_{\T} D s - \proj_{S} D s ) \\
	 &= \inf_{z \in \ker(D_{(I^c)})} \normi{D_{(I^c)}^+ \pa{\Phi_{\S}^* \Phi_{\T}^{+,*} \proj_{\T} - \proj_{S}} D s + z} \\
	 &= \inf_{z \in \ker(D_{(I^c)})} \normi{D_{(I^c)}^+ \pa{(\Id - \proj_{\T})\Phi^*\Phi \proj_{\T}(\Phi_{\T}^*\Phi_{\T})^{-1} \proj_{\T} - \proj_{S}} D s + z} \\
	 &= \inf_{z \in \ker(D_{(I^c)})} \normi{D_{(I^c)}^+ \pa{\Phi^*\Phi \proj_{\T}(\Phi_{\T}^*\Phi_{\T})^{-1} \proj_{\T} - (\proj_{\T} + \proj_{S})}D s + z} \\
	 &= \inf_{z \in \ker(D_{(I^c)})} \normi{D_{(I^c)}^+ \pa{\Phi^*\Phi \proj_{\T}(\Phi_{\T}^*\Phi_{\T})^{-1} \proj_{\T} - \Id} D_{(I)} s_{(I)} + z} ~.
\end{align*}
Introducing $U$ as a matrix whose columns form a basis of $\T$, $\IC(x_0)$ can be equivalently rewritten
\begin{align*}
\IC(x_0) &= \inf_{z \in \ker(D_{(I^c)})} \normi{D_{(I^c)}^+ \pa{\Phi^*\Phi A^{[I^c]} - \Id} D_{(I)} s_{(I)} + z} ~,
\end{align*}
where $A^{[I^c]} = U (U^* \Phi^* \Phi U)^{-1} U^*$. We recover exactly the expression of the $\IC_{\lun-D}$ introduced in~\cite{vaiter2011robust}.

% let  $\IC_{\lun-D}$ introduced in~\cite{vaiter2011robust}. Let $s \in \{-1, 0, +1\}^P$, $I$ its support and $J = I^c$. Suppose that $\Ker(\Phi) \cap \Ker(D_{(J)}^*)$ holds. The analysis Identifiability Criterion $\IC_{\lun-D}$ of $s$ is defined as
% \begin{equation*}
%   \IC_{\lun - D}(s) = \umin{u \in \Ker(D_{(J)})}
%   \norm{\Omega s_{(I)} - u}_\infty
% \end{equation*}
% where
% \begin{equation*}
%   \Omega = D_{(J)}^+(\Phi^* \Phi  - \Id)D_{(I)} 
%   \qandq
%   A = U (U^* \Phi^* \Phi U)^{-1} U^*,
% \end{equation*}
% where $U$ is a matrix whose columns form a basis of $\Ker(D_{(J)}^*)$.
%In fact, one proves that in fact $\IC = \IC_{\lun-D}$.
% \begin{prop}\label{prop:analysis-equivalence}
%   Let $x \in \RR^N$.
%   \begin{equation*}
%     \IC(x) = \IC_{\lun-D}(s)
%   \end{equation*}
%   where $s = \sign(D^* x)$.
% \end{prop}

%%% Local Variables: 
%%% mode: latex
%%% TeX-master: "../../robust-convex-regularization"
%%% End: 

\subsection{$\linf$ Antisparsity Regularization}
\label{sec:ex-antisparsity}

Regularization by the $\linf$-norm corresponds to taking $\J(x) = \norm{x}_\infty = \umax{1 \leq i \leq N} \abs{x_i}$. This regularizer promotes flat solutions. It plays a prominent role in a variety of applications including approximate nearest neighbor search \cite{jegou2012anti} or vector quantization \cite{lyubarskii2010uncertainty}; see also \cite{studer12signal} and references therein.

%%%%
\subsubsection{Structure of the $\linf$-norm}

The norm $\J(x) = \norm{x}_\infty$ is a symmetric partly smooth function relative to a linear manifold, but unlike the $\lun$-norm, it is not strongly so (except for $N=2$). Therefore, in the following proposition, we rule out the trivial case $x=0$.
\begin{prop}\label{prop:ex-linf-swdg}
  $\J = \norm{\cdot}_\infty$ is a symmetric (finite-valued) gauge with
  \begin{gather*}
    \S_x = \enscond{\eta}{ \eta_{(I^c)} = 0 \qandq \dotp{\eta_{(I)}}{s_{(I)}}=0 }, \quad
    \T_x = \enscond{\alpha}{ \alpha_{(I)} = \rho s_{(I)} \qforq \rho \in \RR }, \\
    e_x = \frac{s}{\abs{I}}, \quad
    \f_{x} = e_x, \quad
    \antig_{f_x}(\eta) = \umax{i \in I} (-\abs{I} s_i \eta_i)_+ \qforq \eta \in \S_x ~,
  \end{gather*}
  where $s = \sign(x)$ and $I = I(x) = \enscond{i}{ \abs{x_i} = \norm{x}_\infty }$.
  Moreover, it is partly smooth relative to a linear manifold with
  \begin{equation*}
    \Gamma = \normu{\cdot}, \quad \nu_x = (1-\delta)\big(\norm{x}_\infty - \umax{j \notin I} \abs{x_j}\big), \delta \in ]0,1]
    \qandq
    \mu_x = \tau_x = \xi_x = 0 .
  \end{equation*}
\end{prop}

%%%%
\subsubsection{Relation to previous work}

In the noiseless case, i.e. \eqref{eq:reg-noiseless} with $\J=\normi{\cdot}$, theoretical analysis of $\linf$-regularization goes back to the 70's through the work of~\cite{Cadzow71}. \cite{lyubarskii2010uncertainty} provided results that characterize signal representations with small (but not necessarily minimal) $\linf$-norm subject to linear constraints. A necessary and sufficient condition for a vector to be the unique minimizer of \eqref{eq:reg-noiseless} is derived in \cite{mangasarian2011probability}. The work of \cite{donoho2010counting} analyzes recovery guarantees by $\linf$-regularization in a noiseless random sensing setting.

The authors in \cite{studer12signal} analyzed the properties of solutions obtained from a constrained form of \eqref{eq:reg} with $\J=\normi{\cdot}$. In particular, they improved and generalized the bound of \cite{lyubarskii2010uncertainty} on the $\linf$ of the solution.

The work of \cite{bach2010structured,obozinski2012convex} studies robust recovery with regularization using a subclass of polyhedral norms obtained by convex relaxation of combinatorial penalties. Although this covers the case of the $\linf$-norm, their notion of support is however, completely different from ours. We will come back to this work with a more detailed discussion in Section~\ref{sec:ex-polyhedral}.

%The work of \cite{bach2010structured,obozinski2012convex} gives a robustness criterion for recovering the support of a vector using submodular penalties which are convex relaxation of $F(\supp(x))$. 

%%% Local Variables: 
%%% mode: latex
%%% TeX-master: "../robust-convex-regularization"
%%% End: 

\subsection{Group Sparsity Regularization}
\label{sec:ex-block}

 Let's recall from Section~\ref{sec:convex-geom} that $\Bb$ is a uniform disjoint partition of $\{1,\cdots,N\}$, 
  \[
	\{1,\ldots,N\} = \bigcup_{b \in \Bb} b, \quad b \cap b' = \emptyset, ~  \forall b \neq b' ~.
  \]
  The $\lun-\ldeux$ norm of $x$ is 
  \begin{equation*}
    \J(x) = \norm{x}_{\Bb} = \sum_{b \in \Bb} \norm{x_b}.
  \end{equation*}
This prior has been advocated when the signal exhibits a structured sparsity pattern where the entries are assumed to be clustered in few non-zero groups; see for instance \cite{bakin1999adaptive,yuan2005model}. The corresponding regularized problem \eqref{eq:reg} is known as the group Lasso.

%%%%
\subsubsection{Structure of the $\lun$-$\ldeux$ norm}

The $\lun-\ldeux$ norm is a symmetric partly smooth function relative to a linear manifold.
\begin{prop}\label{prop:ex-l1l2-swdg}
  The $\lun-\ldeux$ norm associated to the partition $\Bb$ is a symmetric (finite-valued) strong gauge with
  \begin{gather*}
    \T_x = \enscond{\eta}{ \forall j \notin I, \: \eta_j = 0 }, \quad
    \S_x = \enscond{\eta}{ \forall i \in I, \: \eta_i = 0 }, \\
    \e_x = ( \Nn(x_b) )_{b \in \Bb}, \quad
    \f_x = \e_x, \quad
    \Js = \norm{\cdot}_{\infty,2} + \iota_{\S_x} ~,
  \end{gather*}
  where $I = I(x) = \enscond{b}{x_b \neq 0}$, and $\Nn(a)=a/\norm{a}$ if $a\neq 0$, and $\Nn(0)=0$.
  Moreover, it is partly smooth relative to a linear manifold with
  \begin{equation*}
    \certCof = \norm{\cdot}_{\infty,2}, \quad
    \nu_x = (1-\delta)\umin{b \in I} \norm{x_b}, \delta \in ]0,1]
    \quad
    \mu_x = \frac{\sqrt{2}}{\nu_x}
    \qandq
    \tau_x = \xi_x = 0 .
  \end{equation*}
\end{prop}

% \begin{proof}
%   The fact that $\norm{\cdot}_{\Bb}$ is a strong gauge is immediate (see the proof for $\lun$-case).
%   We have to prove the stability.
%   We immediately get that since the projection of $e_x$ onto $\S_x$ is 0, one has $\tau_x = 0$ and $\xi_x = 1$.
%   The $(\nu,\mu)$-stability came from the fact that for any vector $u,v$ where $\norm{u - v} \leq \rho \norm{u}$ for $0 < \rho < 1$, one has
%   \begin{equation*}
%     \left\| \frac{u}{\norm{u}} - \frac{v}{\norm{v}} \right\|
%     \leq
%     \frac{\sqrt{2}}{\rho}
%     \frac{\norm{u-v}}{\norm{u}} .
%   \end{equation*}
%   Indeed, \todo{Mohammad proof}.
%   Thus, for \todo{}
%   \begin{equation*}
%     \Js(e_x - e_{x'}) \leq
%     \frac{\sqrt{2}}{\nu_x} \Js(x - x') .
%   \end{equation*}
% \end{proof}
% Figure~\ref{fig:l1l2-gauge-geometry} shows the underlying geometry of the $\lun-\ldeux$ regularization in three dimensions.
% We take $J(x) = \sqrt{x_1^2 + x_2^2} + \abs{x}$.
% Note that $\partial \J(x)$ is included in the dual closed ball.
% \begin{figure}[htbp]
%   \centering
% \includegraphics[scale=0.75]{img/l1l2-gauge-geometry.pdf}
%   \caption{$\lun-\ldeux$ geometry.
%            In red, the $\lun-\ldeux$ ball.
%            In blue, the dual ball.}
%   \label{fig:l1l2-gauge-geometry}
% \end{figure}

%%%%
\subsubsection{Relation to previous work}

Theoretical guarantees of the group Lasso have been investigated by several authors under different performance criteria; see e.g.~\cite{yuan2005model,roth2008group,bach2008consistency,chesneau2008some,LiuZhang09,Wei10} to cite only a few. In particular, the author in~\cite{bach2008consistency} studies the asymptotic group selection consistency of the group Lasso in the overdetermined case, under a group irrepresentable condition. This condition also appears in noiseless identifiability in the work of \cite{candes2011simple}. The group irrepresentable condition is nothing but the specialization to the group Lasso of our condition based on $\IC(x_0)$. Indeed, using Definition~\ref{defn:ic} and Proposition~\ref{prop:ex-l1l2-swdg}, and assuming that $\Phi_{(I)}$ is full column rank (i.e. \eqref{eq:injT} is fulfilled), $\IC(x_0)$ reads
\eql{\label{eq-block-irrepres}
\IC(x_0) = \left\|\Phi_{(I^c)}^* \Phi_{(I)}^{+,*} \pa{ \tfrac{x_b}{\norm{x_b}} }_{b \in I}\right\|_{\infty,2} ~.
}

It is worth mentioning that the discrete isotropic total variation in $d$-dimension, $d \geq 2$, can be viewed as an analysis-type $\lun-\ldeux$ semi-norm. Partial smoothness and theoretical recovery guarantees with such a regularization can be retrieved from those of this paper using the results on the pre-composition rule given in Section~\ref{sec:alge-precomp}.

%%% Local Variables: 
%%% mode: latex
%%% TeX-master: "../robust-convex-regularization"
%%% End: 

\subsection{Polyhedral Regularization}
\label{sec:ex-polyhedral}

The $\lun$ and $\linf$ norms are special cases of polyhedral priors.
There are two alternative ways to define a polyhedral gauge.
The $H$-representation encodes the gauge through the hyperplanes that support the polygonal facets of its unit level set.
The $V$-representation encodes the gauge through the vertices that are the extreme points of this unit level set.
We focus here on the $H$-representation.

% \subsubsection{$H$-representation}
% \label{sec:ex-polyhedral-h}

%%%%
\subsubsection{Structure of polyhedral gauges}
 
A polyhedral gauge in the $H$-representation is defined as 
\begin{equation*}
  \J(x) = \max_{1 \leq i \leq N_H} (\dotp{x}{h_i})_+
  = \J_0(H^* x)
  \qwhereq
  \J_0(u) = \max_{1 \leq i \leq N_H} (u_i)_+, 
\end{equation*}
and we have defined $H=(h_i)_{i=1}^{N_H} \in \RR^{N \times N_H}$.

Such a polyhedral gauge can also be thought as an analysis gauge as considered in Section~\ref{sec:alge-precomp} by identifying $D=H$. One can then characterize decomposability and partial smoothness relative to a linear manifold of $\J_0$ and then invoke Proposition~\ref{prop:analysis-dec-form} and \ref{prop:gauge-analysis-stable} to derive those of $\J$. This is what we are about to do. In the following, we denote $(a^i)_{1 \leq i \leq N_{H}}$ the standard basis of $\RR^{N_H}$. 
%To avoid trivialities, we will exclude the case where the vector $x$ has only non-positive entries.

\begin{prop}\label{prop:ex-polyh-swdg}
  $\J_0(u) = \max_{1 \leq i \leq N_H} (u_i)_+$ is a (finite-valued) gauge and,
  \begin{itemize}
  \item If $u_i \leq 0$, $\forall i \in \ens{1,\cdots,N_H}$, then
  \begin{gather*}
    \S_u = \Span\pa{a^i}_{i \in I_0}, \quad
    \T_u = \Span\pa{a^i}_{i \notin I_0}, \\
    \e_u = 0, \quad
    \f_{u} = \mu \sum_{i \in I_0} a^i, \, \text{ for any } \, 0 < \mu < 1, \\
    \antig_{\f_u}(\eta) = \inf_{\tau \geq \max_{i \in I_0}\pa{-\eta_i}_+/\mu}\max\bpa{\tau\mu\abs{I_0}+\sum_{i \in I_0}\eta_i,\tau} \qforq \eta \in \S_u ~,
  \end{gather*}
  where
  \begin{equation*}
    I_0 = \enscond{i \in \ens{1,\cdots,N_H}}{ u_i = \J_0(u) = 0 } ~.
  \end{equation*}
  \item If $\exists i \in \ens{1,\cdots,N_H}$ such that $u_i > 0$, then
  \begin{gather*}
    \S_u = \enscond{\eta}{ \eta_{(I_+^c)} = 0 \qandq \dotp{\eta_{(I_+)}}{s_{(I_+)}}=0 }, \\
    \T_u = \enscond{\alpha}{ \alpha_{(I_+)} = \mu s_{(I_+)} \qforq \mu \in \RR }, \\
    e_u = \frac{s}{\abs{I_+}}, \quad
    \f_{u} = e_u, \quad
    \antig_{\f_u}(\eta) =  \umax{i \in I_+} (-\abs{I_+} \eta_i)_+ \qforq \eta \in \S_u ~,
  \end{gather*}
  where
  \begin{equation*}
    s = \sum_{i \in I_+} a^i \qandq I_{+} = \enscond{i \in \ens{1,\cdots,N_H}}{ u_i = \J_0(u) \qandq u_i > 0  } ~.
  \end{equation*}
  \end{itemize}
  Moreover, it is partly smooth relative to a linear manifold with parameters (assuming $I_+ \neq \emptyset$)
  \begin{equation*}
    \nu_u = (1-\delta)\big(\max_{i \in I_+} u_i - \umax{j \notin I_+, u_j > 0} u_j\big), \delta \in ]0,1]
    \qandq
    \mu_u = \tau_u = \xi_u = 0 .
  \end{equation*}
\end{prop}   

% Figure~\ref{fig:maxpos-gauge} shows the geometry of this regularization when $u$ is on the positive ray $\RR_+ (1,1)$ in two dimensions.
% Note that the level-set $\ens{\J_0(\cdot) \leq 1}$ to 1 is unbounded.
% \begin{figure}[htbp]
%   \centering 
%   \includegraphics[scale=1]{img/maxpos-gauge.pdf}
%   \caption{\todo{}}
%   \label{fig:maxpos-gauge}
% \end{figure}

%%%%
\subsubsection{Relation to previous works}

As stated in the case of $\linf$-norm, the work of of \cite{bach2010structured} considers robust recovery with a subclass of polyhedral norms but his notion of support is different from ours. The work~\cite{petry2012shrinkage} studies numerically some polyhedral regularizations. Again in a compressed sensing scenario, the work of~\cite{parrilo2010convex} studies a subset of polyhedral regularizations to get sharp estimates of the number of measurements for exact and $\ell^2$-stable recovery. The closest work to ours is that reported in~\cite{vaiter13polyhedral}, where theoretical recovery guarantees by polyhedral regularization were provided under similar conditions to ours and with the same notion of support as considered above. However only finite-valued coercive polyhedral gauges were considered there.

% \subsubsection{$V$-representation}
% \label{sec:ex-polyhedral-v}

% A polyhedral gauge in the $V$-representation is defined as the gauge of its unit ball, which is described as the convex hull of its set of vertices $V = (v_i)_{i=1}^{N_V} \in \RR^{N \times N_V}$ (which is supposed to contain $0$)
% \begin{equation*}
%   J(x) = \gauge_{V \Sigma}(x) = \umin{u \geq 0, x = V u} \dotp{u}{\I}.
% \end{equation*}
% \begin{equation*}
%   \qwhereq
%   \Sigma = \enscond{u \in \RR^{N_V}}{ u \geq 0 \qandq \dotp{u}{\I}=1}  
% \end{equation*}
% The support $I=I(x)$ of a vector indexes the sets of vertices $V_I = (v_i)_{i \in I}$ of the face of the polytope $V\Sigma$ containing $x/J(x)$.
% It is thus the larger index set $I$ such that $x \in J(x) V_I \Sigma_I$.
% \begin{prop}\label{prop:ex-polyv-swdg}
%   The functional $\J = \gauge_{V \Sigma}$ is stable with
%   \begin{gather*}
%     \S_x = \Ker(V_I^*), \quad
%     \T_x = \Im(V_I), \qandq
%     \e_x = V_I^{+,*} \I_I \\
%     \foralls \eta_\S \in \S, \quad \antig_\S(\eta_\S) = \max_{j \notin I} \dotp{\eta_\S+\e_x}{v_j}.
%   \end{gather*}
%   Moreover,
%   \begin{equation*}
%     \nu_x = ?
%     \qandq
%     \mu_x = \tau_x = \xi_x = 0 .
%   \end{equation*}
% \end{prop}
% \begin{proof}
%   We focus in the proof on the gauge $\gauge_{\Sigma}$.
  
%   \todo{to prove}
% \end{proof}

%%% Local Variables: 
%%% mode: latex
%%% TeX-master: "../robust-convex-regularization"
%%% End: 

% !TEX root = ../../IMAIAI-PartlySmoothLinear.tex

\subsection{A Counter-Example: the Nuclear Norm}

The nuclear norm is the natural extension of $\ell^1$ sparsity to matrix-valued data $x \in \RR^{N_0 \times N_0}$ (where $N=N_0^2$). We denote $x = V_{x} \diag(\Lambda_{x}) U_{x}^*$ an SVD decomposition of $x$, where $\Lambda_{x} \in \RR_+^{N_0}$. Note that this can be extended easily to rectangular matrices. The nuclear norm imposes such a sparsity and is defined as 
\eq{
	J(x) = \norm{x}_* = \norm{\Lambda_{x}}_1, 
}
see~\cite{2014-vaiter-ps-stability} and the reference therein. 
This norm can be shown to be partly smooth (in the sense of Definition~\ref{defn:psg}) at some $x$ with respect to the set $\Mm = \enscond{x'}{\rank(x)=\rank(x')}$ that is locally a manifold around $x$. This manifold is however not a linear space, hence one does not have $\Mm = T_x$. This shows that the nuclear norm is not in the set $\PSFL_x$ of fonctions that are partly smooth with respect to a linear manifold (in the sense of Definition~\ref{def:PRG}). In particular, Theorem~\ref{thm:local-stab} cannot be applied to this functional.

It is however possible to show that the manifold $\Mm$ associated to $x$ is stable to small noise perturbation in the observation under the same hypotheses as Theorem~\ref{thm:local-stab}. This result is proved in~\cite{2014-vaiter-ps-stability}, which extends the previous result of Bach~\cite{bach2008trace}. Note however that these proofs do not give explicit stability constants, on the contrary to Theorem~\ref{thm:local-stab}.

%%% Local Variables:
%%% mode: latex
%%% TeX-master: "../../robust-convex-regularization"
%%% End:

%\input{sections/examples/ex-nuclear.tex}
%%% Local Variables: 
%%% mode: latex
%%% TeX-master: "../robust-convex-regularization"
%%% End: 

  % !TEX root = ../IMAIAI-PartlySmoothLinear.tex

\section{Case Study: Compressed Sensing with $\linf$ Regularization}
\label{sec:cs}

% In some cases, one aims at recovering flat vectors, i.e such that for most $i$, $x_i = \normi{x}$.
% This is for instance the case in computer vision applications when performing quantization of random projections, see~\cite{jegou2012anti}.
% One can use as regularizer the $\linf$ norm defined as
% \begin{equation*}
%   \normi{x} = \umax{i \in \ens{1,\dots,N}} \abs{x_i} .
% \end{equation*}
% In this case, one has
% \begin{equation*}
%   T_x = \enscond{x'}{x'_{I} = \rho x_I \text{ for some } \rho \in \RR} ,
% \end{equation*}
% where $I = \enscond{i}{x_i = \normi{x}}$.
% This reflects that fact that $J = \normi{\cdot}$ favors signal having the same saturation pattern as $x$, see Proposition~\ref{prop:ex-linf-swdg}.

In this section, based on the generalized irrepresentable condition, we provide a bound for the sampling complexity to guarantee exact and stable recovery of the model subspace $T_{x_0}$ of antisparsity minimization from noisy Gaussian measurements.

\begin{thm}\label{thm:linf-cs}
Let $x$ be an arbitrary vector with its saturation support $I$, its model tangent subspace $\T_x=\S_x^\perp$ and model vector $\e_x$ as defined above.
Let $\beta > 1$.
For $\Phi$ drawn from the standard Gaussian ensemble with
\begin{equation*}
  Q \geq N - \abs{I} + 2\beta \abs{I} \log(\abs{I}/2) ~,
\end{equation*}
$\IC(x) < 1$ with probability at least $1-2(\abs{I}/2)^{-f(\beta,\abs{I})}$ where
\begin{equation*}
  f(\beta,\abs{I}) = 
  \left(
    \sqrt{\frac{\beta}{2\abs{I}}+\beta-1} - \sqrt{\frac{\beta}{2\abs{I}}}
  \right)^2 ~.
\end{equation*}
\end{thm}

The above bound and probability bears some similarities to what we get with $\lun$ minimization, except that now the probability of success scales in a power of $\abs{I}$ and not $N$ directly. The reason underlying such a similarity is the proof technique usual in compressed sensing-type bounds and the use of the minimal $\ldeux$-norm dual certificate. 
%In particular, a union bound is behind the $\log$ factor. If some improvements is sought after, it is on this step that it can be gained.

The map $f(\beta,\abs{I})$ is an increasing function of $\abs{I}$, so that $\lim_{\abs{I} \to \infty} f(\beta,\abs{I})=\beta-1$ and the probability of success increases with increasing size of the saturation support. But this comes at the price of a stronger requirement on the number of measurements.

For the noiseless problem \eqref{eq:reg-noiseless}, it can be shown using arguments based on the statistical dimension~\cite{amelunxen2013living} of the descent cone of the $\linf$-norm that there is a phase transition exactly at $N - \abs{I}/2$, see also~\cite[Proposition 3.12]{parrilo2010convex}. The reason is that each face of the descent cone of the hypercube at a point living on its $k$-dimensional face is the direct sum of a subspace (the linear hull of the face), and of an orthant of dimension $N-k$ (up to an isometry). The statistical dimension is then $(N-k)/2+k=(N+k)/2=N-\abs{I}/2$, observing that $k=N-\abs{I}$.

%%% Local Variables:
%%% mode: latex
%%% TeX-master: "conclusion"
%%% End:

  % !TEX root = ../IMAIAI-PartlySmoothLinear.tex

\section{Conclusion}
\label{sec:conclusion}

In this paper, we introduced the notion of partly smooth function relative to a linear manifold as a generic convex regularization framework, and presented a unified view to derive exact and robust recovery guarantees for a large class of convex regularizations. In particular, we provided sufficient conditions ensuring uniqueness of the minimizer to both \eqref{eq:reg} and \eqref{eq:reg-noiseless}, whose by-product is to guarantee exact recovery of the original object $x_0$ in the noiseless case by solving \eqref{eq:reg-noiseless}. In presence of noise, sufficient sharp conditions were given to certify exact recovery of the model subspace underlying $x_0$. As shown in the considered examples, these results encompass a variety of cases extensively studied in the literature (e.g. $\lun$, analysis $\lun$, $\lun-\ldeux$), as well as less popular ones ($\linf$, polyhedral).
We exemplified the usefulness of this analysis by providing a sampling complexity bound for exact support recovery in $\linf$ regularization from Gaussian measurements.

% Establishing necessary conditions with this generic class of regularizers for uniqueness and exact subspace recovery are an important perspective that we leave for a future research. 

%%% Local Variables: 
%%% mode: latex
%%% TeX-master: "../robust-convex-regularization"
%%% End: 

  \appendix
  % !TEX root = ../IMAIAI-PartlySmoothLinear.tex

\section{Gauges and their Polars}
\label{sec:app-gauge}

%%%%
\subsubsection{Definitions and main properties}

We start by collecting some important properties of gauges and their polars. A comprehensive account on them can be found in~\cite{rockafellar1996convex}.

\begin{lem}\label{lem:convex-equivalence}
{~\\}\vspace*{-0.5cm}
\begin{enumerate}[(i)]
\item $\gauge_C$ is a non-negative, lsc and sublinear function.
\item $C$ is the unique closed convex set containing the origin such that
\[
C = \enscond{x \in \RR^N}{\gamma_C(x) \leq 1} .
\]
\item $\gamma_C$ is finite everywhere if, and only if, $0 \in \interop C$, in which case $\gamma_C$ is continuous.
\item $\Ker \gamma_C = \{0\}$ if, and only if, $C$ is compact.
\item $\gamma_C$ is finite and coercive on $\dom \gamma_C = \Lin C$ if, and only if, $C$ is compact and $0 \in \ri C$. In particular, $\gamma_C$ is finite everywhere and coercive if, and only if, $C$ is compact and $0 \in \interop C$.
\end{enumerate}
\end{lem}
\begin{proof}
  (i)-(iii) are obtained from \cite[Theorem~V.1.2.5]{hiriart1996convex}. (iv) is obtained by combining \cite[Corollary~V.1.2.6 and Proposition~IV.3.2.5]{hiriart1996convex}. (v): the second statement follows by combining (iii)-(iv), while the first part is the second one written in $\dom \gamma_C = \Aff C = \Lin C$ since $0 \in \ri C$.
\end{proof}

Lemma~\ref{lem:convex-equivalence}(ii) is fundamental result of convex analysis that states that there is a one-to-one correspondence between gauge functions and closed convex sets containing the origin. This allows to identify sets from their gauges, and vice versa.

$\gamma_C$ is a norm, having $C$ as its unit ball, if and only if $C$ is bounded with nonempty interior and symmetric. When $C$ is only symmetric with nonempty interior, then $\gamma_C$ becomes a semi-norm.\\

Let us now turn to the polar of a convex set and a gauge. 
\begin{defn}[Polar set]\label{defn:convex-polarset}
  Let $C$ be a non-empty convex set. The set $C^\circ$ given by
  \begin{equation*}
    C^\circ = \enscond{v \in \RR^N}{\dotp{v}{x} \leq 1 \text{ for all } x \in C}
  \end{equation*}
  is called the \emph{polar} of $C$.
\end{defn}
$C^\circ$ is a closed convex set containing the origin. When the set $C$ is also closed and contains the origin, then it coincides with its bipolar, i.e. $C^{\circ\circ}=C$.
 
We are now in position to define the polar gauge.
\begin{defn}[Polar Gauge]\label{defn:polar-gauge}
The polar of a gauge $\gamma_C$ is the function $\gamma_C^\circ$ defined by
\[
	\gamma_C^\circ(u) = \inf\enscond{\mu \geq 0}{\dotp{x}{u} \leq \mu \gamma_C^\circ(x), \forall x} ~.
\] 
\end{defn}
Observe that gauges polar to each other have the property
\begin{equation*}
    \dotp{x}{u} \leq \gauge_C(x) \gauge_C^\circ(u) \quad \foralls (x,u) \in \dom \gauge_C \times \dom \gauge_C^\circ ~,
  \end{equation*}
just as dual norms satisfy a duality inequality. In fact, polar pairs of gauges correspond to the best inequalities of this type.

\begin{lem}\label{lem:convex-polar-gauge}
  Let $C \subseteq \RR^N$ be a closed convex set containing $0$.
  Then,
  \begin{enumerate}[(i)]
  \item $\gauge_C^\circ$ is a gauge function and $\gamma_C^{\circ\circ}=\gamma_C$.
  \item $\gamma_C^\circ=\gamma_{C^\circ}$, or equivalently
  \[
  C^\circ = \enscond{x \in \RR^N}{\gamma_C^\circ(x) \leq 1} = \enscond{x \in \RR^N}{\gamma_{C^\circ}(x) \leq 1}.
  \]
\item The gauge of $C$ and the support function of $C$ are mutually polar, i.e.
  \begin{equation*}
    \gauge_C = \sigma_{C^\circ} \qandq \gauge_{C^\circ} = \sigma_C ~.
  \end{equation*}
\end{enumerate}
\end{lem}
\begin{proof}
  (i) follows from \cite[Theorem~15.1]{rockafellar1996convex}. (ii) \cite[Corollary~15.1.1]{rockafellar1996convex} or \cite[Proposition~V.3.2.4]{hiriart1996convex}. (iii) \cite[Corollary~15.1.2]{rockafellar1996convex} or \cite[Proposition~V.3.2.5]{hiriart1996convex}.
\end{proof}

%%%%
\subsubsection{Gauge and polar calculus}

We here derive the expression of the gauge function of the Minkowski sum of two sets, as well as that of the image of a set by a linear operator. These results play an important role in Section~\ref{sec:stable}.
\begin{lem}\label{lem:gaugesum}
Let $C_1$ and $C_2$ be nonempty closed convex sets containing the origin. Then
\[
\gamma_{C_1+C_2}(x) = \sup_{\rho \in [0,1]} \rho\gamma_{C_1} \infc (1-\rho)\gamma_{C_2}(x) ~.
\]
If $x$ is such that $\gamma_{C_1}(x_1) + \gamma_{C_2}(x_2)$ is continuous and finite on $\enscond{(x_1,x_2)}{x_1+x_2=x}$, then
\[
\gamma_{C_1+C_2}(x) = \inf_{z \in \RR^N} \max(\gamma_{C_1}(z),\gamma_{C_2}(x-z)) ~.
\]
\end{lem}
\begin{proof}
We have from Lemma~\ref{lem:convex-polar-gauge} and calculus rules on support functions, 
\[
\gamma_{(C_1+C_2)^\circ} = \sigma_{C_1+C_2} = \sigma_{C_1}+\sigma_{C_2} ~.
\]
Thus
\[
(C_1+C_2)^\circ = \{u: \sigma_{C_1}(u)+\sigma_{C_2}(u) \leq 1\} ~.
\]
This yields that
\begin{align*}
\gamma_{C_1 + C_2}(x)	&= \sigma_{(C_1 + C_2)^\circ}(x) \\
					&= \sigma_{\sigma_{C_1}(u)+\sigma_{C_2}(u) \leq 1}(x) \\
					&= \sup_{\sigma_{C_1}(u)+\sigma_{C_2}(u) \leq 1} \dotp{u}{x} \\
					&= \sup_{\rho \in [0,1]}\sup_{\sigma_{C_1}(u) \leq \rho,\sigma_{C_2}(u) \leq 1-\rho} \dotp{u}{x} \\
					&= \sup_{\rho \in [0,1]} \sigma_{\sigma_{C_1}(u) \leq \rho} \infc \sigma_{\sigma_{C_2}(u) \leq 1-\rho}(x) \\
					&= \sup_{\rho \in [0,1]} \rho \sigma_{\sigma_{C_1}(u) \leq 1} \infc (1-\rho)\sigma_{\sigma_{C_2}(u) \leq 1}(x) \\
					&= \sup_{\rho \in [0,1]} \rho \sigma_{C_1^\circ} \infc (1-\rho)\sigma_{C_2^\circ}(x) \\
					&= \sup_{\rho \in [0,1]} \sigma_{\rho C_1^\circ} \infc \sigma_{(1-\rho)C_2^\circ}(x) \\
					&= \sup_{\rho \in [0,1]} \rho \gamma_{C_1} \infc (1-\rho)\gamma_{C_2}(x) ~,
\end{align*}
which is the first assertion.

The last identity can be rewritten
\begin{align*}
\gamma_{C_1 + C_2}(x)	&= \sup_{\rho \in [0,1]}\inf_{x_1+x_2=x} \rho \gamma_{C_1}(x_1) + (1-\rho)\gamma_{C_2}(x_2) ~.
\end{align*}
Under the assumptions of the lemma, the objective in the $\sup\inf$ is a continuous finite concave-convex function on~$[0,1] \times \enscond{(x_1,x_2)}{x_1+x_2=x}$. Since the latter sets are non-empty, closed and convex, and $[0,1]$ is obviously bounded, we have from using \cite[Corollary~37.3.2]{rockafellar1996convex}
\begin{align*}
\gamma_{C_1+C_2}(x) 	&= \inf_{z \in \RR^N}\sup_{\rho \in [0,1]} \rho\gamma_{C_1}(z) + (1-\rho)\gamma_{C_2}(x-z) \\
			&= \inf_{z \in \RR^N} \max(\gamma_{C_1}(z),\gamma_{C_2}(x-z)) ~.
\end{align*}
\end{proof}

\begin{lem}\label{lem:gaugelin}
Let $C$ be a closed convex set such that $0 \in \ri C$, and $D$ a linear operator. Then, for every $x \in \Im(D)$
\[
\gamma_{D(C)}(x) = \inf_{z \in \Ker(D)} \gamma_C(D^+ x + z) ~.
\]
\end{lem}
\begin{proof}
It is immediate to see that $D(C)$ is a closed convex set containing the origin. Moreover, we have $\Im(D^*)\cap\mathrm{dom}(\sigma_C) \neq \emptyset$, since the origin is in both of them. Thus, using \cite[Theorem~X.2.1.1]{hiriart1996convex} and Lemma~\ref{lem:convex-polar-gauge}, we have
\[
\gamma_{\pa{D(C)}^\circ} = \sigma_{D(C)} = \pa{\iota_{D(C)}}^* = \sigma_C \circ D^* ~.
\]
Now, as by assumption $0 \in \ri C$, we have $0 \in \mathrm{ri}(C^\circ)$, and therefore $\Im(D^*) \cap \mathrm{ri}(C^\circ) \neq \emptyset$. By virtue of \cite[Theorem~X.2.2.3]{hiriart1996convex} and Lemma~\ref{lem:convex-polar-gauge}, we get
\begin{align*}
\gamma_{D(C)}(x)	&= \sigma_{\pa{D(C)}^\circ}(x)\\
				&= \sigma_{\sigma_C \circ D^*(u) \leq 1}(x) \\
				&= \pa{\iota_{\sigma_C(w) \leq 1} \circ D^*}^*(x) \\
				&= \inf_{v} \sigma_{\sigma_C(w) \leq 1}(v) \quad \mathrm{s.t.} \quad D v = x \\
				&= \inf_{z \in \Ker(D)} \sigma_{\sigma_C(w) \leq 1}(D^+ x + z) \\
				&= \inf_{z \in \Ker(D)} \sigma_{\sigma_C(w) \leq 1}(D^+ x + z) \\
				&= \inf_{z \in \Ker(D)} \gamma_{C}(D^+ x + z) ~.
\end{align*}
\end{proof}
Using Lemma~\ref{lem:convex-equivalence}(v), one can observe that the infimum is finite if $(D^+ x + \Ker(D)) \cap \Lin C \neq \emptyset$.

%%%%
\subsubsection{Subdifferential of a gauge}

The subdifferential of a gauge $\gamma_C$ at a point $x$ is completely characterized by the face of its polar set $C^\circ$ exposed by $x$. Put formally, we have~\cite{hiriart1996convex}
\begin{equation*}
  \partial \gauge_C(x) = \mathbf{F}_{C^\circ}(x) = 
  \enscond{\eta \in \RR^N}{\eta \in C^\circ \qandq \dotp{\eta}{x} = \gauge_C(x)} ,
\end{equation*}
where $\mathbf{F}_{C^\circ}(x)$ is the face of $C^\circ$ exposed by $x$. The latter is the intersection of $C^\circ$ and the supporting hyperplane $\enscond{\eta \in \RR^N}{\dotp{\eta}{x} = \gauge_C(x)}$.
The special case of $x = 0$ has a much simpler structure; it is the polar set $C^\circ$ from Lemma~\ref{lem:convex-polar-gauge}(ii)-(iii), i.e.
\begin{equation*}
  \partial \gauge_C(x) =
  \enscond{\eta \in \RR^N}{\gauge_{C^\circ}(\eta) \leq 1} = C^\circ .
\end{equation*}
The following proposition gives an equivalent convenient description of the subdifferential of the regularizer $\J=\gamma_C$ at $x$ in terms of a particular supporting hyperplane to $C^\circ$: the affine hull $\bar \S_x$.
\begin{prop}\label{prop:aff-pol}
  Let $\J = \gauge_C$ be a finite-valued gauge.
  Then for $x \in \RR^N$, one has
  \begin{equation*}
    \partial \J(x) = \bar \S_x \cap C^\circ .
  \end{equation*}
\end{prop}
\begin{proof}
  Let $x \in \RR^N$. We have
  \begin{equation*}
    \partial \J(x) =
    \mathbf{F}_{C^\circ}(x) =
    H \cap C^\circ ,
  \end{equation*}
  where $H = \enscond{\eta \in \RR^N}{\dotp{\eta}{x} = \J(x)}$ is the supporting hyperplane of $C^\circ$ at $x$.
  By Proposition~\ref{prop:convex-basics-decompos}(i), we have
  \begin{equation*}
    \bar \S_x = \Aff \partial \J(x) \subseteq H,
  \end{equation*}
  which implies that
  \begin{equation*}
    \bar \S_x \cap C^\circ \subseteq H \cap C^\circ .
  \end{equation*}
  The converse inclusion is true since $\partial \J(x) = H \cap C^\circ \subseteq \bar \S_x$.
\end{proof}

\begin{prop}
\label{prop:convex-basics-decompos-gauge}
  Let $\J = \gauge_C$ be a finite-valued gauge. For any $x \in \RR^N$, one has
  \begin{enumerate}[(i)]
  \item For every $u \in \bar \S_x$, $J(x) = \dotp{u}{x}$.
  \item $x \in \T_x$.
  \item The subdifferential gauge $\antigx$ reads
  \[
  \antigx(\eta) = \inf_{\tau \geq 0} \max( \J^\circ(\tau f_x + \eta) , \tau) + \iota_{\S_x}(\eta) ~.
  \]
  \item The polar of the subdifferential gauge $\antigPx$ reads
    \begin{equation*}
      \antigPx(d)=\J(d_{\S_x}) - \dotp{f_{\S_x}}{d_{\S_x}} .
    \end{equation*}
  \end{enumerate}
\end{prop}
\begin{proof}
  \begin{enumerate}[(i)]
  \item Each element of $\bar \S_x$ can be written as $u = \sum_{i=1}^k \rho_i \eta_i$, for $k > 0$, where $\eta_i \in \partial \J(x)$ and $\sum_{i=1}^k \rho_i = 1$.
    By Fenchel identity applied to the gauge $\J$, and using Lemma~\ref{lem:convex-polar-gauge}(iii), we have
    \begin{equation*}
      \dotp{x}{\eta_i} = \J(x) + \iota_{C^\circ}(\eta_i) , \quad \forall i ~.
    \end{equation*}
    Since $\eta_i \in \partial \J(x) \subseteq C^\circ$, we get
    \begin{equation*}
      \dotp{x}{\eta_i} = \J(x), \quad \forall i ~,
    \end{equation*}
    Multiplying by $\rho_i$ and summing this identity over $i$ and using the fact that $\sum_{i=1}^k \rho_i = 1$ we obtain the desired result.

  \item For any $v \in \S_x$, we have $v+e_x \in \bar \S_x$ since $e_x \in \bar \S_x$. Thus applying (i), we get $\dotp{x}{e_x + v} = \J(x)$ and $\dotp{x}{e_x}=\J(x)$. Combining both identities implies that $\dotp{x}{v}=0$, $\forall v \in \S_x$, or equivalently that $x \in \S_x^\perp = \T_x$.

  \item Since $f_x \in \ri \partial J(x) \subset \bar\S_x$, Proposition~\ref{prop:convex-basics-decompos} implies that $f_x = \proj_{\S_x}(f_x) + \proj_{\T_x}(f_x) = \proj_{\S_x}(f_x) + e_x$. Hence, using Proposition~\ref{prop:aff-pol}, we get
    \begin{align*}
      \partial \J(x) - f_x 	&= (C^\circ - f_x) \cap (\bar \S_x - f_x) \\
    				&= (C^\circ - f_x) \cap (\S_x - \ens{\proj_{\S_x}(f_x)}) \\
				&= (C^\circ - f_x) \cap \S_x ~.
    \end{align*}
    We therefore obtain
    \begin{align*}
      \antigPx(\eta) &= \gauge_{(C^\circ - f_x) \cap \S_x}(\eta) \\
                     &= \max(\gamma_{C^\circ - f_x}(\eta),\gamma_{\S_x}(\eta)) \\
                     &= \max(\gamma_{C^\circ - f_x}(\eta),\iota_{\S_x}(\eta)) \\
                     &= \gamma_{C^\circ - f_x}(\eta) + \iota_{\S_x}(\eta) ~.
    \end{align*}  
    At this stage, Lemma~\ref{lem:gaugesum} does not apply straightforwardly since $0 \in C^\circ$ but $f_x \neq 0$ in general. However, proceeding as in the proof of that lemma, we arrive at
    \begin{align*}
      \gamma_{C^\circ + \ens{-f_x}}(\eta) 	&= \sup_{\rho \in [0,1]} \rho \J^\circ \infc (1-\rho)\sigma_{\ens{-f_x}^\circ}(\eta)
    \end{align*}
    where, from Definition~\ref{defn:convex-polarset}, $\ens{-f_x}^\circ=\enscond{\eta}{\dotp{\eta}{f_x} \geq -1}$, which indeed contains the origin as an interior point. Continuing from the last equality, we get
    \begin{align*}
      \gamma_{C^\circ + \ens{-f_x}}(\eta) 	&= \sup_{\rho \in [0,1]} \rho \J^\circ \infc (1-\rho)\gamma_{\ens{-f_x}^{\circ\circ}}(\eta) \\
                                                &= \sup_{\rho \in [0,1]} \rho \J^\circ \infc (1-\rho)\gamma_{\co{\ens{-f_x} \cup \ens{0}}}(\eta) \\
                                                &= \sup_{\rho \in [0,1]} \rho \J^\circ \infc (1-\rho)\gamma_{\enscond{-\mu f_x}{\mu \in [0,1]}}(\eta) ~.
    \end{align*}
    It is easy to see that
    \[
    \gamma_{\enscond{-\mu f_x}{\mu \in [0,1]}}(-\eta) = 
    % \begin{cases}
    %   \frac{\dotp{\eta}{f_x}}{\norm{f_x}^2} & \text{if } \eta \in \RR_+f_x ~, \\
    %   +\infty & \text{otherwise} ~.
    % \end{cases} = 
    \begin{cases}
      \tau & \text{if } \eta \in \tau f_x, \tau \in \RR_+ ~, \\
      +\infty & \text{otherwise} ~.
    \end{cases}
    \]
    Thus
    \begin{align*}
      \gamma_{C^\circ + \ens{-f_x}}(\eta) = \sup_{\rho \in [0,1]}\inf_{\tau \geq 0} \rho \J^\circ(\tau f_x + \eta) + (1-\rho)\tau ~.
    \end{align*}
    Recalling that $\J^\circ$ is a finite-valued gauge, hence continuous, the objective in the $\sup\inf$ fulfills the assumption of the second assertion of Lemma~\ref{lem:gaugesum}, whence we get
    \begin{align*}
      \gamma_{C^\circ + \ens{-f_x}}(\eta) 	&= \inf_{\tau \geq 0} \max( \J^\circ(\tau f_x + \eta), \tau) ~.
    \end{align*}

  \item Using some calculus rules with support functions and assertion (ii), we have
    \begin{align*}
      \antigPx(d) = \antigPx(d_{\S_x}) 
      &= \sigma_{(C^\circ + \ens{-f_x}) \cap \S_x}(d_{\S_x}) \\
      &= \co{\inf(\sigma_{C^\circ + \ens{-f_x}}(d_{\S_x}),\sigma_{\S}(d_{\S_x}))} \\
      &= \co{\inf(\sigma_{C^\circ + \ens{-f_x}}(d_{\S_x}),\iota_{\T}(d_{\S_x}))} \\
      &= \sigma_{C^\circ + \ens{-f_x}}(d_{\S_x}) \\
      &= \sigma_{C^\circ}(d_{\S_x}) - \dotp{\proj_{\S_x}(f_x)}{d_{\S_x}} \\
      &= \J(d_{\S_x}) - \dotp{\proj_{\S_x}(f_x)}{d_{\S_x}} ~.
    \end{align*}
  \end{enumerate}
\end{proof}

We end by showing that if a set-valued mapping is Lipschitz, then its polar is also Lipschitz continuous.
\begin{lem}\label{lem:polar-lip}
Let $F : \RR^N \rightrightarrows \RR^N$ be a $\beta$-Lipschitz set-valued mapping, such that $F(x)$ is a compact convex set containing the origin as a relative interior point for every $x \in \RR^N$. Then $F^\circ$ defined as $x \mapsto F(x)^\circ$ is $\beta$-Lipschitz. Moreover, the mapping $x \mapsto \gamma_{F(x)}(u)$, for any $u \in \dom(\gamma_{F(x)})=\Lin(F(x))$, is $\beta\norm{u}$-Lipschitz.
\end{lem}
\begin{proof}
  Using the Lipschitz continuity of $F$, we have
  \begin{equation*}
    F(x') \subseteq F(x) + \beta \norm{x' - x} \BB(0) ,
  \end{equation*}
  Using the symmetry of $\BB(0)$, we get that
  \begin{equation*}
    F(x') + \beta \norm{x' - x} \BB(0) \subseteq F(x) .
  \end{equation*}
  Since the polarity reverses the order of inclusion, we have
  \begin{equation*}
    (F(x') + \beta \norm{x' - x} \BB(0))^\circ \supseteq F(x)^\circ ,
  \end{equation*}
  or equivalently, by Lemma~\ref{lem:convex-polar-gauge},
  \begin{equation}\label{eq:lip-polar-1}
    \gauge_{F(x)}(u) = \sigma_{F(x)^\circ}(u) \leq \sigma_{(F(x') + \beta \norm{x' - x} \BB(0))^\circ}(u) = \gauge_{F(x') + \beta \norm{x' - x} \BB(0)}(u),
  \end{equation}
  Observe that these gauges are finite-valued and coercive for $u$ as prescribed (see Lemma~\ref{lem:convex-equivalence}(v)).
%  or equivalently
%  \begin{gather}\label{eq:lip-polar-1}
%    \gauge_{F(x') + \beta \norm{x' - x} \BB(0)} \geq \gauge_{F(x)} .
%  \end{gather}
  According to Lemma~\ref{lem:gaugesum}, one has
  \begin{gather*}
    \gauge_{F(x') + \beta \norm{x' - x} \BB(0)}(u)
    = \inf_{z \in \RR^N} \max(\gamma_{F(x')}(u),\gamma_{\beta \norm{x' - x} \BB(0)}(u-z)) .
  \end{gather*}
  It then follows that
  \begin{align*}
    \gauge_{F(x') + \beta \norm{x' - x} \BB(0)}(u)
    &\leq \gamma_{F(x')}(u) + \gamma_{\beta \norm{x' - x} \BB(0)}(u) \\
    &=  \gamma_{F(x')}(u) + \beta \norm{x' - x} \norm{u} .
  \end{align*}
  Thus, combining this with~\eqref{eq:lip-polar-1}, we get
  \begin{equation*}
    \abs{\gamma_{F(x)}(u) - \gamma_{F(x')}(u)} \leq \beta \norm{x' - x} \norm{u} ~.
  \end{equation*}
%  or equivalently,
%  \begin{equation*}
%    \gauge_{F(x)^\circ} \leq \gauge_{F(x')^\circ}(u) + \beta \norm{x' - x} \norm{u} ,
%  \end{equation*}
  which concludes the proof.
\end{proof}

%%% Local Variables:
%%% mode: latex
%%% TeX-master: "../robust-convex-regularization"
%%% End:

  % !TEX root = ../../IMAIAI-PartlySmoothLinear.tex

%\input{sections/proofs/proofs-polar}
% !TEX root = ../../IMAIAI-PartlySmoothLinear.tex

\section{Proofs of Section~\ref{sec:decomposable}}

\begin{proof}[Proof of Proposition~\ref{prop:convex-basics-decompos}]
{~}\\\vspace*{-0.5cm}
  \begin{enumerate}[(i)]
  \item This is due to the fact that $e_x$ is the orthogonal projection of $0$ on the affine space $\bar \S_x$. It is therefore an element of $\bar \S_x \cap (\bar \S_x - e_x)^\bot$, i.e. $\e_x \in \bar \S_x \cap \T_x$.

  \item This is straightforward from the fact that $\S_x = \enscond{\eta \in \RR^N}{\eta_{\T_x} = 0}$, $\bar \S_x = \S_x + e_x$ and $e_x \in \T_x$ from (i). 
  \end{enumerate}
\end{proof}

\begin{proof}[Proof of Proposition~\ref{prop:anti-coer}]
It follows from Lemma~\ref{lem:convex-equivalence}(v) since $0 \in \ri (\partial \J(x) - f_x)$.
\end{proof}

\begin{proof}[Proof of Proposition~\ref{prop:antig-polar}]
  The gauge $\antigPx$ is the support function of the set 
  \begin{equation*}
    \Kk_x \eqdef \partial J(x) - f_x = \enscond{\eta \in \RR^N}{\antig_{\f_x}(\eta) \leq 1} \subset \S_x ~,
  \end{equation*}
  where the inclusion follows from Proposition~\ref{prop:anti-coer}.
  \begin{enumerate}[(i)]
  \item Since $\Kk_x$ is a bounded set, its support function is finite-valued \cite[Proposition~V.2.1.3]{hiriart1996convex}. It then follows that $\dom \antigPx = \RR^N$.
  \item We have
  \begin{align*}
    \antigPx(d)
    &= \sup_{\eta \in \Kk_x} \dotp{\eta}{d} 
     = \sup_{\antigx(\eta) \leq 1} \dotp{\eta}{d} 
     = \sup_{\antigx(\eta_{\S_x}) \leq 1}
       \dotp{\eta_{S_x}}{d} \\
    &= \sup_{\eta \in \Kk_x} \dotp{\eta}{d_{\T_x}} 
       + \dotp{\eta}{d_{\S_x}} 
     = \sup_{\eta \in \Kk_x} \dotp{\eta}{d_{\S_x}} \\
    &= \antigPx(d_{S_x}) ~,
  \end{align*}
  where we used the fact that $\dotp{\eta}{d_{\T_{x}}}=0$ on $\Kk_x$.
  
  \item As a consequence of (ii), $\antigPx(d_{\T_{x}})= 0$. Clearly, $\T_{x} \subset \Ker(\antigPx)$ and $\antigPx$ is constant along all affine subspaces parallel to $\T_{x}$. But, since $0 \in \ri \Kk_x$, excluding the origin, the supremum in $\antigPx$ is always attained at some nonzero $\eta \in \Kk_x \subset \S_{x}$. Thus $\antigPx(d) > 0$ for all $d$ such that $d \notin \T_{x}$. This shows that actually $\Ker(\antigPx) = \T_{x}$. In particular, this yields that on $\S_x$, the gauge $\antigPx$ is coercive. 
%This is yet another evidence that $\Ker \antigP_{f_x} = \T_x$.
  \end{enumerate}
\end{proof}

\begin{proof}[Proof of Theorem~\ref{thm:decomp}]
  Invoking Proposition~\ref{prop:convex-basics-decompos}, we get that for every $\eta \in \partial \J(x)$, $\eta_{\T_x} = e_x$, and $\proj_{\T_x}(f_x)=e_x$.
  It remains now to uniquely characterize the part of the subdifferential lying in $\S_x$, i.e. $\partial \J(x) - e_x$.
  Since $f_x \in \ri \partial \J(x)$, we have from the one-to-one correspondence of Lemma~\ref{lem:convex-equivalence}(i) and the definition of the subdifferential gauge,
  \begin{align*}
  \eta \in \enscond{\eta \in \RR^N}{\antigx(\eta_{\S_x} - \proj_{\S_x}(f_x)) \leq 1} 
  &\iff \eta_{\S_x} - \proj_{\S_x}(f_x) \in \partial \J(x) - f_x \\
  &\iff \eta_{\S_x} \in \partial \J(x) - e_x \\
  &\iff \eta \in \partial \J(x) ~.
  \end{align*}
\end{proof}

\begin{proof}[Proof of Proposition~\ref{prop:foc}]
  This is a convenient rewriting of the fact that $x$ is a global minimizer if, and only if, $0$ is a subgradient of the objective function at $x$.
  \begin{enumerate}[(i)]
  \item For problem \eqref{eq:reg}, this is equivalent to
  \begin{equation*}
    \frac{1}{\lambda} \Phi^* (y - \Phi x) \in \partial J(x) .
  \end{equation*}
  Projecting this relation on $\T$ and $\S$ yields the desired result.
  
  \item Let's turn to problem \eqref{eq:reg-noiseless}. We have at any global minimizer $x$
  \begin{align*} 
  0 &\in \partial \J(x) + \Phi^* N_{\enscond{\alpha}{\alpha=y}}(\Phi x)
  \end{align*}
  where $N_{\enscond{\alpha}{\alpha=y}}(x)$ is the normal cone of the constraint set $\enscond{\alpha}{\alpha=y}$ at $x$, which is obviously the whole space $\RR^Q$.
  Thus, this monotone inclusion is equivalent to the existence of $\alpha \in \RR^Q$ such that
  \[
  \Phi^* \alpha \in \partial \J(x) ~.
  \]
  Projecting again this on $\T$ and $\S$ proves the assertion.
  \end{enumerate}
\end{proof}

\begin{proof}[Proof of Lemma~\ref{lem:eq-separable}]
  Let $\J=\gamma_C$, $x \in \T$ and $x' \in \S$.

  $\Rightarrow$:
  By virtue of Lemma~\ref{lem:convex-polar-gauge}, we have
  \begin{align*}
    \Js(x + x') &= \sup_{u \in C} \dotp{x + x'}{u} \\
    		&= \sup_{\J(u) \leq 1} \dotp{x + x'}{u} \\
		&= \sup{\J(u_{\T} + u_{\S}) \leq 1} \dotp{x}{u_{\T}} + \dotp{x'}{u_{\S}} \\
		&= \sup{\J(u_{\T}) + \J(u_{\S}) \leq 1} \dotp{x}{u_{\T}} + \dotp{x'}{u_{\S}} \\
		&= \sup_{\rho \in [0,1]}\sup_{\J(u_{\T}) \leq \rho,\J(u_{\S}) \leq 1-\rho} \dotp{x}{u_{\T}} + \dotp{x'}{u_{\S}} \\
		&= \sup_{\rho \in [0,1]} \rho \sup_{\J(u_{\T}) \leq 1} \dotp{x}{u_{\T}} + (1-\rho) \sup_{\J(u_{\S}) \leq 1} \dotp{x'}{u_{\S}} \\
		&= \sup_{\rho \in [0,1]} \rho \sup_{v \in C \cap \T} \dotp{x}{v} + (1-\rho) \sup_{w C \cap \T} \dotp{x'}{w} \\
		&= \sup_{\rho \in [0,1]} \rho \sigma_{C \cap \T}(x) + (1-\rho) \sigma_{C \cap \S}(x') \\
		&= \max(\sigma_{C \cap \T}(x),\sigma_{C \cap \S}(x')) ~.
  \end{align*}
  Since
  \[
  \sigma_{C \cap \T}(x) = \co{\inf(\sigma_C(x),\iota_{\S}(x))} = \sigma_C(x) = \Js(x)
  \]
  and
  \[
  \sigma_{C \cap \S}(x') = \co{\inf(\sigma_C(x'),\iota_{\T}(x'))} = \sigma_C(x') = \Js(x') ~,
  \]
  the implication follows.

  $\Leftarrow$:
  Using again Lemma~\ref{lem:convex-polar-gauge}, we get
  \begin{align*}
    \J(x + x') 	&= \sup_{u \in C^\circ} \dotp{x + x'}{u} \\
    		&= \sup_{\Js(u_{\T}+u_{\S}) \leq 1} \dotp{x}{u_{\T}} + \dotp{x'}{u_{\S}} \\
		&= \sup_{\max(\Js(u_{\T}),\Js(u_{\S})) \leq 1} \dotp{x}{u_{\T}} + \dotp{x'}{u_{\S}} \\
		&= \sup_{\Js(u_{\T}) \leq 1, \Js(u_{\S}) \leq 1} \dotp{x}{u_{\T}} + \dotp{x'}{u_{\S}} \\
		&= \sup_{v \in C^\circ \cap \T} \dotp{x}{v} + \sup_{w \in C^\circ \cap \S} \dotp{x'}{w} \\
		&= \sigma_{C^\circ \cap \T}(x) + \sigma_{C^\circ \cap \S}(x') \\
		&= \co{\inf(\sigma_{C^\circ}(x),\iota_{\S}(x))} + \co{\inf(\sigma_{C^\circ}(x'),\iota_{\T}(x'))} \\
		&= \sigma_{C^\circ}(x) + \sigma_{C^\circ}(x') \\
		&= \J(x) + \J(x') ~.
  \end{align*}
  This concludes the proof.
\end{proof}

\begin{proof}[Proof of Proposition~\ref{prop:strong-dec}]

  Let $\J=\gamma_C$. We only need to show that $J_{e_x}^{x,\circ}(\eta_{\S_x}) = \Js(\eta_{\S_x})$.
  This follows from Proposition~\ref{prop:anti-coer}, Lemma~\ref{lem:eq-separable} and Lemma~\ref{lem:convex-polar-gauge}(ii). Indeed, 
  \begin{align*}
    J_{e_x}^{x,\circ}(\eta_{\S_x}) 	&= \inf_{\tau \geq 0} \max( \Js(\tau e_x + \eta_{S_x}), \tau) & \text{from Proposition~\ref{prop:anti-coer},}\\
    				&= \inf_{\tau \geq 0} \max(\tau \Js(e_x),\Js(\eta_{S_x}),\tau) & \text{from Lemma~\ref{lem:eq-separable},} \\
				&= \inf_{\tau \geq 0} \max(\Js(\eta_{S_x}),\tau)  & \text{from $e_x \in \partial J(x) \subset C^\circ$,} \\
				&= \Js(\eta_{S_x}) ~.
  \end{align*}
\end{proof}

%%% Local Variables: 
%%% mode: latex
%%% TeX-master: "../../robust-convex-regularization"
%%% End: 

% !TEX root = ../../IMAIAI-PartlySmoothLinear.tex

\section{Proofs of Section~\ref{sec:uniqueness}}

\begin{proof}[Proof of Lemma~\ref{lem:same-image}]
  Let $x_1, x_2$ be two (global) minimizers of \eqref{eq:reg}.
  Suppose that $\Phi x^1 \neq \Phi x^2$.
  Define $x_t = t x_1 + (1-t) x_2$ for any $t \in (0,1)$.
  By strict convexity of $u \mapsto \norm{y - u}_2^2$, one has
  \begin{equation*}
    \dfrac{1}{2} \norm{y - \Phi x_t}_2^2
    <
      \frac{t}{2} \norm{y - \Phi x_1}_2^2
    + \frac{1-t}{2} \norm{y - \Phi x_2}_2^2 .
  \end{equation*}
  Since $\J$ is convex, we get
  \begin{equation*}
    \J(x_t) \leq t \J(x_1) + (1-t) \J(x_2) .
  \end{equation*}
  Combining these two inequalities contradicts the fact that $x_1,x_2$ are global minimizers of \eqref{eq:reg}.
\end{proof}

\begin{proof}[Proof of Theorem~\ref{thm:snsp}]
To prove this theorem, we need the following lemmata.
\begin{lem}\label{prop:uniqueness-directional}
  Let $C$ be a non-empty closed convex set and $f$ a proper lsc convex function.
  Let $x$ be a minimizer of $\min_{z \in C} f(z)$.
  If
  \begin{equation*}
    f'(x,z - x) > 0 \quad \forall z \in C, z \neq x ~, 
  \end{equation*}
  then, $x$ is the unique solution of $f$ on $C$.
\end{lem}
\begin{proof}
  We first show that $t \mapsto \pa{f(x + t(z - x)) - f(z)}/t$ is non-decreasing on $(0,1]$.
  Indeed, let $g : [0,1] \rightarrow \RR$ a convex function such that $g(0) = 0$.
  Let $(t,s) \in (0,1]^2$ with $s > t$.
  Then,
  \begin{align*}
    g(t)
    = g \pa{s (t/s)}
    &= g \pa{s (t/s) + (1 - t/s)0} \\
    &\leq t \frac{g(s)}{s} + (1 - t/s)g(0) \\
    &= t \frac{g(s)}{s} ~,
  \end{align*}
  which proves that $t \in (0,1] \mapsto \frac{g(t)}{t}$ is non-decreasing on $(0,1]$.
  Since $f$ is convex, applying this result shows that the function
  \begin{equation*}
    t \in (0,1] \mapsto g(t) = f(x + t(z - x)) - f(z)
  \end{equation*}
  is such that $g(0) = 0$ and $g(t)/t$ is non-decreasing.
  
  Assume now that that $f'(x,z - x) > 0$.
  Then, for every $x \in C$,
  \begin{equation*}
    g(1) = f(z) - f(x) \geq f'(x,z - x) > 0, \quad \forall z \in C, z \neq x ~,
  \end{equation*}
  which is equivalent to $x$ being the unique minimizer of $f$ on $C$.
\end{proof}

We now compute the directional derivative of a finite-valued convex function $\J$.
\begin{lem}\label{lem:directional-gauge}
  The directional derivative $J'(x,\delta)$ at point $x \in \RR^N$ in the direction $\delta$ reads
  \begin{equation*}
    J'(x,\delta) = \dotp{e_x}{\delta_{\T_x}} 
                   + \dotp{P_{\S_x}(\f_x)}{\delta_{\S_x}}
                   + \antigPx(\delta_{\S_x}).
  \end{equation*}
\end{lem}
\begin{proof}
  This comes directly from the structure of $\antigPx$.
  Indeed, one has
  \begin{align*}
    \antigPx(\delta_{\S_x})
    &= \antigPx(\delta)
    & \text{Using Proposition~\ref{prop:antig-polar}(ii)} \\
    &=  \sup_{\eta \in \partial \J(x) - \ens{f_x}} \dotp{\eta}{\delta} \\
    &= -\dotp{\delta}{f_x} + \sup_{\eta \in \partial \J(x)} \dotp{\eta}{d} \\
    &= -\dotp{\delta}{f_x} + \J'(x,\delta) \\
    &= -\dotp{e_x}{\delta_{\T_x}} - \dotp{P_{\S_x}(\f_x)}{\delta_{\S_x}} + \J'(x,\delta) ~.
  \end{align*}
\end{proof}

We are now in position to show Theorem~\ref{thm:focu}. We provide the proof for \eqref{eq:reg}. That of \eqref{eq:reg-noiseless} is similar.

  Let $x$ be a solution of \eqref{eq:reg}.
  According to Lemma~\ref{lem:same-image}, the set of minimizers of \eqref{eq:reg} reads $\mathbf{M} \subseteq x + \Ker(\Phi)$, which is a closed convex set. We can therefore rewrite \eqref{eq:reg} as
  \begin{equation*}
    \min_{z \in \mathbf{M}} J(z) .
  \end{equation*}
  Invoking Lemma~\ref{prop:uniqueness-directional} with $C=\mathbf{M}$, $x$ is thus the unique minimizer if
  \begin{equation*}
    \forall \delta \in \Ker(\Phi) \setminus \ens{0},
    \quad
    \J'(x, \delta) > 0.
  \end{equation*}
  Using Lemma~\ref{lem:directional-gauge} and the fact that $\Ker(\Phi)$ is a subspace, this is equivalent to
  \begin{equation*}
    \forall \delta \in \Ker(\Phi) \setminus \ens{0},
    \quad
    \dotp{e_x}{\delta_{\T}} + \dotp{P_{\S}(\f_{x})}{\delta_{\S}}
    < \antigPx(- \delta_{\S}).
  \end{equation*}
  which is \eqref{eq:nsps}.
\end{proof}

\begin{proof}[Proof of Corollary~\ref{prop:uniqueness-topo}]
  Using~\cite[Theorem V.2.2.3]{hiriart1996convex}, we know that
  \begin{equation*}
    \eta \in \ri(\partial \J(x))
    \Leftrightarrow
    \J'(x,\delta) > \dotp{\eta}{\delta} \quad \forall \delta  \text{ such that } \J'(x,\delta) + \J'(x,-\delta) > 0 .
  \end{equation*}
  Applying this with $\eta = \Phi^* \alpha \in \ri(\partial \J(x))$, and using Lemma~\ref{lem:directional-gauge}, we obtain
  \begin{equation*}
    \Phi^* \alpha \in \ri(\partial \J(x))
    \Leftrightarrow
    \J'(x,\delta) > \dotp{\alpha}{\Phi \delta} \quad \forall \delta  \text{ such that } \antigPx(\delta) + \antigPx(-\delta) > 0 .    
  \end{equation*}
  Moreover, since $\antigPx$ and $\Ker(\antigPx) = \T_x = \T$ from Proposition~\ref{prop:antig-polar}(iii), and \eqref{eq:injT} holds, we get
  \begin{align*}
    \Phi^* \alpha \in \ri(\partial \J(x))
    &\Leftrightarrow
    \J'(x,\delta) > \dotp{\alpha}{\Phi \delta} \quad \forall \delta \notin \T \\
    &\Rightarrow
    \J'(x,\delta) > 0 \quad \forall \delta \in \Ker(\Phi) .
  \end{align*}
  We conclude using Theorem~\ref{thm:snsp}.
\end{proof}

\begin{proof}[Proof of Theorem~\ref{thm:focu}]
{~\\}\vspace*{-0.5cm}
  \begin{enumerate}[(i)]
  \item Let the dual vector $\alpha=(y-\Phi x)/\lambda$, and $\eta=\Phi^*\alpha \in \partial \J(x)$ by Theorem~\ref{thm:decomp}(i). We then observe that 
  \begin{align*}
  \eta \in \enscond{\eta \in \RR^N}{\antig_{f_x}(\eta_{\S} - \proj_{\S}(f_x)) < 1} 
  &\iff \eta_{\S} - \proj_{\S}(f_x) \in \ri (\partial \J(x) - \ens{f_x}) \\
  &\iff \eta \in \ri (\partial \J(x)) ~.
  \end{align*}
  Thus, applying Corollary~\ref{prop:uniqueness-topo} with such a dual vector yields the assertion.
  
  \item The proof is similar to (i) except that we invoke Theorem~\ref{thm:decomp}(ii).
  \end{enumerate}
\end{proof}

%%% Local Variables: 
%%% mode: latex
%%% TeX-master: "../../robust-convex-regularization"
%%% End: 

% !TEX root = ../../IMAIAI-PartlySmoothLinear.tex

\section{Proofs of Section~\ref{sec:stable}}

\begin{proof}[Proof of Theorem~\ref{thm:lip-prg}]
  Without loss of generality, we show this result for $\Gamma = \norm{\cdot}$ since for every $x \in \RR^N$,
  \begin{equation*}
    \Gamma(x) \leq \normOP{\Id}{\Gamma}{\ldeux} \norm{x} .
  \end{equation*}

  Recall that $J$ is partly smooth at $x$ relative to $T_x$, and $\partial J : \RR^N \rightrightarrows \RR^N$ is Lipschitz-continuous around $x$ relative to $T_x$.
  \begin{itemize}
  \item \emph{Existence of $f_x$.}
    Such a mapping exists according to~\cite[Theorem 9.4.3]{aubin2009set}.
  \item \emph{$\nu$-stability.}
    Using~\cite[Proposition 2.10]{lewis2002active} the sharpness property Definition~\ref{defn:psg}(ii) is locally stable.
    Hence, for $x' \in T_x$ in a neighborhood of $x$, $\tgtManif{x'}{T_x} = T_x = T_{x'}$.
    The radius of this neighborhood can be taken as $\nu_x$.   
  \item \emph{$\mu$-stability.}
    Using~\cite[Corollary VI.2.1.3]{hiriart1996convex}, we write for any $h \in T_x$
    \begin{equation*}
      J(x + th) = J(x) + t \dotp{s}{h} + o(t) = J(x) + t \dotp{e_x}{h} + o(t),
    \end{equation*}
    where $s \in \mathbf{F}_{\partial J(x)}(h)$.
  Since $J$ restricted to $T_x \cap U$ is $\Cdeux$ according to the smoothness property, repeating this argument at order 2 allows to conclude that the mapping $z \in T_x \cap U \mapsto e_z$ is $\Calt{1}$, when local Lipschitz continuity follows immediately.  
  \item \emph{$\tau$-stability.}
    One has
    \begin{equation*}
      \antigx(\proj_\S(f_x - f_{x'})) \leq
      \normOP{\proj_{S_x}}{\antigx}{\ldeux} \norm{f_x - f_{x'}} \leq \tau_x \norm{x-x'},
    \end{equation*}
    where $\tau_x = \normOP{\proj_{S_x}}{\antigx}{\ldeux} \beta$ and $\beta$ is the Lipschitz constant associated to $f_x$, proving~\eqref{eq:lip-tau}.
  \item \emph{$\xi$-stability.}
    $\partial J$ is Lipschitz-continous around $x$ relative to $T_x$, and $x \mapsto f_x$ is $\beta$-Lipschitz-continuous.
    Hence, the mapping $x \mapsto (\partial J(x) - f_x)$ is Lipschitz on $T_x$.
    In view of Lemma~\ref{lem:polar-lip}, we get that
    \begin{equation*}
      \antigxp(u) - \antigx(u) \leq \beta \norm{x' - x} \norm{u} .
    \end{equation*}
    Since $\norm{u} \leq \normOP{\Id}{\ldeux}{\antigx} \antigx(u)$, we get the desired bound where $\xi_x = \beta \normOP{\Id}{\ldeux}{\antigx}$.
  \end{itemize}
\end{proof}

\begin{proof}[Proof of Proposition~\ref{prop:sum-dec}]
{~}\\\vspace*{-0.5cm}
\begin{enumerate}[(i)]
\item First, we have 
\[
\partial H(x) = \partial J(x) +   \partial G(x),
\]
Let $\SsJ = \Span(\partial J(x) -\eta^J)$ and $\SsG = \Span(\partial G(x) -\eta^G)$, for any pair $\eta^J\in \partial J(x)$ and $\eta^G \in \partial G(x)$. Choosing $\eta^H = \eta^J +  \eta^G \in \partial H(x)$ we have 
\begin{eqnarray*}
\SsH &=& \Span(\partial H(x) - \eta^H)\\
&=& \Span\left( (\partial J(x)-\eta^J)  \Minksum (\partial G(x)-\eta^G) \right) \\
&=& \Span\left( \Span(\partial J(x)-\eta^J)  \Minksum \Span(\partial G(x)-\eta^G) \right) \\
&=& \Span(\SsJ \cup \SsG).
\end{eqnarray*}
As a consequence we have $\TtH = (\SsH)^\bot= \TtJ \cap \TtG$. 

\item Moreover, since $\TtH \bot \, \SsJ \cup \SsG$ we have from Proposition~\ref{prop:convex-basics-decompos}(iii) that
\begin{eqnarray*}
e_H = \proj_{\TtH}(\partial H(x)) &=& \proj_{\TtH}(\partial J(x) \Minksum   \partial G(x))\\
&=& \proj_{\TtH} \pa{e_J+ \proj_{\SsJ} \partial J(x)+ e_G+  \proj_{\SsG} \partial G(x)}\\
&=& \proj_{\TtH} (e_J+e_G). \label{eq:eH}
\end{eqnarray*}

\item As $f_x^J \in \ri \partial J(x)$ and $f_x^G \in \ri \partial G(x)$, it follows from \cite[Corollary~6.6.2]{rockafellar1996convex} that
\[
f_x^H= f_x^J+f_x^G \in \ri \partial J(x) + \ri \partial G(x) = \ri \pa{\partial J(x) + \partial G(x)} = \ri \partial H(x) ~.
\]
The subdifferential gauge associated to $H$ is then
\[
\AntiG{H} = \gamma_{\partial H(x) - f_x^H} = \gamma_{\pa{\partial J(x) - f_x^J} + \pa{\partial G(x) - f_x^G}} ~,
\]
which is coercive and finite on $\SsH$ according to Proposition~\ref{prop:anti-coer}. Invoking Lemma~\ref{lem:gaugesum}, we get the desired result since for any $\rho \geq 0$,
\[
u \mapsto \rho\AntiG{J}(u) + (1-\rho) \AntiG{G}(\eta-u)= \rho\gamma_{\partial J(x) - f_x^J}(u) + (1-\rho)\gamma_{\partial G(x) - f_x^G}(\eta-u)
\] 
is finite and continuous on $\SsJ \cap (\SsG+\eta)$, for $\eta \in \SsH=\Span(\SsJ+\SsG)$ by (i).
\end{enumerate}
\end{proof}

\begin{proof}[Proof of Proposition~\ref{prop:sum-prg}]
In the following, all operator bounds that appear are finite owing to the coercivity assumption on the involved gauges in Definition~\ref{def:PRG} of a partly smooth regularizer.\\

It is straightforward to see that the function $\certCof^H=\max(\certCof^J,\certCof^G)$ is indeed a gauge, which is finite and coercive on $\TtH=\TtJ \cap \TtG$. Moreover, given that both $J$ and $G$ are partly smooth relative to a linear manifold at $x$ with corresponding parameters $\nu_x^J$ and $\nu_x^G$, we have with the advocated choice of $\certCof^H$ and $\nu_x^H$,
\[
\certCof^J(x - x') \leq \nu_x^J \qandq \certCof^G(x - x') \leq \nu_x^G,
\]
for every $\forall x' \in \TtH_x$ such that $\certCof^H(x - x') \leq \nu_x^H$. It follows that:
\begin{itemize}
\item Since $J$ and $G$ are both partly smmoth relative to a linear manifold, then we have $\TtJ_x=\TtJ_{x'}$ and $\TtG_x=\TtG_{x'}$, and thus by Proposition~\ref{prop:sum-dec}(i)
\[
\TtH_x = \TtJ_x \cap \TtG_x = \TtJ_{x'} \cap \TtG_{x'} = \TtH_{x'} = \TtH.
\]
\item {\bf $\mu^H_x$-stability:} we have from Proposition~\ref{prop:sum-dec}(ii)
\begin{align*}
\certCof^H(e^H_x - e^H_{x'}) &= \certCof^H\pa{ \proj_{\TtH}( e^J_x + e^G_x - e^J_{x'} - e^G_{x'} ) } \\
&\leq \certCof^H\pa{ \proj_{\TtH}( e^J_x - e^J_{x'})} + \certCof^H\pa{ \proj_{\TtH}(e^G_x  - e^G_{x'} ) } \\
& \leq \normOP{\proj_{\TtH}} {\certCof^J} {\certCof^H} \,   \certCof^J \pa{  e^J_x - e^J_{x'}}
+ \normOP{\proj_{\TtH}} {\certCof^G} {\certCof^H} \,   \certCof^G \pa{ e^G_x - e^G_{x'}} \\
&\leq \pa{\mu_x^J \, \normOP{\proj_{\TtH}}{\certCof^J}  {\certCof^H}+ \mu_x^G\, \normOP{\proj_{\TtH}}{\certCof^G} {\certCof^H}} \certCof^H(x-x') ~,
\end{align*}
where we used $\mu^J_x$- and $\mu^G_x$-stability of $J$ and $G$ in the last inequality.

\item {\bf $\tau^H_x$-stability:} the fact that $\SsJ \subseteq \SsH$ and $\SsG \subseteq \SsH$ and subadditivity of gauges lead to
\begin{align}\label{eq:sumstability_f}
&\AntiG{H} \pa{ \proj_{\SsH}(f^H_x -f^H_{x'})} \nonumber\\
 &= \AntiG{H} \pa{  \proj_{\SsJ}(f^J_x - f^J_{x'}) + \proj_{\SsG}(f^G_x - f^G_{x'}) + \proj_{\SsH}(e^J_x - e^J_{x'}) + \proj_{\SsH}(e^G_x - e^G_{x'}) } \nonumber \\
&\leq \AntiG{H} \pa{  \proj_{\SsJ}(f^J_x - f^J_{x'})} + \AntiG{H} \pa{ \proj_{\SsG}(f^G_x - f^G_{x'})} \nonumber\\
&\quad + \AntiG{H} \pa{ \proj_{\SsH}(e^J_x - e^J_{x'})} + \AntiG{H} \pa{ \proj_{\SsH}(e^G_x - e^G_{x'}) } ~. 
\end{align}
According to Proposition~\ref{prop:sum-dec}(iii), we have
\[
\AntiG{H} \pa{  \proj_{\SsJ}(f^J_x - f^J_{x'})} = \inf_{\eta_1+\eta_2=\proj_{\SsJ}(f^J_x - f^J_{x'})} \max(\AntiG{J}(\eta_1),\AntiG{G}(\eta_2)) ~.
\]
Since $\dom \AntiG{J}=\SsJ$, $(\eta_1,\eta_2)=(\proj_{\SsJ}(f^J_x - f^J_{x'}),0)$ is a feasible point of the last problem, and we get
\[
\AntiG{H} \pa{  \proj_{\SsJ}(f^J_x - f^J_{x'})} \leq \AntiG{J} \pa{\proj_{\SsJ}(f^J_x - f^J_{x'})}.
\]
Moreover, as $e^J_x ,e^J_{x'} \in \TtJ$ (see Proposition~\ref{prop:convex-basics-decompos}(ii)) and $\SsJ \subseteq \SsH$, we have
\begin{align*}
&  \min_{\eta_1 \in \TtJ, \eta_2 \SsJ, \eta_1+\eta_2 \in \SsH} \norm{\eta_1+\eta_2-(e^J_x - e^J_{x'})}^2 \\
&= \min_{\eta_1 \in \TtJ, \eta_2 \SsJ, \eta_1+\eta_2 \in \SsH} \norm{\eta_1-(e^J_x - e^J_{x'})}^2+\norm{\eta_2}^2 \\
&= \min_{\eta_1 \in \TtJ, \eta_2 \SsJ, \eta_1 \in \SsH} \norm{\eta_1-(e^J_x - e^J_{x'})}^2+\norm{\eta_2}^2 \\
&= \min_{\eta_1 \in \SsH \cap \TtJ} \norm{\eta_1-(e^J_x - e^J_{x'})}^2 ~.
\end{align*}
That is
\[
\proj_{\SsH}(e^J_x - e^J_{x'})  = \proj_{\SsH \cap \TtJ }(e^J_x - e^J_{x'}) ~.
\]
Thus
\[
\AntiG{H} \pa{  \proj_{\SsH}(e^J_x - e^J_{x'})} \leq \normOP{\proj_{\SsH \cap \TtJ}}{\certCof^J}{\AntiG{H}} \certCof^J \pa{e^J_x - e^J_{x'}} ~.
\]

%where the subspace difference is defined as $\SsH \cap \TtJ \eqdef \SsG \cap (\SsJ)^\bot =  \SsG \cap \TtJ$.
% Invoking again Proposition~\ref{prop:sum-dec}(iii), we get
% \[
% \AntiG{H} \pa{ \proj_{\SsH}(e^J_x - e^J_{x'})} \leq \AntiG{G} \pa{  \proj_{\SsH \cap \TtJ }(e^J_x - e^J_{x'})},
% \]
% where we use that $\dom \AntiG{G}=\SsG \subseteq \SsH$, and thus $(\eta_1,\eta_2)=(0,\proj_{\SsH \cap \TtJ }(e^J_x - e^J_{x'}))$ is a feasible point in this case.

Similar reasoning leads to the following bounds 
\begin{align*}
\AntiG{H} \pa{ \proj_{\SsG}(f^G_x - f^G_{x'})} &\leq \AntiG{G} \pa{\proj_{\SsG}(f^G_x - f^G_{x'})}, \\
\AntiG{H} \pa{ \proj_{\SsH}(e^G_x - e^G_{x'})} &\leq \normOP{\proj_{\SsH \cap \TtG}}{\certCof^J}{\AntiG{H}} \certCof^G \pa{e^G_x - e^G_{x'}} ~.
\end{align*}
Having this, we can continue to bound~\eqref{eq:sumstability_f} as
\begin{align*}
&\AntiG{H} \pa{ \proj_{\SsH}(f^H_x -f^H_{x'})} \\
&\leq \AntiG{J} \pa{  \proj_{\SsJ}(f^J_x - f^J_{x'})} + \AntiG{G} \pa{  \proj_{\SsG}(f^G_x - f^G_{x'})}  \\
&\quad + \normOP{\proj_{\SsH \cap \TtJ}}{\certCof^J}{\AntiG{H}} \certCof^J \pa{e^J_x - e^J_{x'}} + \normOP{\proj_{\SsH \cap \TtG}}{\certCof^J}{\AntiG{H}} \certCof^G \pa{e^G_x - e^G_{x'}} \\
&\leq \tau^J_x \, \certCof^J(x-x') + \tau^G_x \, \certCof^G(x-x') + \mu^J_x \, \normOP{\proj_{\SsH \cap \TtJ}}{\certCof^J}{\AntiG{H}} \certCof^J \pa{x - x'} \\ 
& \quad + \mu^G_x \, \normOP{\proj_{\SsH \cap \TtG}}{\certCof^G}{\AntiG{H}} \certCof^G\pa{x - x'} \\
&\leq \pa{\tau^J_x+\tau^G_x+ \mu^J_x \,  \normOP{\proj_{\SsH \cap \TtJ}}{\certCof^J}{\AntiG{H}} + \mu^G_x \,  \normOP{\proj_{\SsH \cap \TtG}}{\certCof^G}{\AntiG{H}} } \certCof^H(x-x') ~,
\end{align*}
where the last two inequalities $J$ and $G$ follow from $\mu^J_x$-, $\tau^J_x$-, $\mu^G_x$- and $\tau^G_x$- stability of $J$ and $G$.

\item {\bf $\xi^H_x$-stability:} Proposition~\ref{prop:sum-dec}(iii) again yields that for any $\eta \in \SsH$
\begin{align*}
H^\circ_{f^H_{x'}}(\eta) &= 	\inf_{\eta_1+\eta_2=\eta} \max(J^\circ_{f^J_{x'}}(\eta_1),G^\circ_{f^G_{x'}}(\eta_2)) \\
			 &\leq  \max(J^\circ_{f^J_{x'}}(\bar{\eta}_1),G^\circ_{f^G_{x'}}(\bar{\eta}_2))
\end{align*}
for any feasible $(\bar{\eta}_1,\bar{\eta}_2) \in \SsJ \times \SsG \cap \enscond{(\eta_1,\eta_2}{\eta_1+\eta_2=\eta}$.
Now both $J$ and $G$ are partly smooth relative to a linear manifold, hence respectively $\xi^J_x$- and $\xi^G_x$-stable. Therefore, with the form of $\certCof^H$ we have 
\begin{align*}
J^\circ_{f^J_{x'}}(\bar{\eta}_1)  &\leq (1+\xi^J_x \certCof^J(x-x'))\AntiG{J}(\bar{\eta}_1) \leq \beta \AntiG{J}(\bar{\eta}_1) \\
G^\circ_{f^G_{x'}}(\bar{\eta}_2)  &\leq (1+\xi^G_x \certCof^G(x-x'))\AntiG{G}(\bar{\eta}_2) \leq \beta \AntiG{G}(\bar{\eta}_2) ~,
\end{align*}
where $\beta = 1+ \max\pa{\xi^J_x, \xi^G_x} \certCof^H(x-x')$. Whence we get
\[
\max(J^\circ_{f^J_{x'}}(\eta_1),G^\circ_{f^G_{x'}}(\eta_2)) \leq \beta \max(\AntiG{J}(\bar{\eta}_1),\AntiG{G}(\bar{\eta}_2)) ~.
\]
Taking in particular 
\[
(\bar{\eta}_1,\bar{\eta}_2) \in \uArgmin{\eta_1+\eta_2=\eta} \max(\AntiG{J}(\eta_1),\AntiG{G}(\eta_2))
\]
we arrive at
\[
H^\circ_{f^H_{x'}}(\eta) \leq \beta \inf_{\eta_1+\eta_2=\eta} \max(J^\circ_{f^J_{x}}(\eta_1),G^\circ_{f^G_{x}}(\eta_2)) = \beta H^\circ_{f^H_{x}}(\eta) ~.
\]

% Similarly as for~\eqref{eq:gaugeHmin} we have the following implicit formulation for $H^\circ_{f^H_{x'}}(u_H)$:
% \begin{align}
% H^\circ_{f^H_{x'}}(u_H) =&\umin{\rho' \in \mathcal{C}'_{u_H}} \rho' ,
% \label{eq:gaugeHmin2}
% \end{align}
% where
% \[\mathcal{C}'_{u_H}:=
%  \big\{ \rho' \geq 0:  \exists u'_J\in \SsJ, \exists u'_G\in \SsG, \,\, %\\
% u_H = \rho'(u'_J+u'_G),\,\,% \nonumber\\
%  J^\circ_{f^J_{x'}}(u'_J)\leq1,\,\, %\nonumber\\
% G^\circ_{f^G_{x'}}(u'_G)\leq1 
% \big\}.
% \]
% Now since both $J$ and $G$ are PRG and thus $\xi^J_x,\xi^G_x$-stable, we have
% \eq{ J^\circ_{f^J_{x'}}(u'_J) \leq \beta \AntiG{J}(u'_J)= \AntiG{J}(\beta u'_J) \qandq
% G^\circ_{f^G_{x'}}(u'_G) \leq \beta \AntiG{G}(u'_G)= \AntiG{G}(\beta u'_G),
% }
% where $\beta = 1+ \max\left\{ \xi^J_x, \xi^G_x\right\} \certCof^H(x-x')$. As a result and by the following change of variables 
% \[\beta \rho = \rho', \quad u_J=\beta u'_J, \quad u_G=\beta u'_G,\]
% we can deduce that ${C}_{u_H}\subseteq {C}'_{u_H}$, where
% \[
% \mathcal{C}_{u_H}:=
%  \big\{ \rho \geq 0: \exists u_J\in \SsJ, \exists u_G\in \SsG,  \,\,%\\
% u_H = \rho(u_J+u_G),\,\,% \nonumber\\
%  J^\circ_{f^J_{x}}(u_J)\leq1,\,\, %\nonumber\\
% G^\circ_{f^G_{x}}(u_G)\leq1 
% \big\},
% \]
% and thus by considering the implicit formulation~\eqref{eq:gaugeHmin} for  $\AntiG{H}(u_H)$ we have
% \begin{align}
% H^\circ_{f^H_{x'}}(u_H) \leq \beta \umin{\rho \in {C}_{u_H}} \rho = \beta \, \AntiG{H}(u_H).
% \end{align}

\end{itemize}
This completes the proof.

\end{proof}

\begin{proof}[Proof of Corollary~\ref{cor:sum-smooth-sum}]
  G\^ateaux-differentiability entails that $\partial G(x) = \ens{\nabla G(x)}$, whence we obtain $T_x^{G} = \RR^N$ and $e_x^{G} = \nabla G(x)$.
  Applying Proposition~\ref{prop:sum-dec}, we get the result.
  It is sufficient to remark that the smooth perturbation $G$ translates the subdifferential $\partial J(x)$ by $\nabla G(x)$.
  Hence, using our choice of $f_x^{J+G}$, we find the same subdifferential gauge.
\end{proof}

\begin{proof}[Proof of Corollary~\ref{cor:sum-smooth-prg}]
  Since $G$ si $\Cdeux$ on $\RR^N$, it is obviously partly smooth relative to $T_x^G=\RR^N$ according to~\cite[Example 3.1]{lewis2002active}.
  We now exhibit the constants involved.
  
  \begin{itemize}
  \item \textbf{$\nu$-stability.}
  For every $x' \in \RR^N$, $x' \in T_x^G$, and thus $\nu_x^G = + \infty$, implying that $\nu_x^H = \nu_x^J$.
  
  \item \textbf{$\mu$-stability.}
  Using the $\mu$-stability of $J$ and the fact that $\nabla G$ is $\beta$-Lipschitz, we get that
  \begin{equation*}
    \mu_x^H = \mu_x^J \, \normOP{\proj_{\TtJ}}{\certCof^J}{\certCof^H} + \beta \, \normOP{\proj_{\TtJ}}{\ldeux} {\certCof^H} .
  \end{equation*}

  \item \textbf{$\tau$- and $\xi$-stability.}
  Since $S = \ens{0}$, $\tau_x^G = \xi_x^G = 0$, and we get from Proposition~\ref{prop:sum-prg}
  \begin{align*}
    \tau_x^H = \tau^J_x \qandq
    \xi_x^H = \xi_x^J .    
  \end{align*}
  \end{itemize}
\end{proof}

\begin{proof}[Proof of Proposition~\ref{prop:analysis-dec-form}]
  {~}\\\vspace*{-0.5cm}
  \begin{enumerate}[(i)] 
  \item One has $\partial J = D \circ \partial \J_0 \circ D^*$, hence $\S=D \Sz = \Im( D_{\Sz} )$ and $\T=\S^\perp=\Ker(D_{\Sz}^*)$.
  \item As $\S=D \bar\Sz=D \ez + \S$, we get rom Proposition~\ref{prop:convex-basics-decompos}
  \begin{align*}
  e \in \uargmin{z \in \bar\S} \norm{z} = \uargmin{z-D \ez \in \S} \norm{z} 	&= D \ez + \uargmin{h \in \S} \norm{h+D \ez} \\
  										&= D \ez + \proj_\S(-D \ez) = (\Id - \proj_\S) D \ez = \proj_\T D \ez ~.
  \end{align*}
  \item With such a choice of $f_x$, we have
  \[
  f_{0,D^*x} \in \ri \partial \J_0(D^* x) \Rightarrow D f_{0,D^*x} \in D \ri \partial \J_0(D^* x) 
  \]
  \[
  	\iff \f_x \in \ri D\partial \J_0(D^* x)  
  	\iff \f_x \in \ri \partial \J(x)  ~.
  \]
   We follow the same lines as in the proof of Lemma~\ref{lem:gaugelin}, where we additionally invoke Proposition~\ref{prop:antig-polar}(ii) to get
\begin{align*}
\antigPx(d)	&= \sigma_{\partial \J(x) - f_x}(d) \\
		&= \sigma_{D\pa{\partial \J_0(D^* x) - f_{0,D^*x}}}(d) \\
		&= \sigma_{\partial \J_0(D^* x) - f_{0,D^*x}}(D^*d) \\
		&= \antigPxa(D^*d) \\
		&= \antigPxa(D_{\S_0}^*d) ~.
\end{align*}
Note that $\antigPx$ is indeed constant along affine subspaces parallel to $\ker(D_{\S_0}^*)=\S^\perp=\T$.
We now get that for every $\eta \in \S=\ker(D_{\S_0}^+)^\perp$
\begin{align*}
\antigPx(\eta)	&= \sigma_{\antigPx(d) \leq 1}(\eta) \\
			&= \sigma_{\antigPxa(D_{\S_0}^*d) \leq 1}(\eta) \\
			&= \pa{\iota_{\antigPxa(w) \leq 1} \circ D_{\S_0}^*}^*(\eta) \\
			&= \inf_{v} \sigma_{\antigPxa(w) \leq 1}(v) \quad \mathrm{s.t.} \quad D_{\S_0} v = \eta \\
			&= \inf_{z \in \ker(D_{\S_0})} \antigxa(D_{\S_0}^+ \eta + z) ~.
\end{align*}
The infimum is finite and is attained necessarily at some $z \in \ker(D_{\S_0}) \cap \S_0 \neq \emptyset$ since $\dom \antigxa = \S_0$ and $\Im(D_{\S_0}^+)=\Im(D_{\S_0}^*) \subset \S_0$. Moreover, $\ker(D_{\S_0}) \cap \S_0 = \ker(D) \cap \S_0$.
\end{enumerate}
\end{proof}

\begin{proof}[Proof of Proposition~\ref{prop:gauge-analysis-stable}]
In the following, all operator bounds that appear are finite owing to the coercivity assumption on the involved gauges in Definition~\ref{def:PRG} of a partly smooth regularizer.
\begin{itemize}
\item Let $x'$ such that
  \begin{equation*}
    \certCo{x - x'}
    \leq
    \frac{1}{\normOP{D^*}{\Gamma}{\Gamma_0}}\nu_{0,D^* x}.
  \end{equation*}
  Hence,
  \begin{equation*}
    \Gamma_0(D^* x - D^* x')
    \leq
    \normOP{D^*}{\Gamma}{\Gamma_0}
    \Gamma(x - x')
    \leq
    \nu_{0,D^* x}
  \end{equation*}
  As $\J_0$ is a partly smooth relative to a linear manifold at $D^* x$, we have $\T_{0,D^*x} = \T_{0,D^*x'}=\Tz$ and consequently, using Proposition~\ref{prop:analysis-dec-form}(i), $\T_x = \Ker(D_{\S_{0,D^*x}}^*) = \Ker(D_{\S_{0,D^*x'}}^*) = \T_{x'} = \T = \S^\perp$.

\item {\bf $\mu_{x}$-stability:} we now have
  \begin{align*}
    \Gamma(e_x - e_x')
    &=
    \Gamma(\proj_\T D(e_{0,D^* x} - e_{0,D^* x'})) 
      & \text{Proposition~\ref{prop:analysis-dec-form}(ii)}\\
    &\leq
    \normOP{\proj_\T D}{\Gamma_0}{\Gamma}
    \Gamma_0(e_{0,D^* x} - e_{0,D^* x'}) \\
    &\leq
    \mu_{0,D^* x}
    \normOP{\proj_\T D}{\Gamma_0}{\Gamma}
    \Gamma_0(D^*x - D^*x') 
      & \text{using $\mu_{0,D^* x}$-stability of $\J_0$}\\
    &\leq
    \mu_{0,D^* x}
    \normOP{\proj_\T D}{\Gamma_0}{\Gamma}
    \normOP{D^*}{\Gamma}{\Gamma_0}
    \Gamma(x - x') .
  \end{align*}

\item {\bf $\tau_{x}$-stability:} since $f_{0,D^*x} \in \partial J_0(D^* x)$ and $f_{0,D^*x'} \in \partial J_0(D^* x')$, one has
  \begin{equation*}
    f_{0,D^*x} - f_{0,D^*x'}
    =
    \proj_{\Sz} (f_{0,D^*x} - f_{0,D^*x'}) + e_{0,D^* x} - e_{0,D^* x'} . 
  \end{equation*}
  Thus, subadditivity yields
  \begin{align*}
    &\antigx(\proj_\S(f_x - f_{x'}))
    =
    \antigx(\proj_\S D( f_{0,D^*x} - f_{0,D^*x'})) \\
    &\qquad \leq
    \antigx(\proj_\S D \proj_{\Sz}( f_{0,D^*x} - f_{0,D^*x'}))
    +
    \antigx(\proj_\S D ( e_{0,D^* x} - e_{0,D^* x'})) .
  \end{align*}
  Using Proposition~\ref{prop:analysis-dec-form}(iii) and $\tau_{0,D^* x}$-stability of $\J_0$, we get the following bound on the first term
  \begin{align*}
    & \antigx(\proj_\S D \proj_{\Sz}( f_{0,D^*x} - f_{0,D^*x'})) \\
    &=\inf_{z \in \ker(D) \cap \Sz} \antigxa(D_{\Sz}^+ \proj_\S D \proj_{\Sz}( f_{0,D^*x} - f_{0,D^*x'}) + z) \\
    &\leq \antigxa(D_{\Sz}^+ \proj_\S D \proj_{\Sz}( f_{0,D^*x} - f_{0,D^*x'})) \\
    &\leq
    \normOP{D_{\Sz}^+ \proj_\S D}{\antigxa}{\antigxa}
    \antigxa(\proj_{\Sz} (f_{0,D^*x} - f_{0,D^*x'})) \\
    &\leq
    \tau_{0,D^*x}
    \normOP{D_{\Sz}^+ \proj_\S D}{\antigxa}{\antigxa}
    \Gamma_0(D^*x - D^*x') \\
    &\leq
    \tau_{0,D^*x}
    \normOP{D_{\Sz}^+ \proj_\S D}{\antigxa}{\antigxa}
    \normOP{D^*}{\Gamma}{\Gamma_0}
    \Gamma(x - x') .
  \end{align*}
  Now, combining Proposition~\ref{prop:analysis-dec-form}(iii) and $\mu_{0,D^*x}$-stability of $\J_0$, we obtain the following bound on the second term
  \begin{align*}
    \antigx(\proj_\S D ( e_{0,D^* x} - e_{0,D^* x'}))
    &\leq \antigxa(D_{\Sz}^+ \proj_\S D ( e_{0,D^*x} - e_{0,D^*x'})) \\
    &\leq
    \normOP{D_{\Sz}^+ \proj_\S D}{\Gamma_0}{\antigxa}
    \Gamma_0(e_{0,D^* x} - e_{0,D^* x'}) \\
    &\leq
    \mu_{0,D^* x}
    \normOP{D_{\Sz}^+ \proj_\S D}{\Gamma_0}{\antigxa}
    \normOP{D^*}{\Gamma}{\Gamma_0}
    \Gamma(x - x') .
  \end{align*}
  Combining these inequalities, we arrive at
  \begin{align*}
    \antigx(\proj_\S(f_x - f_{x'}))
    &\leq
    \Big(
      \tau_{0,D^*x}
      \normOP{D_{\Sz}^+ \proj_\S D}{\antigxa}{\antigxa}\\
    &\qquad +
      \mu_{0,D^* x}
      \normOP{D_{\Sz}^+ \proj_\S D}{\Gamma_0}{\antigxa}
    \Big)
    \normOP{D^*}{\Gamma}{\Gamma_0}
    \Gamma(x - x') ,
  \end{align*}
  whence we get $\tau_x$-stability.

\item {\bf $\xi_{x}$-stability:} from Proposition~\ref{prop:analysis-dec-form}(iii), we can write for any $\eta \in \S$
\begin{align*}
\antigxp(\eta)	&= \inf_{z \in \ker(D) \cap \Sz} \antigxap(D_{\Sz}^+ \eta + z) \\
			&\leq \antigxp(D_{\Sz}^+ \eta + \bar{z}) 
\end{align*}
for any $\bar{z} \in \ker(D) \cap \Sz$. 

Owing to $\xi_{0,D^*x}$-stability of $\J_0$, and since $D_{\S_0}^+ \eta \in \Sz$, we have for any feasible $\bar{z} \in \ker(D) \cap \Sz$
\[
\antigxap(D_{\Sz}^+ \eta + \bar{z}) \leq  \pa{1+\xi_{0,D^* x}\Gamma_0(D^*x-D^*x')}\antigxa(D_{\Sz}^+ \eta + \bar{z}) ~.
\]
Taking in particular 
\[
\bar{z} \in \uArgmin{z \in \ker(D) \cap \Sz} \antigxa(D_{\Sz}^+ \eta + z)
\]
we get the bound
\begin{align*}
\antigxp(\eta) 	&\leq \pa{1+\xi_{0,D^* x}\Gamma_0(D^*x-D^*x')} \inf_{z \in \ker(D) \cap \Sz}\antigxa(D_{\Sz}^+ \eta + z) \\
			&= \pa{1+\xi_{0,D^* x}\Gamma_0(D^*x-D^*x')} \antigxp(\eta) \\
			&= \pa{1+\xi_{0,D^* x}\normOP{D^*}{\Gamma}{\Gamma_0}\Gamma(x-x')} \antigxp(\eta) ~,
\end{align*}
where we used again Proposition~\ref{prop:analysis-dec-form}(iii) in the first equality.
%   by definition,
%   \begin{equation*}
%     \antig_{\f_{x'}}(\eta_{\S_x}) =
%     \inf_{z \in \Ker(D_{\Sz})}
%     \antig_{0,f_{0,D^* x'}}( D_{\Sz}^{+} \eta_{\S_x} + z ) .
%   \end{equation*}
%   Using the $\xi$-stability of $u \mapsto \antig_{0,f_{0,u}}$ at $D^* x$, one has
%   \begin{equation*}
%     \antig_{\f_{x'}}(\eta_{\S_x}) \leq
%     \inf_{z \in \Ker(D_{\Sz})}
%     \antig_{0,f_{0,D^* x}}( D_{\Sz}^{+} \eta_{\S_x} + z )
%     + \xi_{0,D^*x}
%       \antig_{0,f_{0,u}}(D_{\Sz}^{+} \eta_{\S_x} + z)
%       \Gamma_0(D^*(x - x')) .
%   \end{equation*}
%   Splitting the infimum in two, one get
%   \begin{equation*}
%     \antig_{\f_{x'}}(\eta_{\S_x}) \leq
%     \antig_{\f_{x}}(\eta_{\S_x})
%     + \xi_{0,D^*x} \Gamma_0(D^*(x - x'))
%       \inf_{z \in \Ker(D_{\Sz})} \antig_{0,f_{0,u}}(D_{\Sz}^{+} \eta_{\S_x} + z)
%   \end{equation*}
%   Since
%   \begin{equation*}
%     \inf_{z \in \Ker(D_{\Sz})} \antig_{0,f_{0,u}}(D_{\Sz}^{+} \eta_{\S_x} + z) \leq
%     \normOP{D_{\Sz}^{+}}{\antig_{0,f_{0,u}}}{\antig_{f_x}} \antig_{f_x}(\eta_{\S_x}) ,
%   \end{equation*}
%   we get that
%   \begin{equation*}
%     \frac{\antig_{\f_{x'}}(\eta_{\S_x}) - \antig_{\f_{x}}(\eta_{\S_x})}{\antig_{f_x}(\eta_{\S_x})}
%     \leq
%     \xi_{0,D^*x}
%     \normOP{D_{\Sz}^{+}}{\antig_{0,f_{0,u}}}{\antig_{f_x}}
%     \normOP{D^*}{\Gamma}{\Gamma_0}
%     \Gamma(x - x') .
%   \end{equation*}
\end{itemize}
\end{proof}

%%% Local Variables: 
%%% mode: latex
%%% TeX-master: "../../robust-convex-regularization"
%%% End: 

% !TEX root = ../../IMAIAI-PartlySmoothLinear.tex

\section{Proofs of Section~\ref{sec:small}}

\begin{proof}[Proof of Theorem~\ref{thm:identifiability}]
  This is a straightforward consequence of Theorem~\ref{thm:focu}(ii) by constructing an appropriate dual certificate from $\IC(x_0)$.
  Denote $\e = \e_{x_0}$, $\f = \f_{x_0}$ and $\S=\T^\perp$. Taking the dual vector $\alpha =  \PhiT^{+,*} \e$, we have on the one hand
  \begin{equation*}
    \PhiT^* \PhiT^{+,*} \e = \e 
  \end{equation*}
  since $\e \in \Im(\PhiT^*)$.
  
  On the other hand,
  \begin{equation*}
    \antigxz(\PhiS^* \PhiT^{+,*} \e - \proj_{S} \f ) = \IC(x_0) < 1 .
  \end{equation*}
\end{proof}

% e vector for the solution on T
\newcommand{\te}{\tilde \e}
\newcommand{\he}{\hat \e}
\newcommand{\hf}{\hat \f}
% quandite solution
\newcommand{\hx}{\hat x}
% contraction of \mu and \tau
\newcommand{\mutau}{ \bar\mu }
% IC(x0)
\newcommand{\myIC}{\IC(x_0)}
\newcommand{\choice}[1]{%
	\left\{ \begin{array}{l} #1	\end{array} \right.
}

\begin{proof}[Proof of Theorem~\ref{thm:local-stab}]
To lighten the notation, we let $\epsilon=\norm{w}$, $\nu=\nu_{x_0}, \mu=\mu_{x_0}, \tau=\tau_{x_0}, \xi=\xi_{x_0}, \f = \f_{x_0}$.\\

The strategy is to construct a vector which, by \eqref{eq:injT}, is the unique solution to
\begin{equation}\label{eq:restricted}\tag{$\Pp_\lambda^\T(y)$}
  \umin{x \in \T}
  \dfrac{1}{2} \norm{y - \Phi x}^2
  + \lambda \J(x) ~,
\end{equation}
and then to show that it is actually the unique solution to~\eqref{eq:reg} under the assumptions of Theorem~\ref{thm:local-stab}. 
%when $\T = \T_{x_0}$ and $y = \Phi x_0 + w$, where

The following lemma gives a convenient implicit equation satisfied by the unique solution to~\eqref{eq:restricted}.
\begin{lem}\label{lem:implicit-equation}
  Let $x_0 \in \RR^N$ and denote $\T = \T_{x_0}$.
  Assume that~\eqref{eq:injT} holds.
  Then~\eqref{eq:restricted} has exactly one minimizer $\hx$, and the latter satisfies 
  \begin{equation}\label{eq:implicit}
    \hx = x_0 + \PhiT^+ w - \lambda (\PhiT^* \PhiT)^{-1} \te
    \qwhereq
    \te \in \proj_T(\partial J(\hx)) .
  \end{equation}
\end{lem}
\begin{proof}
  Assumption~\eqref{eq:injT} implies that the objective in~\eqref{eq:restricted} is strongly convex on the feasible set $T$, whence uniqueness follows immediately. By a trivial change of variable,~\eqref{eq:restricted} be also rewritten in the unconstrained form
  \begin{equation*}
    \hx = 
    \uargmin{x \in \RR^N}
    \dfrac{1}{2} \norm{y - \PhiT x}^2
    + \lambda \J(\proj_T x) ~.
  \end{equation*}
  Thus, using Proposition~\ref{prop:foc}(i), $\hx$ has to satisfy
  \begin{equation*}
    \PhiT^* (y - \PhiT \hx) + \lambda \te = 0,
  \end{equation*}
  for any $\te \in \proj_T(\partial J(\hx))$. Owing to the invertibility of $\Phi$ on $T$, i.e.~\eqref{eq:injT}, we obtain~\eqref{eq:implicit}.
\end{proof}

We are now in position to prove Theorem~\ref{thm:local-stab}. %Let $\hx$ be the solution of~\eqref{eq:restricted}.
This is be achieved in three steps:
\begin{enumerate}[{\bf {Step}~1:}] 
\item We show that in fact $\T_{\hx} = \T$. 
\item Then, we prove that $\hx$ is the unique solution of~\eqref{eq:reg} using Theorem~\ref{thm:focu}.
\item We finally exhibit an appropriate regime on $\lambda$ and $\epsilon$ for the above two statements to hold.
\end{enumerate}

%%%%
\subsubsection{Step 1: Subspace equality}

By construction of $\hx$ in~\eqref{eq:restricted}, it is clear that $\hx \in \T$. The key argument now is to use that $\J$ is partly smooth relative to a linear manifold at $x_0$, and to show that
\begin{equation}
\label{eq:prgxhat}
  \certCo{x_0 - \hx} \leq \nu, 
\end{equation}
which in turn will imply subspace equality, i.e. $\T_{\hx} = \T$ (see Definition~\ref{def:PRG}).

We have from~\eqref{eq:implicit} and subadditivity that
\begin{align}
        \certCo{x_0 - \hx}
  &\leq \certCo{-\PhiT^+ w}
      + \lambda \certCo{(\PhiT^* \PhiT)^{-1} \te} \nonumber\\
  &\leq \normOP{(\PhiT^* \PhiT)^{-1}}{\certCof}{\certCof} \left\{ \certCo{-\PhiT^* w} 
      + \lambda \certCo{\te} \right\} \nonumber \\
  &\leq \normOP{(\PhiT^* \PhiT)^{-1}}{\certCof}{\certCof}  
        \left\{ \normOP{\PhiT^{*}}{\ldeux}{\certCof} \epsilon + \alpha_0 \lambda
        \right\}. \label{bound:1}
\end{align}
where $\alpha_0=\certCo{\te}$.
%\marginpar{\todo{ Must be $\certCo{-(\PhiT^* \PhiT)^{-1} \te} $ and $\certCo{- \te}$.}}
%As a result, we can show that the PRG property is implied by a stronger condition of the following form:
Consequently, to show that~\eqref{eq:prgxhat} is verified, it is sufficient to prove that
\begin{equation}\label{eq:condition-1}\tag{$C_1$}
  A \epsilon + B \lambda\leq \nu,
\end{equation}
where we set the positive constants
\begin{align*}
A &= \normOP{(\PhiT^* \PhiT)^{-1}}{\certCof}{\certCof}  \normOP{\PhiT^{*}}{\ldeux}{\certCof},\\
B &=  \alpha_0 \normOP{(\PhiT^* \PhiT)^{-1}}{\certCof}{\certCof}. 
\end{align*}

Suppose for now that~\eqref{eq:condition-1} holds and consequently, $\T_{\hx} = \T$. Then decomposability of $\J$ on $\T$ (Theorem~\ref{thm:decomp}) implies that
\eq{
	\he = \proj_{\T_{\hx}}(\partial J(\hx)) = \proj_{\T}(\partial J(\hx)) = \te,
}
where we have denote $\he = \e_{\hx}$.
Thus~\eqref{eq:implicit} yields the following implicit equation
\begin{equation}\label{eq:implicit1}
  \hx = x_0 + \PhiT^+ w - \lambda(\PhiT^* \PhiT)^{-1} \he.
\end{equation}
%\marginpar{\todo{Question related to (11)(13): is $\PhiT^+ = (\PhiT^* \PhiT)^{-1}\PhiT^*$ ?}}

%%%%
\subsubsection{Step 2: $\hx$ is the unique solution of~\eqref{eq:reg}}

Recall that under condition~\eqref{eq:condition-1}, $\J$ is decomposable at $\hx$ and $x_0$ with the same model subspace $\T$. Moreover,~\eqref{eq:implicit1} is nothing but condition~\eqref{eq:cs-min-T} in Theorem~\ref{thm:focu} satisfied by $\hx$.
To deduce that $\hx$ is the unique solution of~\eqref{eq:reg}, it remains to show that~\eqref{eq:cs-min-S} holds i.e., 
\begin{equation}\label{eq-first-order-cond-proof}
  	\antighf( \lambda^{-1} \PhiS^*(y - \Phi \hx) - \hf_{\S} ) < 1.
\end{equation}
where we use the shorthand notations $\hf = f_{\hat x}$ and $\hf _S= \proj_\S \hf$.
%\todo{what do u mean? $\hat f= f_{\hat x}$? is it better $f_S$ or $\proj_S f$}.

Under condition~\eqref{eq:condition-1}, the $\xi$-stability property~\eqref{eq:lip-xi} of $\J$ at $x_0$ yields 
\begin{align}
  \antighf(\lambda^{-1} \PhiS^*(y - \Phi \hx) - \hf_{\S})
  &\leq \big(1+\xi\certCof(x_0 - \hx) \big) 
   \antigxz(\lambda^{-1} \PhiS^*(y - \Phi \hx) - \hf_{\S}) \label{eq:bound-J-x-x0} .
\end{align}
Furthermore, from~\eqref{eq:implicit1}, we can derive
\begin{equation}\label{eq:bb1}
  \lambda^{-1} \PhiS^*(y - \Phi \hx) - \hf_{\S}
  = \PhiS^* \PhiT^{+,*} \he + \lambda^{-1} \PhiS^* Q_T w - \hf_{\S} ,
\end{equation}
where $Q_\T = \Id - \PhiT \PhiT^{+} = \proj_{ \Ker(\PhiT^* ) }$. %Note that here we use the identity $\PhiT^+ = (\PhiT^* \PhiT)^{-1}\PhiT^*$.
%\marginpar{\todo{now i guess is this $Q_\T = \Id - \PhiT  \PhiT^{+}$ }}
Inserting\eqref{eq:bb1} in~\eqref{eq:bound-J-x-x0}, we obtain
\begin{equation*}
  \antighf(\lambda^{-1} \PhiS^*(y - \Phi \hx) - \hf_{\S})
  \leq
  \big(1+\xi\certCof(x_0 - \hx) \big)  \antigxz(\PhiS^* \PhiT^{+,*} \he + \lambda^{-1} \PhiS^* Q_T w - \hf_{\S}) .
\end{equation*}
Moreover, subadditivity yields
\begin{align}
   \antigxz(\PhiS^* \PhiT^{+,*} \he + \lambda^{-1} \PhiS^* Q_T w - \hf_{\S}) \nonumber 
  &\leq
    \antigxz(\PhiS^* \PhiT^{+,*} \e - \f_{\S}) 
  + \antigxz(\PhiS^* \PhiT^{+,*} (\he - \e) ) \nonumber \\
  & + \antigxz( \proj_{\S} (\f - \hf) )
  + \antigxz(\lambda^{-1} \PhiS^* Q_T w) \label{eq:to-be-bounded}.
\end{align}
We now bound each term of~\eqref{eq:to-be-bounded}.
In the first term, one recognizes
%\marginpar{\todo{$\proj_{\S} \f$ or $f_S$? notation abuse}}
\begin{equation}\label{eq:bound-t1}
  \antigxz(\PhiS^* \PhiT^{+,*} \e -  \f_\S)
  \leq
  \myIC .
\end{equation}
Appealing to the $\mu$-stability property, we get
\begin{align}
  \antigxz(\PhiS^* \PhiT^{+,*} (\he - \e) )
  &\leq
  \normOP{-\PhiS^* \PhiT^{+,*}}{\certCof}{\antigxz}
  \certCo{\e - \he} \nonumber \\
  &\leq
  \mu \normOP{-\PhiS^* \PhiT^{+,*}}{\certCof}{\antigxz}
 \certCo{x_0 - \hx} 
  . \label{eq:bound-t2}
\end{align}
From $\tau$-stability, we have
%\marginpar{\todo{$\proj_{\S} \f$ or $f_S$? notation abuse}}
\begin{align}
  \antigxz(\f_\S - \hf_\S )
  &\leq
  \tau
  \certCo{x_0 - \hx} 
. \label{eq:bound-t3}
\end{align}
Finally, we use a simple operator bound to get
\begin{equation}\label{eq:bound-t4}
  \antigxz(\lambda^{-1} \PhiS^* Q_T w)
  \leq
  \frac{1}{\lambda}
  \normOP{\PhiS^* Q_T}{\ldeux}{\antigxz}
  \epsilon .
\end{equation}
Following the same steps as for the bound~\eqref{bound:1}, except using $\te = \he$ here, gives
\begin{align}
        \certCof\big(x_0 - \hx) \big)
\leq \normOP{(\PhiT^* \PhiT)^{-1}}{\certCof}{\certCof}  
        \left\{ \normOP{\PhiT^{*}}{\ldeux}{\certCof} \epsilon + \lambda \certCo{\he}
        \right\}.\label{bound:111}
\end{align}
Plugging inequalities~\eqref{eq:bound-t1}-\eqref{bound:111} into~\eqref{eq:bound-J-x-x0} we get the upper-bound
\begin{align*}
  &\antighf(\PhiS^* \PhiT^{+,*} \he + \lambda^{-1} \PhiS^* Q_T w - \hf_{\S}) \nonumber \\
  &\leq \,
\pa{1+\xi\certCof(x_0 - \hx)}
  \Big(
    \myIC + \certCof\pa{x_0 - \hx}  
    \big(\mu \normOP{-\PhiS^* \PhiT^{+,*}}{\certCof}{\antigxz} + \tau\big) \nonumber \\
  &\quad  +
    \frac{1}{\lambda}   \normOP{\PhiS^* Q_T}{\ldeux}{\antigxz} \epsilon
  \Big) \nonumber \\
    &\leq \,
\pa{1+\xi(c_1 \epsilon + \lambda c_2)}
  \Big(
    \myIC + (c_1 \epsilon + \lambda c_2)
    \mutau +
    \frac{1}{\lambda}   c_4 \epsilon
  \Big) <1, \label{eq:polynom1}
\end{align*}
where we have introduced
\eq{
	\mutau = \mu c_3 + \tau
	\qandq
	\alpha_1 = \certCo{\he} = \certCo{\te} = \alpha_0
} 
and
\begin{equation*}
\begin{array}{rclrcl}
  c_1 
  &=&
  A,
  & c_2
  & = &
  \alpha_1
  \normOP{(\PhiT^*\PhiT)^{-1}}{\certCof}{\certCof}, \\
  c_3
  &=&
  \normOP{-\PhiS^*\PhiT^{+,*}}{\certCof}{\antigxz},
  & c_4
  &=&
\normOP{\PhiS^* Q_T}{\ldeux}{\antigxz}.  
\end{array}
\end{equation*}
If is then sufficient that
\begin{align}
\pa{1+\xi(c_1 \epsilon + \lambda c_2)}
  \Big(
    \myIC + (c_1 \epsilon + \lambda c_2)
    \mutau +
    \frac{1}{\lambda}   c_4 \epsilon
  \Big) <1, \label{eq:polynom1}
\end{align}
for~\eqref{eq:cs-min-S} in Theorem~\ref{thm:focu} to be in force. 

In particular, if 
\eq{
	C \epsilon \leq \lambda
}
holds for some constant $C>0$ to be fixed later, then inequality~\eqref{eq:polynom1} is true if
\eql{\label{eq:polynom111}
	P(\lambda) = a \lambda^2 + b \lambda + c > 0
	\qwhereq
	\choice{
		a =  -\xi \mutau \pa{c_1/C+c_2}^2  \\
		b = -(c_1/C+c_2) \pa{ \xi \myIC+ \xi c_4/C+ \mutau }  \\
		c = 1-\myIC-c_4/C
	}~.
}
Let us set the value of $C$ to
\eq{
	C = \frac{2 c_4}{1-\myIC},\label{eq:ci}
}
which, for $0\leq \myIC<1$, it ensures that $c=\tfrac{1-\myIC}{2}$ is bounded and positive, and thus, the polynomial $P$ has a negative and a positive root $\lambda_{\max}$ equal to 
\begin{align*}
	\lambda_{\max} &= \frac{b}{2a}\phi\pa{- 4\frac{ac}{b^2} }, 
	\quad
	\choice{
		a =  -\xi \mutau ( (1-\myIC)c_1/(2c_4)+c_2)^2   \\
		b = - ( (1-\myIC)c_1/(2c_4)+c_2)  \pa{ \mutau + (1+\myIC)\xi/2  }  \\
		c = (1-\myIC)/2
	} \\
	%%%%%%%%%
		&= \frac{ \mutau + (1+\myIC ) \xi/2 }
	{ \xi  \mutau  ( (1-\myIC )c_1/c_4 + 2c_2  ) }
	\phi\pa{
		\frac{ 2\xi (1-\myIC) \mutau }{ ( \mutau + (1+\myIC) \xi/2  )^2 } 
	}\\
	%%%%%%%%%
	& \geq \frac{1-\myIC}{\xi} H(\mutau/\xi),
\end{align*}
where
\[ \phi(\beta) = \sqrt{1+\beta}-1, \qandq 	
	H(\beta) = \frac{ \beta + 1/2 }{ \beta (c_1/c_4 + 2c_2)} 
	\phi\pa{ \frac{2\beta}{ (\beta + 1 )^2 } }. 
	\]
To get the above lower-bound on $\lambda_{\max}$, we used that $\phi$ is increasing (in fact strictly) and concave on $\RR_+$ with $\phi(1)=0$, and that $\myIC \in [0,1[$.
%\marginpar{\todo{shouldn't be $1+1/2\rho$ on numerator of H?}}
%\eq{
%	\foralls \beta \geq 0, \quad \phi \pa{ (1-\myIC) \beta} \geq (1-\myIC) \phi(\beta).
%}
Consequently, we can conclude that the bounds
\eql{\label{eq:labound1}\tag{$C_2$}
	\frac{2 c_4}{1-\myIC} \epsilon \leq \lambda \leq \frac{1-\myIC}{\xi} H(\mutau/\xi)	
}
imply condition~\eqref{eq:polynom1}, which in turn yields~\eqref{eq-first-order-cond-proof}.

%%%%
\subsubsection{Step 3:~\eqref{eq:condition-1} and~\eqref{eq:labound1} are in agreement}

It remains now that show the compatibility of~\eqref{eq:condition-1} and~\eqref{eq:labound1}, i.e. to provide appropriate regimes of $\lambda$ and $\epsilon$ such that both conditions hold simultaneously.
%In this part we verify the compatibility of conditions~\eqref{eq:condition-1} and~\eqref{eq:labound1}. 
We first observe that~\eqref{eq:condition-1} and the left-hand-side of~\eqref{eq:labound1} both hold for $\lambda$ fulfilling
\[ 
\lambda \leq C_0 \nu \qwhereq C_0 = \pa{\frac{A}{2c_4} + B}^{-1} \leq \pa{ \frac{1-\myIC}{2c_4} A + B }^{-1} ~.
\]
This updates~\eqref{eq:labound1} to the following ultimate range on $\lambda$
%[detailler un peu plus] The following constraints
%\todo{2nd condition, shouldn't be $\epsilon \leq C^{-1} \min\pa{C_0 \nu, \frac{1-\myIC}{\xi} H(\mutau/\xi)}$? )}
\begin{align*}
	\frac{2c_4}{1-\myIC} \epsilon &\leq \lambda \leq 
		\min\pa{ C_0 \nu, \frac{1-\myIC}{\xi} H(\mutau/\xi) } ~.
\end{align*}
Now in order to have an admissible non-empty range for $\lambda$, the noise level $\epsilon$ must be upper-bounded as
\eq{	\epsilon \leq \frac{1-\myIC}{2c_4} \min\pa{ C_0 \nu, \frac{1-\myIC}{\xi} H(\mutau/\xi) }. 
	%\qwhereq
	%C_0 = \pa{\frac{A}{2c_4} + B}^{-1} \leq \pa{ \frac{1-\myIC}{2c_4} A + B }^{-1}
}
%implies conditions~\eqref{eq:condition-1} and~\eqref{eq:labound1}, which is the desired result. 
Finally, the constants provided in the statement of the theorem (and subsequent discussion) are as follows
\[
A_T = 2c_4, \,\, B_T = C_0, \,\, D_T = c_3, \,\text{and}\,\, E_T = c_1/c_4+2c_2  ~,
\]
which completes the proof.
%\marginpar{\todo{note that A,B,D,E are not totally independent of $x_0$ because $\alpha_0$ is hidden inside!}}

%[ecrire les valeurs des constantes $A_\T,B_\T,C_\T,D_\T,E_\T$ correspondant a l'enonce] 

\end{proof}

%%% Local Variables: 
%%% mode: latex
%%% TeX-master: "../../robust-convex-regularization"
%%% End: 

% !TEX root = ../../IMAIAI-PartlySmoothLinear.tex

\section{Proofs of Section~\ref{sec:examples}}

\begin{proof}[Proof of Proposition~\ref{prop:ex-lun-swdg}]
  The subdifferential of $\normu{\cdot}$ reads
  \begin{equation*}
    \partial \normu{\cdot}(x) = \enscond{\eta \in \RR^N}{\eta_{(I)} = \sign(x_{(I)}) \qandq \normi{\eta_{(I^c)}} \leq 1} .
  \end{equation*}
  The expressions of $\S_x$, $\T_x$, $\e_x$ and $f_x$ follow immediately.
  Since $e_x \in \ri \partial \normu{\cdot}(x)$ and $\normu{\cdot}$ is separable, it follows from Definition~\ref{defn:strong-gauge} that the $\lun$-norm is a strong gauge. Therefore $\Js_{f_x} = \Js = \normi{\cdot}$, and Proposition~\ref{prop:strong-dec} specializes to the stated subdifferential.\\
  
  Turning to partial smothness, let $x' \in \T$, i.e. $I(x') \subseteq I(x)$, and assume that 
  \[
  \normi{x-x'} \leq \nu_x=(1-\delta)\umin{i \in I} \abs{x_i} \, , \delta \in ]0,1] ~.
  \] 
  This implies that $\forall i \in I(x)$, $|x'_i| > \nu_x - \normi{x - x'} \geq 0$, which in turn yields $I(x')=I(x)$, and thus $\T_{x'}=\T_{x}$. Since the sign is also locally constant on the restriction to $\T$ of the $\linf$-ball centered at $x$ of radius $\nu_x$, one can choose $\mu_x = 0$. Finally $\tau_x=\xi_x=0$ because $f_x=e_x$.
\end{proof}

\begin{proof}[Proof of Proposition~\ref{prop:ex-linf-swdg}]
The proof of the first part was given Section~\ref{sec:convex-geom} and Section~\ref{sec:decomposable-prop} where the $\linf$-norm example was considered.\\

It remains to show partial smothness.
  Let $x' \in \T$, and assume that 
  \[
  \normu{x - x'} \leq \nu_x=(1-\delta)\big(\norm{x}_\infty - \umax{j \notin I} \abs{x_j}\big), \delta \in ]0,1] ~.
  \]
  This means that $x'$ lies in the relative interior of the $\lun$-ball (relatively to $\T$) centered at $x$ of radius $\norm{x}_\infty - \umax{j \notin I} \abs{x_j}$. Within this ball, the support and the sign pattern restricted to the support are locally constant, i.e. $I(x)=I(x')$ and $\sign(x_{(I(x))})=\sign(x'_{(I(x'))})$. Thus $\T_{x'} = \T_x = T$ and $\e_{x'} = \e_x$, and from the latter we deduce that $\mu_x=0$. As $f_x=e_x$ we also conclude that $\tau_x=\xi_x=0$, which completes the proof.
\end{proof}

\begin{proof}[Proof of Proposition~\ref{prop:ex-l1l2-swdg}]
Again, the proof of the first part was given Section~\ref{sec:convex-geom} and Section~\ref{sec:decomposable-special} where the $\lun-\ldeux$-norm example was handled.\\

Let $x'\in \T$, i.e. $I(x') \subseteq I(x)$, and $\nu_x =(1-\delta)\umin{b\in I}\norm{x_b}$, $\delta \in ]0,1]$. First, observe that the condition 
\[
\norm{x - x'}_{\infty,2} = \max_{b \in \Bb} \norm{x_b - x'_b} \leq \nu_x
\] 
ensures that for all $b \in I$
\begin{align*}
\norm{x'_b} \geq \norm{x_b} -\norm{x_b-x'_b} > \nu_x  - \norm{x - x'}_{\infty,2} \geq 0,
\end{align*}
and thus $I(x')=I(x)$, i.e. $\T_{x'}=\T_{x}$. Moreover, since the gauge is strong, one has $\tau_x = \xi_x = 0$. To establish the $\mu_x$-stability we use the following lemma.

\begin{lem}\label{lem1}
Given any pair of non-zero vectors $u$ and $v$ where, $\norm{u-v} \leq \rho \norm{u}$, for $0< \rho < 1$,  we have
\eq{
\left\| \frac{u}{\norm{u}}-\frac{v}{\norm{v}} \right\| \leq C_\rho \frac{\norm{u-v}}{\norm{u}},
}
where $C_\rho = \frac{\sqrt{2}}{\rho} \sqrt{1-\sqrt{1-\rho^2}}\in ]1,\sqrt2[$.
\end{lem}

\begin{proof}
  Let $d=v-u$ and $\beta=\frac{\dotp{u}{d}}{\norm{u}\norm{d}}\in [-1,1]$. We then have the following identities
  \begin{align}
    \left\| \frac{u}{\norm{u}}-\frac{v}{\norm{v}} \right\|^2 = 2-2\frac{\dotp{u}{v}}{\norm{u}\norm{v}}%\nonumber \\
    =2-2 \frac{\norm{u}^2+\norm{u}\norm{d}\beta}{\norm{u}\sqrt{\norm{u}^2+\norm{d}^2+2\norm{u}\norm{d}\beta}}, \label{eq:18}
  \end{align}
  for non-zero vectors $u$ and $d$, the unique maximizer of \eqref{eq:18} is $\beta^{\star}=-\norm{d}/\norm{u}$.
  Note that the assumption $\norm{d}/\norm{u}\leq \rho <1$ assures $\beta^{\star}$ to comply with the admissible range of $\beta$ and further, the argument of the square root will be always positive.
  Now, inserting $\beta^\star$ in \eqref{eq:18}, using concavity of $\sqrt{\cdot}$ on $\RR_+$, and that $\norm{d}/\norm{u}\leq \rho $, we can deduce the following bound 
  \begin{align*}
    \left\| \frac{u}{\norm{u}}-\frac{v}{\norm{v}} \right\|^2 \leq 2-2\sqrt{1-\frac{\norm{d}^2}{\norm{u}^2}}
    &= 2 - 2 \sqrt{\pa{1-\frac{\norm{d}^2}{\rho^2\norm{u}^2}}+\frac{\norm{d}^2}{\rho^2\norm{u}^2}\pa{1-\rho^2}} \\
    &\leq 2 - 2 \pa{\pa{1-\frac{\norm{d}^2}{\rho^2\norm{u}^2}} + \frac{\norm{d}^2}{\rho^2\norm{u}^2}\sqrt{1-\rho^2}} \\
    &= 2-2\left( 1- \frac{1-\sqrt{1-\rho^2}}{\rho^2}  \frac{\norm{d}^2}{\norm{u}^2} \right) \label{eq:19} \\
    &= 2\frac{1-\sqrt{1-\rho^2}}{\rho^2}  \frac{\norm{d}^2}{\norm{u}^2} ~.
  \end{align*}
  %Note that the liner approximation bound \eqref{eq:19} holds for $\norm{d}/\norm{u}\leq \rho $. 
\end{proof}

By definition of $\nu_x$, we have $(1-\delta)\norm{x_b} > \nu_x$, for $\delta \in ]0,1]$, $\forall b\in I$, and thus $\norm{x_b-x'_b} \leq \nu_x \leq (1-\delta)\norm{x_b}$. Lemma~\ref{lem1} then applies, and it follows that, $\forall b\in I$
\[
\norm{\Nn(x_b)-\Nn(x'_b)}  \leq C_\rho \frac{ \norm{x'_b-x_b} }{\norm{x_b}} \leq C_\rho \frac{ \norm{x'_b-x_b} }{\nu_x},
\]
and therefore we get
\begin{align*}
\norm{\Nn(x)-\Nn(x')}_{\infty,2}  
\leq \frac{C_\rho}{\nu_x} \norm{x'-x}_{\infty,2},
\end{align*}
which implies $\mu_x$-stability for $\mu_x=C_\rho/\nu_x$. 

\end{proof}

\begin{proof}[Proof of Proposition~\ref{prop:ex-polyh-swdg}]
  In general, the subdifferential of $\J_0$ reads
  \begin{equation*}
    \partial J_0(u) =
    \enscond
    {\sum_{i \in I} \rho_i s_i a^i}
    {\rho \in \Sigma_I, s_i \in
      \begin{cases}
        \ens{1} & \text{if } u_i > 0 \\
        [0,1]   & \text{if } u_i = 0 \\
        \ens{0} & \text{if } u_i < 0
      \end{cases}
    } ,
  \end{equation*}
  where $\Sigma_I$ is the canonical simplex in $\RR^{\abs{I}}$, and $I=\enscond{i \in \ens{1,\cdots,N_H}}{(x_i)_+=\J_0(x)}$.
  \begin{itemize}
  \item If $u_i \leq 0$, $\forall i \in \ens{1,\cdots,N_H}$, the above expression becomes
  \begin{equation*}
    \partial \J_0(u) = \enscond{\sum_{i \in I_0} \rho_i s_i a^i}{\rho \in \Sigma_{I_0}, s_i \in [0,1]} ~,
  \end{equation*}
  where $I_0 = \enscond{i \in \ens{1,\cdots,N_H}}{ u_i = \J_0(u) = 0 }$. Equivalently, $\partial \J_0(u)$ is the intersection of the unit $\ell^1$ ball and the positive orthant on $\RR^{|I_0|}$. The expressions of $\S_u$, $\T_u$ and $\e_u$ then follow immediately. $\partial \J_0(u)$ then contains $\e_u=0$, but not in its relative interior. Choosing any $\f_u$ as advocated, we have $\f_u \in \ri \partial \J_0(u)$. To get the subdifferential gauge, we use some calculus rules on gauges and apply Lemma~\ref{prop:anti-coer} to get
  \[
  J_{\f_u}^{u,\circ}(\eta_{(I_0)}) = \inf_{\tau \geq 0, ~~ \tau (\f_u)_i \geq -\eta_i ~ \forall i \in I_0} \max( \normu{\tau \f_u + \eta}, \tau) ~,
  \]
  where the extra-constraints on $\tau$ come from the fact that $\partial \J_0(u)$ is in the positive orthant, and the $\lun$ norm is the gauge of the unit $\lun$-ball. We then have
  \begin{align*}
  J_{\f_u}^{u,\circ}(\eta_{(I_0)}) 	&= \inf_{\tau \geq 0, ~ \mu\tau \geq \max_{i \in I_0}-\eta_i} \max( \tau \sum_{i \in I_0}\pa{\mu a^i + \eta_i}, \tau) \\
  				&= \inf_{\tau \geq \max_{i \in I_0}\pa{-\eta_i}_+/\mu} \max( \tau \mu |I_0|+\sum_{i \in I_0}\eta_i, \tau) ~.
  \end{align*}
  \item Assume now that $u_i > 0$ for at least one $i \in \ens{1,\cdots,N_H}$. In such a case, $\J_0(u)=\normi{u}$, and the subdifferential becomes
  \begin{equation*}
    \partial \J_0(u) = \Sigma_{I_+} ~,
  \end{equation*}
  where $I_+\enscond{i \in \ens{1,\cdots,N_H}}{ u_i = \J_0(u) \qandq u_i > 0  }$. The forms of $\S_u$, $\T_u$, $\e_u$, $\f_u$ and the subdifferential gauge can then be retrieved from those of the $\linf$-norm with $s_{(I_+)}=1$ and $s_{(I_+^c)}=0$.
  \end{itemize}
  
  For partial smothness, the parameters are derived following the same lines as for the $\linf$-norm.
  Let $u' \in \T$, and assume that
  \begin{equation*}
    \normu{u - u'} \leq \nu_u
    = (1-\delta) \left( \umax{i \in I_+} u_i - \umax{j \not\in I_+, u_j > 0} u_j \right) ,
  \end{equation*}
  for $\delta \in ]0,1]$.
  This means that $x'$ lies in the relative interior of the $\lun$-ball (relatively to $\T$) centered at $x$ of radius 
  \begin{equation*}
    \umax{i \in I_+} u_i - \umax{j \not\in I_+, u_j > 0} u_j
    =
    \norm{u}_\infty - \umax{j \not\in I_+, u_j > 0} \abs{u_j}
  \end{equation*}
  Within this set, one can observe that the set $I_+$ associated to $u$ is constant. Moreover, the sign pattern is also constant leading to the fact that $\T_{u'} = \T_u = T$. 
  Hence, we deduce as in the $\linf$-case that $\mu_u=\tau_u=\xi_u=0$.
\end{proof}

% !TEX root = ../../IMAIAI-PartlySmoothLinear.tex

\section{Proofs of Section~\ref{sec:cs}}

\begin{proof}[Proof of Theorem~\ref{thm:linf-cs}]
  To lighten the notation, we drop the dependence on $x$ of $\T$, $\S$ and $\e$. Without of loss of generality, by symmetry of the norm, we will assume that the entries of $x$ are positive.

  We follow the same program as in the compressed sensing literature, see e.g.~\cite{candes2011simple}. The key ingredient of the proof is the fact that owing to the isotropy of the Gaussian ensemble, $\ceF$ and $\Phi^*_{\S}$ are independent. Thus, for some $\tau > 0$
  \[
  \Pr\pa{\IC(x) \geq 1} \leq \Pr\pa{\IC(x) \geq 1 \Big\vert \norm{\ceF} \leq \tau} + \Pr\pa{\norm{\ceF} \geq \tau} ~.
  \]
  As soon as $Q \geq \dim(\T) = N - \abs{I} + 1$, $\Phi_{\T}$ is full-column rank. Thus
  \[
  \norm{\ceF}^2 = \dotp{\e}{\pa{\Phi_{\T}^*\Phi_{\T}}^{-1}\e} ~.
  \]
  $\pa{\Phi_{\T}^*\Phi_{\T}}^{-1}$ is an inverse Wishart matrix with $Q$ degrees of freedom. To estimate the deviation of this quadratic form, we use classical results on inverse $\chi^2$ random variables with $Q-N+\abs{I}$ degrees of freedom and we get the tail bound
  \[
  \Pr\pa{\norm{\ceF} \geq \sqrt{\frac{1}{Q-N+\abs{I}-t}}\norm{\e}} \leq e^{-\tfrac{t^2}{4(Q-N+\abs{I})}}
  \]
  for $t > 0$. Now, conditionally on $\ceF$, the entries of $\alpha_{\S} = \proj_{\S}\Phi^*\ceF$ are i.i.d. $\Nn(0,\norm{\ceF}^2)$ and so are those of $-\alpha_{\S}$ by trivial symmetry of the centered Gaussian. Thus, using a union bound, we get
  \begin{align*}
    \Pr\pa{\IC(x) \geq 1 \Big\vert \norm{\ceF} \leq \tau} 	
    &\leq \Pr\pa{\umax{i \in I} (-(\alpha_{\S_x})_i)_+ \geq 1/\abs{I} \Big\vert \norm{\ceF} \leq \tau} \\
    &\leq \Pr\pa{\umax{i \in I} ((\alpha_{\S_x})_i)_+ \geq 1/\abs{I} \Big\vert \norm{\ceF} \leq \tau} \\
    &\leq \abs{I}\Pr\pa{(z)_+ \geq 1/(\tau\abs{I})} \\
    &\leq \abs{I}\Pr\pa{z \geq 1/(\tau\abs{I})} \\
    &\leq \abs{I}e^{-\tfrac{1}{2\tau^2\abs{I}^2}} ~.
  \end{align*}
  Observe that $(\alpha_{\S})_i=0$ for all $i \in I^c$. Choosing
  \[
  \tau = \sqrt{\frac{1}{\abs{I}(Q-N+\abs{I}-t)}}
  \]
  where we used that $\norm{\e}=1/\sqrt{I}$, and inserting in the above probability terms, we get
  \begin{align*}
    \Pr\pa{\norm{\ceF} \geq \tau} &\leq e^{-\tfrac{t^2}{4(Q-N+\abs{I})}} ~,\\
    \Pr\pa{\IC(x) \geq 1 \Big\vert \norm{\ceF} \leq \tau} 	&\leq e^{-\pa{\tfrac{Q-N+\abs{I}-t}{2\abs{I}}-\log(\abs{I}/2)}} ~.
  \end{align*}
  Equating the arguments of the exponentials and solving 
  \[
  \frac{t^2}{4 q} + \frac{t}{2\abs{I}} - \pa{\frac{q}{2\abs{I}}-\log\pa{\tfrac{\abs{I}}{2}}} = 0 
  \] 
  for $t$ to get equal probabilities, we get
  \[
  t = \frac{q}{\abs{I}}\pa{\sqrt{1 + 2\abs{I}\pa{1-2\tfrac{2\abs{I}\log\pa{\tfrac{\abs{I}}{2}}}{q}}} - 1} ~,
  \]
  where $q=Q-N+\abs{I} \geq 1$ by the restricted injectivity assumption. Setting
  \[
  \beta = \frac{q}{2\abs{I}\log\pa{\tfrac{\abs{I}}{2}}} ~,
  \]
  we get under the bound on $Q$ that $\beta > 1$, and 
  \[
  t = 2\beta \log\pa{\tfrac{\abs{I}}{2}}\pa{\sqrt{1 + 2\abs{I}\tfrac{\beta-1}{\beta}} - 1} ~.
  \]
  Inserting $t$ in one of the probability terms, and after basic algebraic rearrangements, we get the probability of success with the expression of the function $f(\beta,\abs{I})$.  
\end{proof}

%%% Local Variables:
%%% mode: latex
%%% TeX-master: "../../robust-convex-regularization"
%%% End:

%%% Local Variables: 
%%% mode: latex
%%% TeX-master: "../../robust-convex-regularization"
%%% End: 

  \bibliographystyle{plain}
  \bibliography{IMAIAI-PartlySmoothLinear}
\end{document}